\documentclass[a4paper]{scrartcl}
\KOMAoptions{fontsize=10pt}
\usepackage[utf8]{inputenc}
\usepackage[english]{babel}
\usepackage{xcolor}
\usepackage{graphicx}
\usepackage{hyperref}
\usepackage[T1]{fontenc}
\usepackage{lmodern}
\usepackage{rotating}
\usepackage{tikz}
\usepackage{tikz-cd}
\usepackage[backend=biber,sortlocale=en_GB,url=false,doi=false,isbn=false,giveninits=true,style=alphabetic,maxbibnames=1000,date=year,sorting=nyt]{biblatex}
\bibliography{../../../jabref.bib}
\DeclareFieldFormat*{title}{\mkbibemph{#1\isdot}}
\DeclareFieldFormat*{journaltitle}{#1\isdot}
\renewbibmacro{in:}{}
\usepackage{amsmath}
\usepackage{amsfonts}
\usepackage{amssymb}
\usepackage{amsthm}
\usepackage{mathtools}
\usepackage{stmaryrd}
\usepackage{enumerate}
\usepackage{verbatim}
\usepackage{dsfont}
\usepackage{mathabx}
\usepackage{faktor}
\usepackage{xr}
\DeclareMathOperator\Gal{Gal}

\DeclareMathOperator\Hom{Hom}

\DeclareMathOperator\QB{QB}
\DeclareMathOperator\wt{wt}

\DeclareMathOperator\inv{inv}
\DeclareMathOperator\maxinv{max\,inv}

\DeclareMathOperator\supp{supp}

\DeclareMathOperator\LP{LP}
\DeclareMathOperator\cl{cl}
\DeclareMathOperator\Adm{Adm}
\DeclareMathOperator\Perm{Perm}
\DeclareMathOperator\GSp{GSp}

\def\GL{{\mathrm{GL}}}

\def\dom{{\mathrm{dom}}}

\hypersetup{colorlinks=false, linkbordercolor={white}, hidelinks, pdfauthor={Felix Schremmer},pdftitle={Affine Bruhat order and Demazure products}}
\usepackage{csquotes}
\author{Felix Schremmer}\date{\today}
\title{Affine Bruhat order and Demazure products}
\numberwithin{equation}{section}
\newtheorem{theorem}[equation]{Theorem}
\newtheorem{proposition}[equation]{Proposition}
\newtheorem{lemma}[equation]{Lemma}
\newtheorem{corollary}[equation]{Corollary}

\theoremstyle{definition}
\newtheorem{definition}[equation]{Definition}
\newtheorem{situation}[equation]{Situation}
\theoremstyle{remark}
\newtheorem{example}[equation]{Example}
\newtheorem{remark}[equation]{Remark}
\def\abs#1{{\left\lvert{#1}\right\rvert}}

\let\oldqedsymbol\qedsymbol
\def\qedaddendum{}
\def\qedsymbol{\oldqedsymbol\qedaddendum}

\def\af{{\mathrm{af}}}
\def\rightqed{\pushQED{\qed}\qedhere\popQED}
\def\presig{\prescript\sigma{}}
\newif\ifthesis
\thesisfalse
\def\textname{article}
\def\Xast{X_\ast}
\def\crossRef#1#2{\cite[\protect\NoHyper#1\ref{ex-#2}\protect\endNoHyper]{Schremmer2022_newton}}
\def\weightEstimateLong{\crossRef{Lemma~}{lem:weightEstimate}}

\def\positiveLengthFormulaLong{Lemma~\ref{lem:positiveLengthFormula}}
\def\positiveLengthFormulaShort{L\ref{lem:positiveLengthFormula}}
\begin{document}
\maketitle
% !TeX spellcheck = en_GB
\begin{abstract}We give new descriptions of the Bruhat order and Demazure products of affine Weyl groups in terms of the weight function of the quantum Bruhat graph. These results can be understood to describe certain closure relations concerning the Iwahori-Bruhat decomposition of an algebraic group. As an application towards affine Deligne-Lusztig varieties, we present a new formula for generic Newton points.\end{abstract}
% !TeX spellcheck = en_GB
\section{Introduction}
Let us begin by considering a Coxeter group $(W, S)$. The \emph{Bruhat order} on $W$ can be defined by inclusion of reduced words, namely $x_1\leq x_2$ if some reduced word for $x_1$ can be obtained from some fixed reduced word for $x_2$ by deleting any number of letters. This partial order is of central importance for the general theory of Coxeter groups, and it enjoys a number of remarkable properties and applications \cite[Chapter~2 and beyond]{Bjorner2005}. E.g.\ the Kazhdan-Lusztig polynomials associated with $(W, S)$ satisfy that $P_{u,v}\neq 0$ if and only if $u\leq v$ \cite[Proposition~5.1.5]{Bjorner2005}.

Related to this is the notion of \emph{Demazure products}. The Demazure product $x_1\ast x_2$ of two elements $x_1, x_2\in W$ is the largest element of the form $x_1' x_2'\in W$ where $x_1'\leq x_1$ and $x_2'\leq x_2$ in the Bruhat order. The Demazure product describes the multiplication in the $0$-Hecke algebra of $(W, S)$, cf.\ \cite[Section~1.2]{He2021c}. It, too, has a number of remarkable properties and applications.

In this paper, we focus on a specific class of (quasi-)Coxeter groups, namely \emph{affine Weyl groups}. These groups arise naturally in the context of arithmetic geometry. In a sense, affine Weyl groups are the \enquote{simplest} examples of infinite Coxeter groups, so they are also important examples from a pure Coxeter theoretic viewpoint.

If $G$ is a connected reductive group over a non-archimedian local field $F$, we get an associated \emph{extended affine Weyl group} $\widetilde W$. This group famously occurs as the indexing set of the \emph{Iwahori-Bruhat decomposition}
\begin{align*}
G(\breve F) = \bigsqcup_{x\in \widetilde W} IxI.
\end{align*}
Here, $\breve F$ is the maximal unramified extension of $F$, and $I\subseteq G(\breve F)$ is an Iwahori subgroup.

The closure relations of the above decomposition are precisely given by the Bruhat order, i.e.\
\begin{align*}
\overline{IxI} = \bigsqcup_{y\leq x}IyI\subseteq G(\breve F).
\end{align*}
If $x, y\in \widetilde W$, the product $IxI\cdot IyI\subseteq G(\breve F)$ will in general not be of the form $IzI$ for any $z\in \widetilde W$. However, if we pass to closures, we have
\begin{align*}
\overline{IxIyI} = \overline{I(x\ast y)I}
\end{align*}
for the Demazure product.

The Iwahori-Bruhat decomposition has been studied intensively, partly because of its connection to the Bruhat-Tits building \cite[Section~4]{Bruhat1972}. Due to this, both the Bruhat order and Demazure products of affine Weyl groups have been used and studied in the past. We mention the definition of admissible sets due to Kottwitz and Rapoport \cite{Kottwitz2000, Rapoport2002}, the description of generic Newton points in terms of the Bruhat order due to Viehmann \cite{Viehmann2014} and the recent works on generic Newton points and Demazure products due to He and Nie \cite{He2021b, He2021c}.

The \emph{Iwahori Hecke algebra} $\mathcal H$ of $G$, that received tremendous interest starting with the discovery of the Satake isomorphism \cite{Satake1963}, can be defined as follows: $\mathcal H$ is an algebra over $\mathbb Z[v, v^{-1}]$, and it is a free $\mathbb Z[v^{\pm 1}]$ module with basis given by $\{T_x\mid x\in \widetilde W\}$. The multiplication is defined by
\begin{align*}
T_x T_y =& T_{xy}\qquad\qquad\qquad\text{ if }\ell(xy) = \ell(x) + \ell(y),
\\T_s^2 =&(v -v^{-1})T_s+1~~\text{ if }s\in \widetilde W\text{ is a simple affine reflection}.
\end{align*}
The multiplication of the Iwahori Hecke algebra is quite complicated and poorly understood. For $x, y\in \widetilde W$, the product $T_x T_y$ will in general have the form
\begin{align*}
T_x T_y = \sum_{z\in \widetilde W}f_{x,y,z}(v-v^{-1})T_z
\end{align*}
for some polynomials $f_{x,y,z}(X)\in \mathbb Z[X]$. This product $T_x T_y$ can be seen as a combinatorial model for the multiplication of Iwahori double cosets $IxI\cdot IyI$ in $G(\breve F)$. Among all $z\in \widetilde W$ such that $f_{x,y,z}\neq 0$, there is a unique largest one, which is the Demazure product $z = x\ast y$. We may summarize that understanding Demazure products is a first step towards fully understanding the multiplication in Iwahori Hecke algebras, which is related to important geometric problems. E.g.\ the dimensions of affine Deligne-Lusztig varieties can be expressed in terms of degrees of class polynomials of the Iwahori-Hecke algebra \cite[Theorem~6.1]{He2014}. In view of this connection, our result on Demazure products is enough to describe generic Newton points associated with the Iwahori-Bruhat decomposition of an algebraic group.

Our main results fully describe the Bruhat order and Demazure products for $\widetilde W$. We refer to the corresponding sections for the most general statements. To summarize our results roughly, recall that each element $x\in \widetilde W$ can be written as $x = w\varepsilon^\mu$, where $w$ is an element of the \emph{finite Weyl group} $W$ and $\mu$ is an element of an abelian group denoted $X_\ast$ (that can be chosen as the coweight lattice of our root system). By $\wt : W\times W\rightarrow X_\ast$, we denote the weight function of the quantum Bruhat graph, cf.\ Section~\ref{chap:quantum-bruhat-graph}.
\begin{theorem}\label{thm:bruhat1}
Let $x_1, x_2\in \widetilde W$, and write them as $x_1 = w_1\varepsilon^{\mu_1}, x_2 = w_2 \varepsilon^{\mu_2}$. Then $x_1\leq x_2$ in the Bruhat order if and only if for each $v_1\in W$, there exists some $v_2\in W$ satisfying
\begin{align*}
v_1^{-1}\mu_1 + \wt(v_2\Rightarrow v_1) + \wt(w_1 v_1\Rightarrow w_2 v_2)\leq v_2^{-1}\mu_2.
\end{align*}
\end{theorem}
For more refined descriptions of the Bruhat order, we refer to Theorems \ref{thm:bruhat2} and \ref{thm:bruhat4} as well as Remark~\ref{rem:bruhat5}. The description of Demazure products has the following form:
\begin{theorem}[Cf.\ Theorem~\ref{thm:demazure}]\label{thm:demazure-rough}
Let $x_1, x_2\in \widetilde W$, written as $x_1 = w_1\varepsilon^{\mu_1}$ and $x_2 = w_2 \varepsilon^{\mu_2}$. Then for explicitly described $v_1, v_2\in W$, we have
\begin{align*}
x_1 \ast x_2 = w_1 v_1 v_2^{-1}\varepsilon^{v_2 v_1^{-1}\mu_1 + \mu_2 - v_2\wt(v_1\Rightarrow w_2 v_2)}.
\end{align*}
\end{theorem}

As an application of our results, we describe the admissible sets as introduced in \cite{Kottwitz2000} and \cite{Rapoport2002} as Propositions \ref{prop:admissibleSubset} and \ref{prop:generalizedAdmissibleSubset}. We also get an explicit description of Bruhat covers in $\widetilde W$ (Proposition~\ref{prop:bruhatCovers}) and of the semi-infinite order on $\widetilde W$ (Corollary~\ref{cor:semiInfiniteOrder}). Finally, combining the aforementioned result of Viehmann \cite{Viehmann2014} with ideas of He \cite{He2021b}, we present a new description of generic Newton points (Theorem~\ref{thm:gnpViaDemazure}).

The methods of this paper build upon a previous paper by the same author \cite{Schremmer2022_newton}. In particular, the language and results on length functionals as introduced there will be used throughout this paper. To complement the combinatorial prerequisites, this paper introduces and proves a number of new properties of the quantum Bruhat graph in Sections \ref{chap:quantum-bruhat-graph} and \ref{sec:generic-action}. These new results on the quantum Bruhat graph are not only the foundation of our results on the Bruhat order and Demazure products, they also may have potentially further-reaching applications, given the previous usage of the quantum Bruhat graph for quantum cohomology \cite{Postnikov2005} or Kirillov-Reshetikhin crystals \cite{Lenart2015, Lenart2017}. In addition to the previously studied weight functions of the (parabolic) quantum Bruhat graph, we introduce a new \emph{semi-affine} weight function.

Note that while both this paper and our previous paper \cite{Schremmer2022_newton} provide explicit formulas for generic Newton points, these results are actually complementing rather than overlapping. In terms of logical dependencies, this paper only relies on the discussion of root functionals and length positivity in Section~2.2 of \cite{Schremmer2022_newton} and is otherwise independent. Together, both papers cover the contents of the author's PhD thesis.
\subsection{Acknowledgements}
First and foremost, I would like to thank my advisor Eva Viehmann for her constant support throughout my PhD time. I am deeply thankful for her invaluable help in both mathematical and administrative matters.

I would like to thank Paul Hamacher and Xuhua He for inspiring discussions.

The author was partially supported by the ERC Consolidator Grant 770936: \emph{NewtonStrat}, the German Academic Scholarship Foundation, the Marianne-Plehn-Program and the DFG Collaborative Research Centre 326: \emph{GAUS}.
% !TeX spellcheck = en_GB
 \section{Affine root system}
In this section, we describe the fundamental root-theoretic setup. In the literature, there are several different notions of \emph{affine Weyl groups} studied in different contexts, so we present a uniform setup that covers all cases. Readers with a combinatorial background are invited to consider any reduced root datum, whereas readers whose background is closer to arithmetic geometry may find more appealing to have the context of an algebraic group, as presented e.g.\ in \crossRef{Section~}{sec:notation}.

Let $\Phi$ be a reduced crystallographic root system. We choose a basis $\Delta\subseteq \Phi$ and denote the set of positive/negative roots by $\Phi^{\pm}$. 

Let $X_\ast$ denote an abelian group with a fixed embedding of the coroot lattice $\mathbb Z\Phi^\vee\subseteq X_\ast$. The group $X_\ast$ is allowed to have a torsion part. We assume that a bilinear map
\begin{align*}
\langle\cdot,\cdot\rangle : X_\ast\otimes \mathbb Z\Phi\rightarrow\mathbb Z
\end{align*}
has been chosen that extends the natural pairing between $\Phi^\vee$ and $\Phi$. E.g.\ both the coroot lattice $X_\ast = \mathbb Z\Phi^\vee$ and the coweight lattice $X_\ast = \Hom_{\mathbb Z}(\mathbb Z\Phi, \mathbb Z)$ are possible choices for $X_\ast$. We turn $X_\ast$ and $X_\ast\otimes\mathbb Q$ into ordered abelian groups by defining that $\mu_1\leq\mu_2$ if $\mu_2-\mu_1$ is a $\mathbb Z_{\geq 0}$-linear, resp.\ $\mathbb Q_{\geq 0}$-linear, combination of positive coroots. An element $\mu$ in $X_\ast$ or $X_\ast\otimes\mathbb Q$ will be called \emph{$C$-regular} for some constant $C>0$ if $\abs{\langle\mu,\alpha\rangle}\geq C$ for all $\alpha\in \Phi$. Typically, we will not specify the constant and talk of \emph{sufficiently regular} or \emph{superregular} elements. An element $\mu$ in $X_\ast$ or $X_\ast\otimes\mathbb Q$ is dominant if $\langle \mu,\alpha\rangle\geq 0$ for each positive root $\alpha$.

Denote the Weyl group of $\Phi$ by $W$, and the set of simple reflections by \begin{align*}
S = \{s_\alpha\mid\alpha\in \Delta\}\subseteq W.
\end{align*}
The Weyl group $W$ acts on $X_\ast$ via the usual convention
\begin{align*}
s_\alpha(\mu) = \mu-\langle\mu,\alpha\rangle\alpha^\vee,\qquad \alpha\in \Phi,~\mu\in X_\ast.
\end{align*}
The semi-direct product $\widetilde W:=W\ltimes X_\ast$ is called \emph{extended affine Weyl group}. Elements in $\widetilde W$ will typically be expressed as $x = w\varepsilon^\mu\in \widetilde W$ with $w\in W$ and $\mu\in X_\ast$.

By abuse of notation, we write $\Phi^+$ for the indicator function of positive roots, i.e.
\begin{align*}
\Phi^+(\alpha) :=\begin{cases}1,&\alpha\in \Phi^+,\\
0,&\alpha\in \Phi^-.\end{cases}
\end{align*}
The following easy facts will be used often, usually without further reference:
\begin{lemma}\label{lem:phiPlusFacts}
Let $\alpha\in \Phi$.
\begin{enumerate}[(a)]
\item $\Phi^+(\alpha) + \Phi^+(-\alpha)=1$.
\item If $\beta\in \Phi$ and $k,\ell\geq 1$ are such that $k\alpha+\ell \beta\in \Phi$, we have
\begin{align*}
&0\leq \Phi^+(\alpha)+\Phi^+(\beta)-\Phi^+(k\alpha+\ell\beta)\leq 1.\rightqed
\end{align*}
\end{enumerate}
\end{lemma}
The sets of \emph{affine roots, positive affine roots, negative affine roots} and \emph{simple affine roots} are given by
\begin{align*}
\Phi_\af :=&\Phi\times\mathbb Z,
\\\Phi_\af^+:=&(\Phi^+\times \mathbb Z_{\geq 0})\sqcup (\Phi^-\times \mathbb Z_{\geq 1}) = \{(\alpha,k)\in \Phi_\af\mid k\geq \Phi^+(-\alpha)\},
\\\Phi_\af^- :=&-\Phi_\af^+ = \Phi_\af\setminus \Phi_\af^+= \{(\alpha,k)\in \Phi_\af\mid k< \Phi^+(-\alpha)\},
\\\Delta_\af:=&\{(\alpha,0)\mid\alpha\in \Delta\}\cup\\&~\{(-\theta,1)\mid\theta\text{ is the longest root of an irreducible component }\Phi'\subseteq \Phi\}\subseteq \Phi_\af^+.
\end{align*}
One checks that the positive affine roots are precisely those affine roots which are a sum of simple affine roots.

The action of $\widetilde W$ on $\Phi_\af$ is given by
\begin{align*}
(w\varepsilon^\mu)(\alpha,k) := (w\alpha,k-\langle \mu,\alpha\rangle).
\end{align*}
The \emph{length} of an element $x=w\varepsilon^\mu\in \widetilde W$ is defined as
\begin{align*}
\ell(x) := \#\{a\in \Phi_\af^+\mid xa\in\Phi_\af^-\}.
\end{align*}

Associated to each affine root $a=(\alpha,k)$, we have the affine reflection
\begin{align*}
r_a = s_\alpha \varepsilon^{k\alpha^\vee}\in \widetilde W.
\end{align*}
Denote by $W_\af\subseteq W$ the subgroup generated by the affine reflections (called \emph{affine Weyl group}) and write $S_\af := \{r_a\mid a\in \Delta_\af\}$ (the set of \emph{simple affine reflections}). It is easy to check that $(W_\af, S_\af)$ is a Coxeter group with length function $\ell$ as defined above, and $W_\af = W\ltimes\mathbb Z\Phi^\vee\subseteq \widetilde W$.

Denoting the subgroup of length zero elements of $\widetilde W$ by $\Omega \leq \widetilde W$, we get a semi-direct product decomposition $\widetilde W = \Omega\ltimes W_\af$.

The Bruhat order on $W_\af$ is the usual Coxeter-theoretic notion. We define the Bruhat order on $\widetilde W$ by declaring that
\begin{align*}
\omega_1 x_1\leq \omega_2 x_2\iff \left(\omega_1=\omega_2\text{ and }x_1\leq x_2\in W_\af\right),
\end{align*}
where $\omega_1, \omega_2\in \Omega$ and $x_1, x_2\in W_\af$. Equivalently, this is the partial order on $\widetilde W$ generated by the relations $x< xr_a$ for $x\in \widetilde W$ and $a\in \Phi_\af$ such that $\ell(x)<\ell(xr_a)$.

We will occasionally denote the \emph{classical part} of an affine root $a = (\alpha,k)$ or an element $x = w\varepsilon^\mu\in \widetilde W$ by
\begin{align*}
\cl(a) = \alpha\in \Phi,\qquad \cl(x) = w\in W.
\end{align*}

We need the language of length functionals from \crossRef{Section~}{sec:root-functionals}. We recall the basic definitions here, and refer to the cited paper for some geometric intuition and fundamental properties.
\begin{definition}Let $x = w\varepsilon^\mu\in \widetilde W$.
\begin{enumerate}[(a)]
\item For $\alpha\in \Phi$, we define the \emph{length functional} of $x$ by
\begin{align*}
\ell(x,\alpha) := \langle\mu,\alpha\rangle + \Phi^+(\alpha) - \Phi^+(w\alpha).
\end{align*}
\item An element $v\in W$ is called \emph{length positive for $x$}, written as $v\in \LP(x)$, if every positive root $\alpha\in \Phi^+$ satisfies $\ell(x,v\alpha)\geq 0$.
\item If $v\in W$ is not length positive for $x$ and $\alpha\in \Phi^+$ satisfies $\ell(x,v\alpha)<0$, we call $vs_\alpha\in W$ an \emph{adjustment} of $v$ for $\ell(x,\cdot)$.
\end{enumerate}
\end{definition}
The name \enquote{length functional} comes from the fact that the length of $x$ can be expressed as the sum of all positive values $\ell(x,\alpha)$ for $\alpha\in \Phi$.

We prove in \crossRef{Lemma~}{lem:rootFunctionalAdjustment} that iteratively adjusting any given $v\in W$ yields a length positive element for $x$. The following characterization of length positive elements will frequently come handy:
\begin{lemma}[{\crossRef{Corollary~}{cor:positiveLengthFormula}}]\label{lem:positiveLengthFormula}
Let $x = w\varepsilon^\mu\in \widetilde W$ and $v\in W$. Then
\begin{align*}
\ell(x)\geq \langle v^{-1}\mu,2\rho\rangle - \ell(v) + \ell(wv).
\end{align*}
Equality holds if and only if $v$ is length positive for $x$.\rightqed
\end{lemma}
The length functional can be used to characterize the \emph{shrunken Weyl chambers} \crossRef{Proposition~}{prop:shrunkenWeylChamber}: We have that $x \in \widetilde W$ is in a shrunken Weyl chamber if and only if $\ell(x,\alpha)\neq 0$ for all $\alpha\in \Phi$, which is equivalent to saying that $\LP(x)$ contains only one element.
% !TeX spellcheck = en_GB
\section{Quantum Bruhat graph}\label{chap:quantum-bruhat-graph}
In this section, we recall the definition of quantum Bruhat graphs and study its weight functions.
Before turning to the abstract theory of these graphs, we will discuss the situation of root systems of type $A_n$ as a motivational example.

For each simple affine root $a = (\alpha,k)\in \Delta_\af$, we define a coweight $\omega_a \in \mathbb Q \Phi^\vee$ as follows: For $\beta\in \Delta$, we define
\begin{align*}
\langle \omega_a,\beta\rangle = \begin{cases}1,&\alpha=\beta,\\
0,&\alpha\neq\beta.\end{cases}
\end{align*}
In particular, $\omega_a = 0$ if $\alpha\notin \Delta$.

Let now $x_1 = w_1\varepsilon^{\mu_1}, x_2 = w_2 \varepsilon^{\mu_2}\in \widetilde W$. By \cite[Theorem~8.3.7]{Bjorner2005}, we have
\begin{align*}
x_1\leq x_2\iff \forall a,a'\in \Delta_{\af}:~(\mu_1 + \omega_a - w_1^{-1}\omega_{a'})^{\dom} \leq (\mu_2 + \omega_a -w_2^{-1}\omega_{a'})^{\dom}.
\end{align*}
Here, we write $\nu^\dom\in X_\ast$ for the unique dominant element in the $W$-orbit of $\nu\in X_\ast$.

Suppose that $\mu_1$ and $\mu_2$ are sufficiently regular, such that we find $v_1, v_2\in W$ with
\begin{align*}
\forall a, a'\in \Delta_{\af}:~(\mu_i + \omega_a - w_i^{-1}\omega_{a'})^{\dom} = v_i^{-1}(\mu_i + \omega_a - w_i^{-1}\omega_{a'}).
\end{align*}
Then we conclude
\begin{align*}
x_1\leq x_2\iff& \forall a, a':~v_1^{-1}(\mu_1 + \omega_a - w_1^{-1}\omega_{a'})\leq v_2^{-1}(\mu_2 + \omega_a - w_2^{-1}\omega_{a'})
\\\iff&v_1^{-1}\mu_1 + \sup_{a\in \Delta_{\af}} (v_1^{-1}\omega_a - v_2^{-1}\omega_a) + \sup_{a'\in \Delta_{\af}} (w_2 v_2)^{-1} \omega_{a'} - (w_1 v_1)^{-1}\omega_{a'}\leq v_2^{-1}\mu_2.
\end{align*}
So if we define
\begin{align}
\wt(v_1\Rightarrow v_2) :=  \sup_{a\in \Delta_{\af}} (v_2^{-1}\omega_a - v_1^{-1}\omega_a),\label{eq:wtAn}
\end{align}
we can conclude a version of our result on the Bruhat order (Theorem~\ref{thm:bruhat1}).

Indeed, formula (\ref{eq:wtAn}) holds true for root systems of type $A_n$, but not for any other root system. Many properties of the weight function are easier to prove for type $A_n$, where an explicit formula exists, so it is helpful to keep this example in mind.

We refer to a paper of Ishii \cite{Ishii2021} for explicit formulas for the weight functions of all classical root systems (while he discusses explicit criteria for the semi-infinite order, these can be translated to explicit formulas for the weight function as outlined above in the $A_n$ case).
\subsection{(Parabolic) quantum Bruhat graph}
We start with a discussion of the quantum roots in $\Phi^+$.
\begin{lemma}\label{lem:quantumRoots}
Let $\alpha \in \Phi^+$. Then
\begin{align*}
\ell(s_\alpha) \leq \langle \alpha^\vee,2\rho\rangle-1.
\end{align*}
Equality holds if and only if for all $\alpha\neq \beta\in \Phi^+$ with $s_\alpha(\beta)\in \Phi^-$, we have $\langle \alpha^\vee,\beta\rangle = 1$.
\end{lemma}
Roots satisfying the equivalent properties of Lemma~\ref{lem:quantumRoots} are called \emph{quantum roots}. We see that all long roots are quantum (so in a simply laced root system, all roots are quantum). Moreover, all simple roots are quantum.

The first inequality of Lemma~\ref{lem:quantumRoots} is due to \cite[Lemma~4.3]{Brenti1998}. By \cite[Lemma~7.2]{Braverman2011}, we have the following more explicit (but somehow less useful for us) result: A short root $\alpha$ is quantum if and only if $\alpha$ is a sum of short simple roots.
\begin{proof}[Proof of Lemma~\ref{lem:quantumRoots}]
We calculate
\begin{align*}
\langle \alpha^\vee, 2\rho\rangle = \frac 12\left(\langle \alpha^\vee, 2\rho\rangle + \langle s_\alpha(\alpha^\vee), s_\alpha(2\rho)\rangle\right) = \frac 12\langle \alpha^\vee, 2\rho - s_\alpha(2\rho)\rangle.
\end{align*}
Let
\begin{align*}
I := \{\beta\in \Phi^+\mid s_\alpha(\beta)\in \Phi^-\}.
\end{align*}
Then $s_\alpha(I) = -I$ and $s_\alpha(\Phi^+\setminus I) = \Phi^+\setminus I$. It follows that
\begin{align*}
2\rho - s_\alpha(2\rho) = &\sum_{\beta \in I} \left(\beta - s_\alpha(\beta)\right) + \sum_{\beta\in \Phi^+\setminus I} \left(\beta-s_\alpha(\beta)\right)
\\=&2\sum_{\beta\in I} \beta.
\end{align*}
Therefore, we obtain
\begin{align*}
\langle \alpha^\vee, 2\rho\rangle = \sum_{\beta\in I}\langle\alpha^\vee, \beta\rangle.
\end{align*}
Certainly, $\alpha\in I$. Hence
\begin{align*}
\langle \alpha^\vee, 2\rho\rangle = 2+\sum_{\substack{\alpha \neq \beta\in \Phi^+\\s_\alpha(\beta)\in \Phi^-}}\langle\alpha^\vee, \beta\rangle.
\end{align*}
Now if $\alpha,\beta\in \Phi^+$ and $s_\alpha(\beta) = \beta-\langle \alpha^\vee, \beta\rangle \alpha\in \Phi^-$, we get $\langle \alpha^\vee, \beta\rangle\geq 1$. We conclude
\begin{align*}
\langle \alpha^\vee, 2\rho\rangle = 2+\sum_{\substack{\alpha \neq \beta\in \Phi^+\\s_\alpha(\beta)\in \Phi^-}}\langle\alpha^\vee, \beta\rangle\geq 2 + \#\{\beta\in \Phi^+\setminus\{\alpha\}\mid s_\alpha(\beta)\in \Phi^-\} = 1+\ell(s_\alpha).
\end{align*}
All claims of the lemma follow immediately from this.
\end{proof}
The parabolic quantum Bruhat graph as introduced by Lenart-Naito-Sagaki-Schilling-Schi\-mo\-zo\-no \cite{Lenart2015} is a generalization of the classical construction of the quantum Bruhat graph by Brenti-Fomin-Postnikov \cite{Brenti1998}. To avoid redundancy, we directly state the definition of the parabolic quantum Bruhat graph, even though we will be mostly concerned with the (ordinary) quantum Bruhat graph.

Fix a subset $J\subseteq \Delta$. We denote by $W_J$ the Coxeter subgroup of $W$ generated by the reflections $s_\alpha$ for $\alpha\in J$. We let
\begin{align*}
W^J = \{w\in W\mid w(J)\subseteq \Phi^+\}.
\end{align*}
For each $w\in W$, let $w^J\in W^J$ and $w_J\in W_J$ be the uniquely determined elements with $w = w^J \cdot w_J$ \cite[Proposition~2.4.4]{Bjorner2005}.

We write $\Phi_J = W_J(J)$ for the root system generated by $J$. The sum of positive roots in $\Phi_J$ is denoted $2\rho_J$. The quotient lattice $\mathbb Z\Phi^\vee / \mathbb Z \Phi_J^\vee$ is ordered by declaring $\mu_1+\Phi_J^\vee \leq \mu_2 + \Phi_J^\vee$ if the difference $\mu_2 - \mu_1 + \Phi_J^\vee$ is equal to a sum of positive coroots modulo $\Phi_J^\vee$.
\begin{definition}
\begin{enumerate}[(a)]
\item The \emph{parabolic quantum Bruhat graph} associated with $W^J$ is a directed and $(\mathbb Z\Phi^\vee/\mathbb Z\Phi_J^\vee)$-weighted graph, denoted $\QB(W^J)$. The set of vertices is given by $W^J$. For $w_1, w_2\in W^J$, we have an edge $w_1\rightarrow w_2$ if there is a root $\alpha\in \Phi^+\setminus \Phi_J$ such that $w_2 = (w_1s_\alpha)^J$ and one of the following conditions is satisfied:\begin{itemize}
\item[(B)] $\ell(w_2) = \ell(w_1)+1$ or
\item[(Q)] $\ell(w_2) = \ell(w_1) + 1 - \langle \alpha^\vee,2\rho - 2\rho_J\rangle$.
\end{itemize}
Edges of type (B) are \emph{Bruhat edges} and have weight $0\in \mathbb Z\Phi^\vee/\mathbb Z\Phi_J^\vee$. Edges of type (Q) are \emph{quantum edges} and have weight $\alpha^\vee \in \mathbb Z\Phi^\vee/\mathbb Z\Phi_J^\vee$.
\item A \emph{path} in $\QB(W^J)$ is a sequence of adjacent edges
\begin{align*}
p : w = w_1\rightarrow w_2\rightarrow\cdots \rightarrow w_{\ell+1} = w'.
\end{align*}
The \emph{length} of $p$ is the number of edges, denoted $\ell(p)\in \mathbb Z_{\geq 0}$. The \emph{weight} of $p$ is the sum of its edges' weights, denoted $\wt(p)\in \mathbb Z\Phi^\vee/\mathbb Z\Phi^\vee_J$. We say that $p$ is a path \emph{from $w$ to $w'$}.
\item If $w, w'\in W^J$, we define the \emph{distance function} by
\begin{align*}
d_{\QB(W^J)}(w\Rightarrow w') = \inf\{\ell(p)\mid p\text{ is a path in $\QB(W^J)$ from $w$ to $w'$}\}\in \mathbb Z_{\geq 0}\cup\{\infty\}.
\end{align*}
A path $p$ from $w$ to $w'$ of length $d_{\QB(W^J)}(w\Rightarrow w')$ is called \emph{shortest}.
\item The \emph{quantum Bruhat graph} of $W$ is the parabolic quantum Bruhat graph associated with $J=\emptyset$, denoted $\QB(W) := \QB(W^\emptyset)$. We also shorten our notation to \begin{align*}
d(w\Rightarrow w') := d_{\QB(W)}(w\Rightarrow w').
\end{align*}
\end{enumerate}
\end{definition}
%We refer to the appendix for pictures of some quantum Bruhat graphs.
\begin{remark}
Let us consider the case $J=\emptyset$, i.e.\ the quantum Bruhat graph. If $w \in W$ and $\alpha\in \Delta$, then $w\rightarrow ws_\alpha$ is always an edge of weight $\alpha^\vee \Phi^+(-w\alpha)$.

The quantum edges are precisely the edges of the form $w\rightarrow ws_\alpha$ where $\alpha$ is a quantum root and $\ell(ws_\alpha) = \ell(w) - \ell(s_\alpha)$.
\end{remark}
\begin{proposition}[{\cite[Proposition~8.1]{Lenart2015} and \cite[Lemma~7.2]{Lenart2017}}]\label{prop:qbgWeightDefinition}Consider $w, w'\in W^J$.
\begin{enumerate}[(a)]
\item The graph $\QB(W^J)$ is strongly connected, i.e.\ there exists a path from $w$ to $w'$ in $\QB(W^J)$.
\item Any two shortest paths from $w$ to $w'$ have the same weight, denoted \begin{align*}
\wt_{\QB(W^J)}(w\Rightarrow w')\in \mathbb Z\Phi^\vee/\mathbb Z\Phi_J^\vee.
\end{align*}
\item Any path $p$ from $w$ to $w'$ has weight $\wt(p)\geq \wt_{\QB(W^J)}(w\Rightarrow w')\in \mathbb Z\Phi^\vee/\mathbb Z\Phi_J^\vee$.
\item The image of
\begin{align*}
\wt(w\Rightarrow w') := \wt_{\QB(W)}(w\Rightarrow w')\in \mathbb Z\Phi^\vee
\end{align*}
under the canonical projection $\mathbb Z\Phi^\vee\rightarrow \mathbb Z\Phi^\vee/\mathbb Z\Phi_J^\vee$ is given by $\wt_{\QB(W^J)}(w\Rightarrow w')$.\rightqed
\end{enumerate}
\end{proposition}
One interpretation of the weight function is that it measures the failure of the inequality $w_1 W_J\leq w_2 W_J$ in the Bruhat order on $W/W_J$ (cf.\ \cite[Section~2.5]{Bjorner2005}): Indeed,  $w_1W_J\leq w_2 W_J$ if and only if $\wt_{\QB(W^J)}(w_1\Rightarrow w_2)=0$.

We have the following converse to part (c) of Proposition~\ref{prop:qbgWeightDefinition}:
\begin{lemma}[{Cf.\ \cite[Formula~4.3]{Milicevic2020}}]\label{lem:weight2rho}
Let $w_1, w_2\in W^J$.
 For any path $p$ from $w_1$ to $w_2$ in $\QB(W^J)$, we have
\begin{align*}
\langle \wt(p),2\rho-2\rho_J\rangle = \ell(p) + \ell(w_1) - \ell(w_2).
\end{align*}
In particular, 
\begin{align*}
\langle \wt_{\QB(W^J)}(w_1\Rightarrow w_2),2\rho-2\rho_J\rangle = d_{\QB(W^J)}(w_1\Rightarrow w_2) + \ell(w_1) - \ell(w_2),
\end{align*}
and $p$ is shortest if and only if $\wt(p) = \wt_{\QB(W^J)}(w_1\Rightarrow w_2)$.
\end{lemma}
\begin{proof}
Note that if $p: w_1\rightarrow w_2 = (w_1s_\alpha)^J$ is an edge in $\QB(W^J)$, then by definition,
\begin{align*}
\ell(w_2) = \ell(w_1) + 1 - \langle \wt(p),2\rho-2\rho_J\rangle.
\end{align*}
In general, iterate this observation for all edges of $p$.
\end{proof}
The weights of non-shortest paths do not add more information:
\begin{lemma}\label{lem:pathForWeight}
Let $\mu \in \mathbb Z\Phi^\vee/\mathbb Z\Phi_J^\vee$ and $w_1, w_2\in W$. Then $\mu\geq \wt_{\QB(W^J)}(w_1\Rightarrow w_2)$ if and only if there is a path $p$ from $w_1$ to $w_2$ in $\QB(W^J)$ of weight $\mu$.
\end{lemma}
\begin{proof}
By part (d) of Proposition~\ref{prop:qbgWeightDefinition}, it suffices to consider the case $J=\emptyset$, i.e.\ the quantum Bruhat graph.

The \emph{if} condition is part (c) of Proposition~\ref{prop:qbgWeightDefinition}. It remains to show the \emph{only if} condition. Note that for each $w\in W$ and $\alpha\in\Delta$, we get a \enquote{silly path} of the form
\begin{align*}
w\rightarrow ws_\alpha\rightarrow w
\end{align*}
in $\QB(W)$. Precisely one of the edges is quantum with weight $\alpha^\vee$, and the other one is Bruhat with weight $0$.

If $\mu\geq \wt(w_1\Rightarrow w_2)$, we may compose a shortest path from $w_1$ to $w_2$ with suitably chosen silly paths as above to obtain a path from $w_1$ to $w_2$ of weight $\mu$.
\end{proof}
\begin{lemma}[{\cite[Lemma~7.7]{Lenart2015}}]\label{lem:qbgDiamond}Let $J\subseteq \Delta$, $w_1, w_2\in W^J$ and $a = (\alpha,k)\in \Delta_\af$ such that $w_2^{-1}\alpha \in \Phi^-$.
\begin{enumerate}[(a)]
\item We have an edge $(s_\alpha w_2)^J\rightarrow w_2$ in $\QB(W^J)$ of weight  $-kw_2^{-1}\alpha^\vee\in\mathbb Z\Phi^\vee/\mathbb Z\Phi^\vee_J$.
\item If $w_1^{-1}\alpha\in \Phi^+$, then the above edge is part of a shortest path from $w_1$ to $w_2$, i.e.
\begin{align*}
d_{\QB(W^J)}(w_1\Rightarrow w_2) = d_{\QB(W^J)}(w_1\Rightarrow (s_\alpha w_2)^J) + 1.
\end{align*}
\item If $w_1^{-1}\alpha\in \Phi^-$, we have
\begin{align*}
d_{\QB(W^J)}(w_1\Rightarrow w_2) =& d_{\QB(W^J)}((s_\alpha w_1)^J\Rightarrow (s_\alpha w_2)^J),
\\
\wt_{\QB(W^J)}(w_1\Rightarrow w_2) =& \wt_{\QB(W^J)}((s_\alpha w_1)^J\Rightarrow (s_\alpha w_2)^J) + k(w_1^{-1}\alpha^\vee - w_2^{-1}\alpha^\vee).\rightqed
\end{align*}
\end{enumerate}
\end{lemma}
We can use this lemma to reduce the calculation of weights $\wt(w_1\Rightarrow w_2)$ to weights of the form $\wt(w\Rightarrow 1)$: If $w_2\neq 1$, we find a simple root $\alpha\in \Delta$ with $w_2^{-1}\alpha\in \Phi^-$. Then
\begin{align*}
\wt(w_1\Rightarrow w_2) =& \begin{cases}\wt(w_1\Rightarrow s_\alpha w_2),&w_1^{-1}\alpha\in \Phi^+,\\
\wt(s_\alpha w_1\Rightarrow s_\alpha w_2),&w_1^{-1}\alpha \in \Phi^-,\end{cases}
\\=&\wt(\min(w_1, s_\alpha w_1), s_\alpha w_2).
\end{align*}
For an alternative proof of this reduction, cf.\ \cite[Corollary~3.3]{Sadhukhan2021}.

The quantum Bruhat graph has a number of useful automorphisms.

\begin{lemma}\label{lem:weightIdentities}
Let $w_1, w_2\in W$, and let $w_0\in W$ be the longest element.
\begin{enumerate}[(a)]
\item $\wt(w_0 w_1\Rightarrow w_0 w_2) = \wt(w_2 \Rightarrow w_1)$.
\item $\wt(w_0 w_1 w_0\Rightarrow w_0 w_2 w_0) = -w_0\wt(w_1 \Rightarrow w_2)$.
\item $\wt(w_1\Rightarrow 1) = \wt(w_1^{-1}\Rightarrow 1)$.
\end{enumerate}
\end{lemma}
\begin{proof}
Part (a) follows from \cite[Proposition~4.3]{Lenart2015}.

For part (b), observe that we have an automorphism of $\Phi$ given by $\alpha\mapsto -w_0\alpha$. The induced automorphism of $W$ is given by $w\mapsto w_0 w w_0$. Since the function $\wt(\cdot\Rightarrow\cdot)$ is compatible with automorphisms of $\Phi$, we get the claim.

Now for (c), consider a reduced expression
\begin{align*}
w_0 w_1 = s_1\cdots s_q.
\end{align*}
Then, iterating Lemma~\ref{lem:qbgDiamond}, we get
\begin{align*}
\wt(w_1\Rightarrow 1) \underset{\text{(a)}}=& \wt(w_0\Rightarrow w_0 w_1) = \wt(w_0\Rightarrow s_1\cdots s_q)
\\=&\wt(s_1w_0\Rightarrow s_2\cdots s_q)=\cdots = \wt(s_q\cdots s_1 w_0\Rightarrow 1)
\\=&\wt((w_0 w_1)^{-1}w_0\Rightarrow 1) = \wt(w_1^{-1}\Rightarrow 1).\qedhere
\end{align*}
\end{proof}
\subsection{Lifting the parabolic quantum Bruhat graph}
For sufficiently regular elements of the extended affine Weyl group, the Bruhat covers in $\widetilde W$ are in a one-to-one correspondence with edges in the quantum Bruhat graph \cite[Proposition~4.4]{Lam2010}. This result is very useful for deriving properties about the quantum Bruhat graph. Moreover, our strategy to prove our results on the Bruhat order will be to reduce to this superregular case.

The result of Lam and Shimozono has been generalized by Lenart et.\ al.\ \cite[Theorem~5.2]{Lenart2015}, and the extra generality of the latter result will be useful for us. Throughout this section, let $J\subseteq \Delta$ be any subset.
\begin{definition}[\cite{Lenart2015}]\label{def:liftingPQBG}
\begin{enumerate}[(a)]
\item Define
\begin{align*}
(W^J)_\af :=& \{x\in W_\af\mid \forall \alpha\in \Phi_J:~\ell(x,\alpha)=0\},
\\
\widetilde{(W^J)} :=& \{x\in \widetilde W\mid \forall \alpha\in \Phi_J:~\ell(x,\alpha)=0\}.
\end{align*}
\item Let $C>0$ be any real number. We define $\Omega_J^{-C}$ to be the set of all elements $x = w\varepsilon^\mu\in \widetilde{(W^J)}$ such that
\begin{align*}
\forall \alpha\in \Phi^+\setminus \Phi_J:~\langle\mu,\alpha\rangle\leq -C.
\end{align*}
Similarly, we say $x\in \Omega_J^C$ if
\begin{align*}
\forall \alpha\in \Phi^+\setminus \Phi_J:~\langle\mu,\alpha\rangle\geq C.
\end{align*}
\item For elements $x, x'\in \widetilde W$, we write $x\lessdot x'$ and call $x'$ a \emph{Bruhat cover} of $x$ if $\ell(x') = \ell(x)+1$ and $x^{-1} x'$ is an affine reflection in $\widetilde W$.
\end{enumerate}
\end{definition}
\begin{theorem}[{\cite[Theorem~5.2]{Lenart2015}}]\label{thm:pqbgLift}
There is a constant $C>0$ depending only on $\Phi$ such that the following holds:
\begin{enumerate}[(a)]
\item If $x = w\varepsilon^\mu\lessdot x' = w'\varepsilon^{\mu'}$ is a Bruhat cover with $x\in \Omega^{-C}_J$ and $x' \in \widetilde {(W^J)}$, there exists an edge $(w')^J\rightarrow w^J$ in $\QB(W^J)$ of weight $\mu- \mu'+\mathbb Z \Phi_J^\vee$.
\item If $x = w\varepsilon^\mu\in \Omega_J^{-C}$ and $\tilde w'\rightarrow w^J$ is an edge in $\QB(W^J)$ of weight $\omega$, then there exists a unique element $x\lessdot x' = w'\varepsilon^{\mu'}\in \widetilde {(W^J)}$ with $\tilde w' = (w')^J$ and $\mu \equiv \mu' + \omega\pmod{\mathbb Z \Phi_J^\vee}$.\rightqed
\end{enumerate}
\end{theorem}
This theorem \enquote{lifts} $\QB(W^J)$ into the Bruhat covers of $\Omega_J^{-C}$ for sufficiently large $C$.

The theorem is originally formulated only for $(W^J)_\af$, but the generalization to $\widetilde{(W^J)}$ is straightforward.

With a bit of book-keeping, we can compare paths in $\QB(W^J)$ (i.e.\ sequences of edges) with the Bruhat order on $\Omega_J^{-C}$ (i.e.\ sequences of Bruhat covers).
\begin{lemma}
Let $C_1>0$ be any real number. Then there exists some $C_2>0$ such that for all $x = w\varepsilon^{\mu}\in \Omega_J^{C_2}$ and $x' = w'\varepsilon^{\mu'}\in \widetilde {(W^J)}$ with $\ell(x^{-1}x')\leq C_1$, we have
\begin{align*}
x\leq x'\iff \mu - \wt(w'\Rightarrow w)\leq \mu'\pmod{\Phi_J^\vee}.
\end{align*}
The latter condition is shorthand for
\begin{align*}
\mu - \wt(w'\Rightarrow w)-\mu' + \mathbb Z\Phi_J^\vee\leq 0 + \mathbb Z\Phi_J^\vee\in \mathbb Z\Phi^\vee/\mathbb Z\Phi_J^\vee.
\end{align*}
\end{lemma}
\begin{proof}
Let $C>0$ be a constant sufficiently large for the conclusion of Theorem~\ref{thm:pqbgLift} to hold. We see that if $x_1\lessdot x_2$ is any cover in $\Omega_J^{-C}$, then there are only finitely many possibilities for $x_1^{-1} x_2$, so the length $\ell(x_1^{-1} x_2)$ is bounded. We fix a bound $C'>0$ for this length.

We can pick $C_2>0$ such that for all $x_1 = w\varepsilon^\mu\in \Omega_J^{-C_2}$ and $x_2\in \widetilde W^J$ with $\ell(x_1^{-1}x_2)\leq C_1 C'$, we must at least have $x_2 \in \Omega_J^{-C}$.

We now consider elements $x = w\varepsilon^\mu\in \Omega_J^{-C_2}$ and $x' = w'\varepsilon^{\mu'}\in \widetilde W^J$ with $\ell(x^{-1} x')\leq C_1$.

First suppose that $x\leq x'$. We find elements $x= x_1\lessdot x_2\lessdot\cdots\lessdot x_k = x'$. Note that $k = \ell(x') - \ell(x) \leq \ell(x^{-1} x')\leq C_1$. By choice of $C'$, we conclude that $\ell(x^{-1} x_i)\leq C' i\leq C'C_1$ for $i=1,\dotsc,k$. Thus $x_i \in \Omega_J^{-C}$.

By Theorem~\ref{thm:pqbgLift}, we get a path from $(w')^J$ to $w^J$ of weight $\mu-\mu'+\mathbb Z\Phi_J^\vee$. Thus \begin{align*}\wt(w_1\Rightarrow w_2)\leq \mu-\mu'\pmod{\Phi_J^\vee},\end{align*}
which is the estimate we wanted to prove.

Now suppose conversely that we are given $\mu - \wt(w'\Rightarrow w)\geq \mu'\pmod{\Phi_J^\vee}$. By Lemma~\ref{lem:pathForWeight}, we find a path $(w')^J = w_1\rightarrow w_2\rightarrow \cdots\rightarrow w_k = w^J$ in $\QB(W^J)$ of weight $\mu-\mu'+\mathbb Z\Phi_J^\vee$. Since $\mu-\mu'$ is bounded in terms of $C_1$, the length $k$ of this path is bounded in terms of $C_1$ as well. By adding another lower bound for $C_2$, we can guarantee that each such path $w_1\rightarrow\cdots\rightarrow w_k$ can indeed be lifted to $\Omega_J^{-C}$, proving that $x\leq x'$.
\end{proof}
We find working with superdominant instead superantidominant coweights a bit easier, so let us restate the lemma for $\Omega_J^C$ instead of $\Omega_J^{-C}$.

\begin{corollary}\label{cor:superdominantParabolicEmbedding}
Let $C_1>0$ be any real number. Then there exists some $C_2>0$ such that for all $x = w\varepsilon^\mu\in \Omega_J^{C_2}$ and $x' = w'\varepsilon^{\mu'}\in \widetilde {(W^J)}$ with $\ell(x^{-1} x')\leq C_1$, we have
\begin{align*}
x\leq x'\iff\mu + \wt(w\Rightarrow w')\leq \mu'\pmod{\Phi_J^\vee}.
\end{align*}
\end{corollary}
\begin{proof}
Let $w_0(J)\in W_J$ be the longest element. Let $C_2>0$ such that the conclusion of the previous Lemma is satisfied.

If $x\in \Omega_J^{C_2}$, then $x w_0(J) w_0 \in \Omega_{-w_0(J)}^{-C_2}$. Moreover, $w_0(J)w_0$ is a length positive element for $x$, so $\ell(x w_0(J) w_0) = \ell(x) + \ell(w_0(J) w_0)$.  Choosing $C_2$ appropriately, we similarly may assume $x' \in \Omega^{C}_{J}$ for some $C>0$ and obtain $\ell(x'w_0(J)w_0) = \ell(x') + \ell(w_0(J)w_0)$. Then, with the right choice of constants and using the automorphism $\alpha\mapsto -w_0\alpha$ of $\Phi$, we get
\begin{align*}
x\leq x'\iff&xw_0(J)w_0\leq x'w_0(J)w_0
\\\iff&w_0 w_0(J)\mu - \wt(w' w_0(J) w_0\Rightarrow ww_0(J) w_0)\geq w_0 w_0(J)\mu'\pmod{\Phi_{-w_0(J)}^\vee}
\\\iff&w_0(J)\mu + \wt(w_0 w'w_0(J)\Rightarrow w_0 ww_0(J))\leq w_0(J)\mu'\pmod{\Phi_J^\vee}
\\\underset{\text{\cite[Pr.\,4.3]{Lenart2015}}}\iff&w_0(J)\mu + \wt(w\Rightarrow w')\leq w_0(J)\mu'\pmod{\Phi_J^\vee}
\end{align*}
Since $w_0(J)\mu \equiv \mu\pmod {\Phi_J^\vee}$, we get the desired conclusion.
\end{proof}
\ifthesis
As an immediate consequence, we obtain a crucial estimate on the weight function.
\begin{corollary}\label{cor:weightEstimate}
Let $w\in W$ and $\alpha\in \Phi^+$. Then
\begin{align*}
\wt(ws_\alpha\Rightarrow w)\leq \Phi^+(w\alpha)\alpha^\vee.
\end{align*}
\end{corollary}
\begin{proof}
The claim is clear if $w\alpha\in \Phi^-$, as then $ws_\alpha<w$ in the Bruhat order, and we find a path from $ws_\alpha$ to $w$ consisting solely of Bruhat edges.

Now suppose that $w\alpha\in \Phi^+$. Let $\mu \in Q^\vee$ be dominant and superregular. Put $x := w\varepsilon^\mu$. Then $x(\alpha,\langle\mu,\alpha\rangle-1)\in \Phi_\af^-$, so that
\begin{align*}
w\varepsilon^\mu = x>w\varepsilon^\mu s_\alpha \varepsilon^{(\langle\mu,\alpha\rangle-1)\alpha^\vee}
=ws_\alpha \varepsilon^{\mu -\alpha^\vee}.
\end{align*}
With the superregularity constant for $\mu$ sufficiently large, we get\begin{align*}
\mu -\alpha^\vee + \wt(ws_\alpha\Rightarrow w)\leq \mu,
\end{align*}
showing the desired claim.
\end{proof}
\fi
\subsection{Computing the weight function}
We already saw in Lemma~\ref{lem:qbgDiamond} how to find for all $w_1, w_2\in W$ an element $w \in W$ such that $\wt(w_1\Rightarrow w_2) = \wt(w\Rightarrow 1)$. It remains to find a method to compute these weights. First, we note that we only need to consider quantum edges for this task.
\begin{lemma}[{\cite[Proposition~4.11]{Milicevic2020}}]\label{lem:quantumEdgesSuffice}For each $w\in W$, there is a shortest path from $w$ to $1$ in $\QB(W)$ consisting only of quantum edges.\rightqed
\end{lemma}
So we only need to find for each $w\in W\setminus\{1\}$ a quantum edge $w\rightarrow w'$ in $\QB(W)$ with $d(w'\Rightarrow 1) = d(w\Rightarrow 1)-1$. In this section, we present a new method to obtain such edges.
\begin{definition}Let $w\in W$.
\begin{enumerate}[(a)]
\item The set of \emph{inversions} of $w$ is
\begin{align*}
\inv(w) := \{\alpha\in \Phi^+\mid w^{-1}\alpha\in \Phi^-\}.
\end{align*}
\item An inversion $\gamma\in \inv(w)$ is a \emph{maximal inversion} if there is no $\alpha\in \inv(w)$ with $\alpha\neq \gamma\leq \alpha$. Here, $\gamma\leq\alpha$ means that $\alpha-\gamma$ is a sum of positive roots.

We write $\maxinv(w)$ for the set of maximal inversions of $w$.
\end{enumerate}
\end{definition}
\begin{remark}
If $\theta\in \inv(w)$ is the longest root of an irreducible component of $\Phi$, then certainly $\theta\in \maxinv(w)$. In this case, everything we want to prove is already shown in \cite[Section~5.5]{Lenart2015}. Our strategy is to follow their arguments as closely as possible while keeping the generality of maximal inversions.
\end{remark}
\begin{lemma}\label{lem:maximalEdges}
Let $w\in W$ and $\gamma\in\maxinv(w)$. Then $w\rightarrow s_\gamma w$ is a quantum edge.
\end{lemma}
\begin{proof}
Note that $s_\gamma w = ws_{-w^{-1}\gamma}$. We have to show that $-w^{-1}\gamma$ is a quantum root and that 
\begin{align*}
\ell(ws_{-w^{-1}\gamma}) = \ell(w) - \ell(s_{-w^{-1}\gamma}).
\end{align*}
\textbf{Step 1.} We show that $-w^{-1}\gamma$ is a quantum root using Lemma~\ref{lem:quantumRoots}. So pick an element $-w^{-1}\gamma\neq \beta\in \Phi^+$ with $s_{-w^{-1}\gamma}(\beta)\in \Phi^-$. We want to show that $\langle -w^{-1}\gamma^\vee,\beta\rangle=1$.

Note that
\begin{align*}
s_{-w^{-1}\gamma}(\beta) = \beta+\langle -w^{-1}\gamma^\vee,\beta\rangle w^{-1}\gamma.
\end{align*}
In particular, $k := \langle -w^{-1}\gamma^\vee,\beta\rangle>0$. It follows from the theory of root systems that
\begin{align*}
\beta_i := \beta +i w^{-1}\gamma\in \Phi,\qquad i=0,\dotsc,k.
\end{align*}
Since $\beta_0 = \beta\in \Phi^+$ and $\beta_{k} = s_{-w^{-1}\gamma}(\beta)\in \Phi^-$, we find some $i\in\{0,\dotsc,k-1\}$ with $\beta_i\in \Phi^+$ and $\beta_{i+1}\in \Phi^-$. We show that $k\leq 1$ as follows:
\begin{itemize}
\item Suppose $w\beta_i \in \Phi^+$. Then $w\beta_{i+1} = w\beta_i+\gamma > \gamma$. In particular, $w\beta_{i+1}\in \Phi^+$. We see that $w\beta_{i+1}\in\inv(w)$, contradicting maximality of $\gamma$.
\item Suppose $w\beta_{i+1}\in \Phi^-$. Then $-w\beta_i = -w\beta_{i+1}+\gamma > \gamma$. In particular, $-w\beta_{i}\in \Phi^+$. We see that $-w\beta_{i}\in\inv(w)$, contradicting maximality of $\gamma$.
\item Suppose $i\geq 1$. Then $\gamma - w\beta_i = -w\beta_{i-1}\in \Phi$. We already proved $w\beta_i\in \Phi^-$, so $-w\beta_i\in \inv(w)$. Since also $\gamma\in \inv(w)$, we conclude $\gamma<-w\beta_{i-1}\in \inv(w)$, contradicting the maximality of $\gamma$.
\item Suppose $i\leq k-2$. Then $w\beta_{i+2} = w\beta_{i+1} + \gamma \in \Phi$. Since both $\gamma$ and $w\beta_{i+1}$ are in $\inv(w)$, we conclude that $\gamma < w\beta_{i+2}\in \inv(w)$, which is a contradiction to the maximality of $\gamma$.
\end{itemize}
In summary, we conclude $0=i\geq k-1$, thus $k\leq 1$. This shows $\langle -w^{-1}\gamma^\vee,\beta\rangle=1$.

\textbf{Step 2.} We show that
\begin{align*}
\ell(ws_{-w^{-1}\gamma}) = \ell(w) - \ell(s_{-w^{-1}\gamma}).
\end{align*}
Suppose this is not the case. Then we find some $\alpha\in \Phi^+$ such that $w\alpha\in \Phi^+$ and $s_{-w^{-1}\gamma}(\alpha) \in \Phi^-$. As we saw before, $\langle -w^{-1}\gamma^\vee,\alpha\rangle=1$, so $s_{-w^{-1}\gamma}(\alpha) = \alpha + w^{-1}\gamma\in \Phi^-$. Now consider the element $ws_{-w^{-1}\gamma}(\alpha) = w\alpha + \gamma\in \Phi$. Since $w\alpha\in \Phi^+$ by assumption, we have $ws_{-w^{-1}\gamma}(\alpha)>\gamma$, in particular $ws_{-w^{-1}\gamma}(\alpha)\in \Phi^+$. We conclude $ws_{-w^{-1}\gamma}(\alpha)\in \inv(w)$, yielding a final contradiction to the maximality of $\gamma$.
\end{proof}
\begin{lemma}
Let $w\in W$ and $\alpha\in \Phi^+$ such that $w\rightarrow ws_\alpha$ is a quantum edge. Let moreover $-w\alpha\neq \gamma\in \maxinv(w)$. Then $\gamma\in \maxinv(ws_\alpha)$ and $\langle -w^{-1}\gamma^\vee,\alpha\rangle \geq 0$.
\end{lemma}
\begin{proof}
We first show $\gamma\in \inv(ws_\alpha)$, i.e.\ $s_\alpha w^{-1}\gamma\in \Phi^-$.

Aiming for a contradiction, we thus suppose that \begin{align*}
s_\alpha (-w^{-1}\gamma) = \langle \alpha^\vee,w^{-1}\gamma\rangle \alpha - w^{-1}\gamma  \in \Phi^-.
\end{align*} Then $-w^{-1}\gamma$ is a positive root whose image under $s_\alpha$ is negative. Since $\alpha$ is quantum, we conclude $\langle \alpha^\vee,-w^{-1}\gamma\rangle=1$. Thus $-\alpha - w^{-1}\gamma \in \Phi^-$. Consider the element \begin{align*}w(\alpha+w^{-1}\gamma) = \gamma+w\alpha \in \Phi.\end{align*} We distinguish the following cases:
\begin{itemize}
\item If $\gamma+w\alpha\in \Phi^-$, we get $\gamma<-w\alpha \in \inv(w)$, contradicting maximality of $\gamma$.
\item If $\gamma+w\alpha\in \Phi^+$, we compute
\begin{align*}
ws_\alpha(-w^{-1}\gamma) = -(ws_\alpha w^{-1})\gamma = -s_{w\alpha}(\gamma) = -(\gamma+w\alpha)\in \Phi^-.
\end{align*}
In other words, the positive root $-w^{-1}\gamma\in \Phi^+$ gets mapped to negative roots both by $s_\alpha$ and by $ws_\alpha\in W$. This is a contradiction to $\ell(w) = \ell(ws_\alpha) + \ell(s_\alpha)$ (since $w\rightarrow ws_\alpha$ was supposed to be a quantum edge).
\end{itemize}
In any case, we get a contradiction. Thus $\gamma\in \inv(ws_\alpha)$.

The quantum edge condition $w\rightarrow ws_\alpha$ implies $\ell(w) = \ell(ws_\alpha) + \ell(s_\alpha)$, so $\inv(ws_\alpha)\subset \inv(w)$. Because $\gamma$ is maximal in $\inv(w)$ and $\gamma \in \inv(ws_\alpha)\subseteq \inv(w)$, it follows that $\gamma$ must be maximal in $\inv(ws_\alpha)$ as well.

Finally, we have to show $\langle -w^{-1}\gamma^\vee,\alpha\rangle\geq 0$. If this was not the case, then we would get
\begin{align*}
\gamma <s_\gamma(-w\alpha) = -w\alpha + \langle w^{-1}\gamma^\vee,\alpha\rangle \gamma\in \inv(w),
\end{align*}
again contradicting maximality of $\gamma$.
\end{proof}
\begin{proposition}\label{prop:maximalEdges}
Let $w\in W$ and $\gamma\in \maxinv(w)$. Then
\begin{align*}
\wt(w\Rightarrow 1) = \wt(s_\gamma w\Rightarrow 1) - w^{-1}\gamma^\vee.
\end{align*}
\end{proposition}
\begin{proof}
Since the estimate
\begin{align*}
\wt(w\Rightarrow 1)\leq&\wt(w\Rightarrow s_\gamma w) + \wt(s_\gamma w\Rightarrow 1)
\\\leq&-w^{-1}\gamma^\vee + \wt(s_\gamma w\Rightarrow 1)
\end{align*}
follows from \weightEstimateLong, all we have to show is the inequality \enquote{$\geq$}.

For this, we use induction on $\ell(w)$. If $1\neq w\in W$, we find by Lemma~\ref{lem:quantumEdgesSuffice} some quantum edge $w\rightarrow ws_\alpha$ with $\wt(w\Rightarrow 1) = \wt(ws_\alpha\Rightarrow 1)+\alpha^\vee$. If $\alpha=-w^{-1}\gamma$, we are done.

Otherwise, $\gamma\in \maxinv(ws_\alpha)$ and $\langle -w^{-1}\gamma^\vee,\alpha\rangle\geq 0$ by the previous lemma. By induction, we have
\begin{align}
\wt(w\Rightarrow 1)=&\wt(ws_\alpha \Rightarrow 1)+\alpha^\vee
\notag\\=& \wt(s_\gamma ws_\alpha\Rightarrow 1) + \alpha^\vee - (ws_\alpha)^{-1}\gamma^\vee.\label{eq:maximalEdges1}
\end{align}
By Lemma~\ref{lem:maximalEdges}, we get the following three quantum edges:
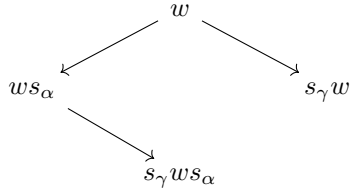
\begin{align*}
\begin{tikzcd}[ampersand replacement=\&]
\&w\ar[dl]\ar[dr]\\
ws_\alpha\ar[dr]\&\&s_\gamma w\\
\&s_\gamma ws_\alpha
\end{tikzcd}
\end{align*}
This allows for the following computation:
\begin{align}
\ell(s_\gamma ws_\alpha) =&\,\ell(ws_\alpha) + 1-\langle -(ws_\alpha)^{-1}\gamma^\vee,2\rho\rangle
\notag\\=&\,\ell(w) + 2-\langle \alpha^\vee,2\rho\rangle - \langle -w^{-1}\gamma^\vee - \langle -w^{-1}\gamma^\vee,\alpha\rangle \alpha^\vee,2\rho\rangle
\notag\\=&\,\ell(s_\gamma w) + 1 +(\langle -w^{-1}\gamma^\vee,\alpha\rangle-1)\langle \alpha^\vee,2\rho\rangle.\label{eq:maximalEdges2}
\end{align}
We now distinguish several cases depending on the value of $\langle -w^{-1}\gamma^\vee,\alpha\rangle \in\mathbb Z_{\geq 0}$.
\begin{itemize}
\item Case $\langle -w^{-1}\gamma^\vee,\alpha\rangle=0$. In this case, we get a quantum edge $s_\gamma w\rightarrow s_\gamma w s_\alpha$ by (\ref{eq:maximalEdges2}). Evaluating this in (\ref{eq:maximalEdges1}), we get
\begin{align*}
\wt(w\Rightarrow 1)=&\wt(s_\gamma w s_\alpha\Rightarrow 1)+\alpha^\vee - (ws_\alpha)^{-1}\gamma^\vee
\\\geq&\wt(s_\gamma w\Rightarrow 1) - s_\alpha w^{-1}\gamma^\vee
\\=&\wt(s_\gamma w\Rightarrow 1) - w^{-1}\gamma^\vee.
\end{align*}
\item Case $\langle -w^{-1}\gamma^\vee,\alpha\rangle=1$. In this case, we get a Bruhat edge $s_\gamma w\rightarrow s_\gamma w s_\alpha$ by (\ref{eq:maximalEdges2}). Evaluating this in (\ref{eq:maximalEdges1}), we get
\begin{align*}
\wt(w\Rightarrow 1)=&\wt(s_\gamma w s_\alpha\Rightarrow 1)+\alpha^\vee - (ws_\alpha)^{-1}\gamma^\vee
\\\geq&\wt(s_\gamma w\Rightarrow 1) + \alpha^\vee - s_\alpha w^{-1}\gamma^\vee
\\=&\wt(s_\gamma w\Rightarrow 1) - w^{-1}\gamma^\vee.
\end{align*}
\item Case $\langle -w^{-1}\gamma^\vee,\alpha\rangle\geq 2$.
We get
\begin{align*}
\ell(s_\gamma ws_\alpha) \leq& \ell(s_\gamma w) + \ell(s_\alpha) \underset{\text{L\ref{lem:quantumRoots}}}\leq \ell(s_\gamma w) + \langle \alpha^\vee,2\rho\rangle -1
\\<&\ell(s_\gamma w) + \ell(s_\alpha) \leq \ell(s_\gamma w) +1+\left(\langle -w^{-1}\gamma^\vee,\alpha\rangle-1\right) \langle \alpha^\vee,2\rho\rangle.
\end{align*}
This is a contradiction to \eqref{eq:maximalEdges2}.
\end{itemize}
In any case, we get a contradiction or the required conclusion, finishing the proof.
\end{proof}
\begin{remark}
\begin{enumerate}[(a)]
\item By Lemma~\ref{lem:weight2rho}, it follows that concatenating the quantum edge $w\rightarrow s_\gamma w$ with a shortest path $s_\gamma w\Rightarrow 1$ yields indeed a shortest path from $w$ to $1$. Thus, iterating Proposition~\ref{prop:maximalEdges}, we get a shortest path from $w$ to $1$.
\item If $w\in W^J$ and $\gamma\in\maxinv(w)$, we do not in general have a quantum edge $w\rightarrow (s_\gamma w)^J$ in $\QB(W^J)$. However, we can concatenate a shortest path from $w$ to $(s_\gamma w)^J$ (which will have weight $-w^{-1}\gamma^\vee + \mathbb Z \Phi_J^\vee$) with a shortest path from $(s_\gamma w)^J$ to $1$ in $\QB(W^J)$ to get a shortest path from $w$ to $1$.
\end{enumerate}
\end{remark}
\subsection{Semi-affine quotients}\label{sec:semi-affine-quotients}
We saw that for $w_1, w_2\in W$ and $J\subseteq \Delta$, we can assign a weight to the cosets $w_1 W_J$ and $w_2 W_J$ in $\mathbb Z\Phi^\vee / \mathbb Z\Phi_J^\vee$. In this section, we consider left cosets $W_J w$ instead. This is pretty straightforward if $J\subseteq \Delta$; however, it is more interesting if $J$ is instead allowed to be a subset of $\Delta_{\af}$. The quotient of the finite Weyl group by a set of simple affine roots will be called \emph{semi-affine} quotient.

In this section, we introduce the resulting \emph{semi-affine weight function}. This new function generalizes properties of the ordinary weight function. We have the following two motivations to study it:
\begin{itemize}
\item For root systems of type $A_n$, we can explicitly express the weight function using formula~\eqref{eq:wtAn}:
\begin{align*}
\wt(v_2\Rightarrow v_1) =  \sup_{a\in \Delta_{\af}} (v_2^{-1}\omega_a - v_1^{-1}\omega_a).
\end{align*}
Using the semi-affine weight function, we can prove a generalization of this formula, expressing the weight $\wt(v_2\Rightarrow v_1)$ as a supremum of semi-affine weights (Lemmas \ref{lem:semiAffineWeightSimplification} and \ref{lem:semiAffineWeightSupremum})
\item There is a close relationship between the quantum Bruhat graph and the Bruhat order of the extended affine Weyl group $\widetilde W$. Now \emph{Deodhar's lemma} \cite{Deodhar1977} is an important result on the Bruhat order of general Coxeter groups. Translating the statement of Deodhar's lemma to the quantum Bruhat graph yields exactly the semi-affine weight function.

Conversely, using the semi-affine weight function and Deodhar's lemma, we can generalize our result on the Bruhat order in Section~\ref{sec:bruhat-order-deodhar}.
\end{itemize}
In this \textname{}, the results of this section are only used in Section~\ref{sec:bruhat-order-deodhar}, whose results are not used later. A reader who is not interested in the aforementioned applications is thus invited to skip these two sections.
\begin{definition}
Let $J\subseteq \Delta_{\af}$ be any subset.
\begin{enumerate}[(a)]
\item We denote by $\Phi_J$ the root system generated by the roots
\begin{align*}
\cl J := \{\cl( a)\mid  a\in J\} = \{\alpha\mid (\alpha,k)\in J\}.
\end{align*}
\item We denote by $W_J$ the Weyl group of the root system $\Phi_J$, i.e.\ the subgroup of $W$ generated by $\{s_\alpha\mid \alpha\in \cl J\}$.
\item Similarly, we denote by $(\Phi_\af)_J\subseteq \Phi_J$ the (affine) root system generated by $J$, and by $\widetilde W_J$ the Coxeter subgroup of $W_\af$ generated by the reflections $r_{ a}$ with $ a\in J$.
\item We say that $J$ is a \emph{regular} subset of $\Delta_{\af}$ if no connected component of the affine Dynkin diagram of $\Phi_{\af}$ is contained in $J$, i.e.\ if $\widetilde W_J$ is finite.
\end{enumerate}
\end{definition}
\begin{lemma}\label{lem:semiAffineRootSystemFacts}
Let $J\subseteq \Delta_{\af}$ be a regular subset. 
\begin{enumerate}[(a)]
\item
$\cl J$ is a basis of $\Phi_J$. The map $(\Phi_\af)_J\rightarrow \Phi_J, (\alpha,k)\mapsto \alpha$ is bijective.
\item Writing $\Phi_J^+$ for the positive roots of $\Phi_J$ with respect to the basis $\cl J$, we get a bijection
\begin{align*}
\Phi_J^+\rightarrow (\Phi_\af)_J^+,\quad \alpha\mapsto (\alpha,\Phi^+(-\alpha)).
\end{align*}
\end{enumerate}
\end{lemma}
\begin{proof}
\begin{enumerate}[(a)]
\item Consider the Cartan matrix
\begin{align*}
C_{\alpha,\beta} := \langle \alpha^\vee,\beta\rangle,\qquad \alpha,\beta\in \cl J.
\end{align*}
This must be the Cartan matrix associated to a certain Dynkin diagram, namely the subdiagram of the affine Dynkin diagram of $\Phi_{\af}$ with set of nodes given by $J$. We know that this must coincide with the Dynkin diagram of a finite root system by regularity of $J$. Hence, $C_{\bullet,\bullet}$ is the Cartan matrix of a finite root system. Both claims follow immediately from this observation.
\item
Let $\varphi$ denote the map
\begin{align*}
\varphi : \Phi_J^+\rightarrow \Phi_\af^+,\qquad \alpha\mapsto (\alpha,\Phi^+(-\alpha)).
\end{align*}
By (a), the map is injective. For each root $\alpha\in \cl(J)$, we certainly have $\varphi(\alpha)\in \Phi_J^+$.

Now, for an inductive argument, suppose that $\alpha\in \Phi_J^+, \beta\in \cl(J)$ and $\alpha+\beta\in \Phi$ satisfy $\varphi(\alpha)\in \Phi_J^+$. We want to show that $\varphi(\alpha+\beta)\in \Phi_J^+$.

We have $(\alpha,\Phi^+(-\alpha)),(\beta,\Phi^+(-\beta))\in \Phi_J^+$, hence
\begin{align*}
(\alpha+\beta,\Phi^+(-\alpha) + \Phi^+(-\beta))\in \Phi_J^+.
\end{align*}
Hence it suffices to show that $\Phi^+(-\alpha) + \Phi^+(-\beta) = \Phi^+(-\alpha-\beta)$.

If $\beta\in \Delta$, this is clear. Hence we may assume that $\beta=-\theta$, where $\theta$ is the longest root of the irreducible component of $\Phi$ containing $\alpha,\beta$. Then $\alpha-\theta\in \Phi$ implies $\alpha\in \Phi^+$ and $\alpha-\theta\in \Phi^-$. We see that $\Phi^+(-\alpha) + \Phi^+(\theta) = \Phi^+(-\alpha+\theta)$ holds true.
\qedhere
\end{enumerate}
\end{proof}

The parabolic subgroup $\widetilde W_J\subseteq W_\af$ allows the convenient decomposition of $W_\af$ as $W_\af = \widetilde W_J \cdot \prescript{J}{}{W_\af}$ \cite[Proposition~2.4.4]{Bjorner2005}. We get something similar for $W_J\subseteq W$.
\begin{definition}
Let $J\subseteq \Delta_{\af}$.
\begin{enumerate}[(a)]
\item By $\Phi_J^+$, we denote the set of positive roots in $\Phi_J$ with respect to the basis $\cl(J)$. By abuse of notation, we also use $\Phi^+_J$ as the symbol for the indicator function of $\Phi^+_J$, i.e.
\begin{align*}
\Phi^+_J(\alpha) :=\begin{cases}1,&\alpha\in \Phi^+_J,\\
0,&\alpha\in \Phi\setminus \Phi^+_J.\end{cases}.
\end{align*}
\item We define
\begin{align*}
\prescript J{}W :=& \{w\in W\mid \forall b\in J: w^{-1}\cl(b)\in \Phi^+\}
\\=& \{w\in W\mid \forall \beta\in \Phi_J^+: w^{-1}\beta\in \Phi^+\}.
\end{align*}
\item For $w\in W$, we put \begin{align*}
\prescript J{}\ell(w) :=\#\{\beta\in \Phi^+_J\mid w^{-1}\beta \in \Phi^-\}.
\end{align*}
\end{enumerate}
\end{definition}
\begin{lemma}\label{lem:semiAffineLength}
If $w\in W$ and $\beta\in \Phi^+_J$ satisfy $w^{-1}\beta \in \Phi^-$, then
\begin{align*}
\prescript J{}\ell(s_\beta w)<\prescript J{}\ell(w).
\end{align*}
\end{lemma}
\begin{proof}
Write
\begin{align*}
I := \{\beta\neq \gamma\in \Phi^+_J\mid s_\beta(\gamma)\notin \Phi^+_J\}.
\end{align*}
Then
\begin{align*}
\prescript J{}\ell(s_\beta w) =&\, \#\{\gamma\in \Phi^+_J\mid w^{-1}s_\beta(\gamma)\in \Phi^-\}
\\=&\,\#\{\gamma\in \Phi^+_J\setminus (I\cup\{\beta\})\mid w^{-1}s_\beta(\gamma)\in \Phi^-\} + \#\{\gamma\in I\mid w^{-1}s_\beta(\gamma)\in \Phi^-\}.\\
\intertext{Since $s_\beta$ permutes the set $\Phi^+_J\setminus (I\cup\{\beta\})$, we get}
\ldots =& \,\#\{\gamma\in \Phi^+_J\setminus (I\cup\{\beta\})\mid w^{-1}\gamma\in \Phi^-\} + \#\{\gamma\in I\mid w^{-1}s_\beta(\gamma)\in \Phi^-\}.
\end{align*}
Note that if $\gamma\in I$, then $\langle \beta^\vee,\gamma\rangle >0$ and thus 
\begin{align*}
w^{-1}s_\beta(\gamma) = w^{-1}\gamma - \langle \beta^\vee,\gamma\rangle w^{-1}\beta > w^{-1}\gamma.
\end{align*}
We obtain
\begin{align*}
&\#\{\gamma\in \Phi^+_J\setminus (I\cup\{\beta\})\mid w^{-1}\gamma\in \Phi^-\} + \#\{\gamma\in I\mid w^{-1}s_\beta(\gamma)\in \Phi^-\}
\\\leq&\,\#\{\gamma\in \Phi^+_J\setminus (I\cup\{\beta\})\mid w^{-1}\gamma\in \Phi^-\} + \#\{\gamma\in I\mid w^{-1}\gamma\in \Phi^-\}
\\=&\,\prescript J{}\ell(w)-1.\qedhere
\end{align*}
\end{proof}
\begin{lemma}\label{lem:semiAffineProjection}
Let $J\subseteq \Delta_{\af}$ be a regular subset. Then there exists a uniquely determined map $\prescript J{}\pi : W\rightarrow \prescript J{}W\times \mathbb Z\Phi^\vee$ with the following two properties:
\begin{enumerate}[(1)]
\item For all $w\in \prescript J{}W$, we have $\prescript J{}\pi(w) = (w,0)$.
\item For all $w\in W$ and $\beta\in \Phi_J^+$ where we write $\prescript J{}\pi(w) = (w', \mu)$, we have
\begin{align*}
\prescript J{}\pi(s_\beta w) = (w', \mu + \Phi^+(-\beta)w^{-1}\beta^\vee)
\end{align*}
and $w\mu \in \mathbb Z \cl(J)$.
\end{enumerate}
\end{lemma}
\begin{proof}
We fix an element $\lambda\in\mathbb Z\Phi^\vee$ that is dominant and sufficiently regular (the required regularity constant follows from the remaining proof).

For $w\in W$, we consider the element $w\varepsilon^\lambda\in \widetilde W$. Then there exist uniquely determined elements $w'\varepsilon^{\lambda'}\in \prescript J{}{W_\af}$ and $y \in \widetilde W_J$ such that\begin{align*}
w\varepsilon^\lambda = y\cdot w' \varepsilon^{\lambda'}.
\end{align*}
We define $\prescript J{}\pi(w) := (w', \lambda-\lambda')$ and check that it has the required properties.
\begin{enumerate}[(1)]
\addtocounter{enumi}{-1}
\item $w'\in \prescript J{}W$: Since $\widetilde W_J$ is a finite group, we may assume that $\lambda'$ is superregular and dominant as well. For $(\alpha,k)\in J$, we have
\begin{align*}
(w'\varepsilon^{\lambda'})^{-1}(\alpha,k) = ((w')^{-1}\alpha, k+\langle \lambda', (w')^{-1} \alpha\rangle)\in \Phi^+_{\af},
\end{align*}
because $w'\varepsilon^{\lambda'}\in \prescript J{}{W_\af}$. Since $\lambda'$ is superregular and dominant, we have
\begin{align*}
((w')^{-1}\alpha, k+\langle \lambda', (w')^{-1} \alpha\rangle)\in \Phi^+_{\af}\iff (w')^{-1}\alpha\in \Phi^+.
\end{align*}
This proves $w'\in \prescript J{}W$.
\item If $w\in \prescript J{}W$, then $\prescript J{}\pi(w) = (w,0)$: The proof of (0) shows that $w\varepsilon^\lambda\in \prescript J{}{W_\af}$, so that $w\varepsilon^\lambda = w'\varepsilon^{\lambda'}$.
\item Let $w\in W$ and $\beta\in \Phi_J^+$. We have to show
\begin{align*}
\prescript J{}\pi(s_\beta w) = (w', \lambda - \lambda' + \Phi^+(-\beta)w^{-1}\beta^\vee).
\end{align*}
Put
\begin{align*}
 b := (\beta, \Phi^+(-\beta))\in \Phi_{\af}^+.
\end{align*}
By Lemma~\ref{lem:semiAffineRootSystemFacts}, we have $ b \in (\Phi_\af)_J^+$. The projection of
\begin{align*}
r_{ b}w\varepsilon^\lambda = s_\beta w \varepsilon^{\lambda + \Phi^+(-\beta)w^{-1}\beta^\vee} \in \widetilde W_J\cdot w\varepsilon^\lambda
\end{align*}
onto $\prescript J{}{W_\af}$ must again be $w'\varepsilon^{\lambda'}$. We obtain
\begin{align*}
\prescript J{}\pi(s_\beta w) = (w', \lambda + \Phi^+(-\beta)w^{-1}\beta^\vee-\lambda')
\end{align*}
as desired.

For the second claim, it suffices to observe that
\begin{align*}
\varepsilon^{w(\lambda-\lambda')} = w\varepsilon^\lambda \varepsilon^{-\lambda'} w^{-1} = y w' \varepsilon^{\lambda'} \varepsilon^{-\lambda'} w^{-1} = y \underbrace{w' w^{-1}}_{\in W_J}\in \widetilde W_J.
\end{align*}
\end{enumerate}
The fact that 
$\prescript J{}\pi$ is uniquely determined (in particular, independent of the choice of $\lambda$) can be seen as follows: If $w\in \prescript J{}W$, then $\prescript J{}\pi(w)$ is determined by (1). Otherwise, we find $\beta \in \Phi_J^+$ with $w^{-1}\beta \in \Phi^-$. We multiply $w$ on the left with $s_\beta$, and iterate this process, until we obtain an element in $\prescript J{}W$. This process will terminate after at most $\prescript J{}\ell(w)$ steps with an element in $\prescript J{}W$. Now for each of these steps, we can use property (2) to determine the value of $\prescript J{}\pi(w)$.
\end{proof}
We call the set $\prescript J{}W$ a \emph{semi-affine quotient} of $W$, as it is a quotient of a finite Weyl group by a set of affine roots. The map $\prescript J{}\pi$ is the \emph{semi-affine projection}. We now introduce the semi-affine weight function.
\begin{lemma}\label{lem:semiAffineWeight}
Let $w_1, w_2\in W$ and $J\subseteq \Delta$ be a regular subset. Write
\begin{align*}
\prescript J{}\pi(w_1) = (w_1', \mu_1),\qquad \prescript J{}\pi(w_2) = (w_2', \mu_2).
\end{align*}
Then
\begin{align*}
\wt(w_1'\Rightarrow w_2') - \mu_1+\mu_2 = \wt(w_1'\Rightarrow w_2) -\mu_1 \leq \wt(w_1\Rightarrow w_2).
\end{align*}
\end{lemma}
\begin{proof}
We first show the equation
\begin{align*}
\wt(w_1'\Rightarrow w_2')+\mu_2 = \wt(w_1'\Rightarrow w_2).
\end{align*}
Induction by $\prescript J{}\ell(w_2)$. The statement is trivial if $w_2\in \prescript J{}W$. Otherwise, we find some $\alpha\in \cl(J)$ with $w_2^{-1}\alpha \in \Phi^-$.
Because $(w_1')^{-1}\alpha\in \Phi^+$, we obtain from Lemma~\ref{lem:qbgDiamond} that
\begin{align*}
\wt(w_1'\Rightarrow w_2) =& \wt(w_1'\Rightarrow s_\alpha w_2) - \Phi^+(-\alpha)w_2^{-1}\alpha^\vee.
\end{align*}
By Lemma~\ref{lem:semiAffineProjection}, we have
\begin{align*}
\prescript J{}\pi(s_\alpha w_2) = (w'_2, \mu_2 + \Phi^+(-\alpha)w_2^{-1}\alpha^\vee).
\end{align*}
Using the inductive hypothesis, we get
\begin{align*}
\wt(w_1'\Rightarrow w_2) =& \wt(w_1'\Rightarrow s_\alpha w_2) - \Phi^+(-\alpha)w_2^{-1}\alpha^\vee
\\=&\wt(w_1'\Rightarrow w_2') + \mu_2 + \Phi^+(-\alpha)w_2^{-1}\alpha^\vee - \Phi^+(-\alpha)w_2^{-1}\alpha^\vee
\\=&\wt(w_1'\Rightarrow w_2') + \mu_2.
\end{align*}
This finishes the induction.

It remains to prove the inequality
\begin{align*}
\wt(w_1'\Rightarrow w_2)-\mu_1\leq \wt(w_1\Rightarrow w_2).
\end{align*}
The argument is entirely analogous, using \weightEstimateLong{} in place of Lemma~\ref{lem:qbgDiamond}.
\end{proof}
\begin{definition}\label{def:semi-affine-weight}
Let $w_1, w_2\in W$ and $J\subseteq \Delta_\af$ be a regular subset. We write
\begin{align*}
\prescript J{}\pi(w_1) = (w_1', \mu_1),\qquad \prescript J{}\pi(w_2) = (w_2', \mu_2).
\end{align*}
\begin{enumerate}[(a)]
\item
We define the \emph{semi-affine weight function} by
\begin{align*}
\prescript J{}\wt(w_1\Rightarrow w_2) := \wt(w_1'\Rightarrow w_2') - \mu_1+\mu_2 = \wt(w_1'\Rightarrow w_2) -\mu_1\in \mathbb Z\Phi^\vee.
\end{align*}
\item If $\beta \in \Phi_J$ and $(\beta,k)\in (\Phi_\af)_J$ is the image of $\beta$ under the bijection of Lemma~\ref{lem:semiAffineWeightProperties}, we define $\chi_J(\beta) := -k$.

If $\beta \in \Phi\setminus \Phi_J$, we define $\chi_J(\beta) := \Phi^+(\beta)$.

In other words, for $\beta\in \Phi$, we have
\begin{align*}
\chi_J(\beta) = \Phi^+(\beta) - \Phi^+_J(\beta).
\end{align*}
\end{enumerate}
\end{definition}
\begin{example}
Suppose that $\Phi$ is irreducible of type $A_2$ with basis $\alpha_1, \alpha_2$.
Let $J = \{(-\theta,1)\} = \{(-\alpha_1-\alpha_2,1)\}$, such that $\Phi_J^+ = \{-\theta\} = \{-\alpha_1-\alpha_2\}$. We want to compute $\prescript J{}\wt(1\Rightarrow s_1 s_2)$ (writing $s_i := s_{\alpha_i}$).

Observe that $\prescript J{}\pi(1) = (s_\theta,\theta^\vee)$. Hence
\begin{align*}
\prescript J{}\wt(1\Rightarrow s_1) =& \wt(s_\theta\Rightarrow s_1 s_2) - \theta^\vee
\\=&\wt(s_1 s_2 s_1\Rightarrow s_1 s_2) - \alpha_1^\vee-\alpha_2^\vee = -\alpha_2^\vee.
\end{align*}
Unlike the usual weight function, the value $\prescript J{}\wt(w_1\Rightarrow w_2)$ no longer needs to be a sum of positive coroots. In general for root systems of type $A_n$, we have
\begin{align*}
\prescript J{}{}\wt(w_1\Rightarrow w_2) = \sup_{\alpha\in \Delta_\af\setminus J} (w_1^{-1}\omega_a - w_2^{-1}\omega_a).
\end{align*}
\end{example}
\begin{lemma}\label{lem:semiAffineWeightProperties}
Let $w_1, w_2, w_3\in W$ and let $J\subseteq \Delta$ be a regular subset.
\begin{enumerate}[(a)]
\item The semi-affine weight function satisfies the triangle inequality,
\begin{align*}
\prescript J{}\wt(w_1\Rightarrow w_3)\leq \prescript J{}\wt(w_1\Rightarrow w_2) + \prescript J{}\wt(w_2\Rightarrow w_3).
\end{align*}
\item If $\alpha\in \Phi_J$, we have
\begin{align*}
\prescript J{}\wt(s_\alpha w_1\Rightarrow w_2)=& \prescript J{}\wt(w_1\Rightarrow w_2) + \chi_J(\alpha)w_1^{-1}\alpha^\vee,\\
\prescript J{}\wt(w_1\Rightarrow s_\alpha w_2)=& \prescript J{}\wt(w_1\Rightarrow w_2) - \chi_J(\alpha)w_2^{-1}\alpha^\vee.
\end{align*}
\item If $\beta\in \Phi^+$, we have
\begin{align*}
\prescript J{}\wt(w_1s_\beta\Rightarrow w_2)\leq& \prescript J{}\wt(w_1\Rightarrow w_2) + \chi_J(w_1\beta)\beta^\vee,\\
\prescript J{}\wt(w_1\Rightarrow w_2s_\beta)\leq& \prescript J{}\wt(w_1\Rightarrow w_2) + \chi_J(-w_2\beta)\beta^\vee.
\end{align*}
\end{enumerate}
\end{lemma}
\begin{proof}
Part (a) follows readily from the definition. Let us prove part (b). We focus on the first identity, as the proof of the second identity is analogous.

Up to replacing $\alpha$ by $-\alpha$, which does not change the reflection $s_\alpha$ nor the value of
\begin{align*}
\chi_J(\alpha)w_1^{-1}\alpha^\vee,
\end{align*}
we may assume $\alpha\in \Phi_J^+$. Now write
\begin{align*}
\prescript J{}\pi(w_1) = (w_1', \mu_1),\qquad \prescript J{}\pi(w_2) = (w_2', \mu_2).
\end{align*}
Then $\prescript J{}\pi(s_\alpha w_1) = (w_1', \mu_1 + \Phi^+(-\alpha)w_1^{-1}\alpha^\vee)$. Thus
\begin{align*}
\prescript J{}\wt(s_\alpha w_1\Rightarrow w_2) =& \wt(w_1'\Rightarrow w_2') - \mu_1 - \Phi^+(-\alpha)w_1^{-1}\alpha^\vee + \mu_2
\\=&\prescript J{}\wt(w_1\Rightarrow w_2) - \Phi^+(-\alpha)w_1^{-1}\alpha^\vee
\\=& \prescript J{}\wt(w_1\Rightarrow w_2) +\chi_J(\alpha)w_1^{-1}\alpha^\vee
\end{align*}
as $\alpha\in \Phi_J^+$.

Now we prove part (c). Again, we only show the first inequality. If $w_1\beta \in \Phi_J$, the inequality follows from part (b). Otherwise, we use (a) and \weightEstimateLong to compute
\begin{align*}
\prescript J{}\wt(w_1s_\beta\Rightarrow w_2)\leq~\,& \prescript J{}\wt(w_1s_\alpha\Rightarrow w_1) + \prescript J{}\wt(w_1\Rightarrow w_2)
\\\underset{\text{L\ref{lem:semiAffineWeight}}}\leq&\wt(w_1s_\alpha\Rightarrow w_1) + \prescript J{}\wt(w_1\Rightarrow w_2)
\\\leq~\,&\Phi^+(w\alpha)\alpha^\vee + \prescript J{}\wt(w_1\Rightarrow w_2)
\\=~\,&\chi_J(w\alpha)\alpha^\vee + \prescript J{}\wt(w_1\Rightarrow w_2).
\end{align*}
This finishes the proof.
\end{proof}
\begin{lemma}\label{lem:semiAffineWeightSimplification}
Let $w_1, w_2\in W$ and $J\subseteq \Delta_{\af}$ be regular. Suppose that for all $\alpha\in \Phi^+_J$, at least one of the following conditions is satisfied:
\begin{align*}w_1^{-1}\alpha\in \Phi^+\text{ or }w_2^{-1}\alpha\in \Phi^-.\end{align*}
Then $\prescript J{}\wt(w_1\Rightarrow w_2) = \wt(w_1\Rightarrow w_2)$.
\end{lemma}
\begin{proof}

We show the claim via induction on $\prescript J{}\ell(w_1)$. If $w_1\in \prescript J{}W$, then the claim follows from Lemma~\ref{lem:semiAffineWeight}.

Otherwise, we find some $\alpha \in \cl(J)$ with $w_1^{-1}\alpha \in \Phi^-$. By assumption, also $w_2^{-1}\alpha \in \Phi^-$. Using Lemma~\ref{lem:qbgDiamond}, we get
\begin{align*}
\wt(w_1\Rightarrow w_2)=&\wt(s_\alpha w_1\Rightarrow s_\alpha w_2) + \chi_J(\alpha)w_1^{-1}\alpha^\vee - \chi_J(\alpha)w_2^{-1}\alpha^\vee.
%\\=&\wt(s_\alpha w_1\Rightarrow w_2) + \chi_J(\alpha) w_1^{-1}\alpha^\vee.
\end{align*}
Since $\prescript J{}\ell(s_\alpha w_1)<\prescript J{}\ell(w_1)$ by Lemma~\ref{lem:semiAffineLength}, we want to show that $(s_\alpha w_1,s_\alpha w_2)$ also satisfy the condition stated in the lemma.

For this, let $\beta\in \Phi_J^+$. If $\beta = \alpha$, then $(s_\alpha w_1)^{-1}\alpha = -w_1^{-1}\alpha\in \Phi^+$ by choice of $\alpha$. Now assume that $\beta\neq\alpha$, so that $s_\alpha\beta \in \Phi^+_J$. By the assumption on $w_1$ and $w_2$, we must have $w_1^{-1}s_\alpha(\beta)\in \Phi^+$ or $w_2^{-1} s_\alpha(\beta)\in \Phi^-$. In other words, we have
\begin{align*}
(s_\alpha w_1)^{-1}\beta\in \Phi^+\text{ or }(s_\alpha w_2)^{-1}\beta \in \Phi^-.
\end{align*}
This shows that $(s_\alpha w_1, s_\alpha w_2)$ satisfy the desired properties.

By the inductive hypothesis and Lemma~\ref{lem:semiAffineWeightProperties}, we get
\begin{align*}
&\wt(s_\alpha w_1\Rightarrow s_\alpha w_2) + \chi_J(\alpha) w_1^{-1}\alpha^\vee  - \chi_J(\alpha)w_2^{-1}\alpha^\vee\\=&
\prescript J{}\wt(s_\alpha w_1\Rightarrow s_\alpha w_2) + \chi_J(\alpha) w_1^{-1}\alpha^\vee  - \chi_J(\alpha)w_2^{-1}\alpha^\vee
\\=&\prescript J{}\wt(w_1\Rightarrow w_2).
\end{align*}
This completes the induction and the proof.
\end{proof}
% !TeX spellcheck = en_GB
\section{Bruhat order}\label{chap:bruhat-order}
The Bruhat order on $\widetilde W$ is a fundamental Coxeter-theoretic notion that has been studied with great interest, e.g.\ \cite{Bjorner1995, Kottwitz2000, Rapoport2002, Lenart2015}. In this section, we present new characterizations of the Bruhat order on $\widetilde W$.

The structure of this section is as follows: In Section~\ref{sec:bruhat-order-statement}, we state our main criterion for the Bruhat order as Theorem~\ref{thm:bruhat2} and discuss some of its applications. We then prove this criterion in Section~\ref{sec:bruhat-order-proof}. Finally, Section~\ref{sec:bruhat-order-deodhar} will cover some consequences of Deodhar's lemma (cf.\ \cite{Deodhar1977}) and feature an even more general criterion.
\subsection{A criterion}\label{sec:bruhat-order-statement}
\begin{definition}
Let $x = w\varepsilon^\mu\in \widetilde W$. A \emph{Bruhat-deciding datum} for $x$ is a tuple $(v, J_1,\dotsc,J_m)$ where $v\in W$ and $J_\bullet$ is a finite collection of arbitrary subsets $J_1,\dotsc,J_m\subseteq \Delta$ with $m\geq 1$, satisfying the following two properties:
\begin{enumerate}[(1)]
\item The element $v$ is length positive for $x$, i.e.\ $\ell(x,v\alpha)\geq 0$ for all $\alpha\in \Phi^+$.
\item Writing $J := J_1\cap\cdots\cap J_m$, we have $\ell(x,v\alpha)=0$ for all $\alpha\in \Phi_J$.
\end{enumerate}
\end{definition}
The name \emph{Bruhat-deciding} is justified by the following result.
\begin{theorem}\label{thm:bruhat2}
Let $x = w\varepsilon^\mu, x' = w'\varepsilon^{\mu'}\in \widetilde W$. Fix a Bruhat-deciding datum $(v, J_1,\dotsc,J_m)$ for $x$. Then the following are equivalent:
\begin{enumerate}[(1)]
\item $x\leq x'$.
\item For all $i=1,\dotsc,m$, there exists an element $v'_i\in W$ such that
\begin{align*}
v^{-1}\mu + \wt(v'_i\Rightarrow v) + \wt(wv\Rightarrow w'v'_i)\leq (v'_i)^{-1}\mu'\pmod{\Phi_{J_i}^\vee}.
\end{align*}
\end{enumerate}
\end{theorem}
We again use the shorthand notation $\mu_1\leq \mu_2\pmod{\Phi^\vee_J}$ for $\mu_1-\mu_2+\mathbb Z\Phi^\vee_J\leq 0+\mathbb Z\Phi^\vee_J$ in $\mathbb Z\Phi^\vee/\mathbb Z\Phi^\vee_J$.

This theorem is the main result of this section. We give a proof in Section~\ref{sec:bruhat-order-proof}.

First, let us remark that the construction of a Bruhat-deciding datum is easy. It suffices to choose any length positive element $v$ for $x$, and then $(v,\emptyset)$ is Bruhat-deciding.

The inequality of Theorem~\ref{thm:bruhat2} is only interesting for $v\in \LP(x)$ and $v_i'\in \LP(x')$, as explained by the following lemma in conjunction with \crossRef{Lemma~}{lem:rootFunctionalAdjustment}.
\begin{lemma}\label{lem:criterionAdjustment}
Let $x = w\varepsilon^\mu, x' = w'\varepsilon^{\mu'}\in \widetilde W$. Suppose we are given elements $v, v'\in W$, a subset $J\subseteq \Delta$ and a positive root $\alpha\in \Phi^+$.
\begin{enumerate}[(a)]
\item Assume $\ell(x,v\alpha)<0$. Then the inequality
\begin{align*}
(vs_\alpha)^{-1}\mu + \wt(v'\Rightarrow vs_\alpha) + \wt(wvs_\alpha\Rightarrow w'v')\leq (v')^{-1}\mu'\pmod{\Phi_J^\vee}
\end{align*}
implies
\begin{align*}
v^{-1}\mu + \wt(v'\Rightarrow v) + \wt(wv\Rightarrow w'v')\leq (v')^{-1}\mu'\pmod{\Phi_J^\vee}.
\end{align*}
\item Assume $\ell(x',v\alpha)<0$. Then the inequality
\begin{align*}
v^{-1}\mu + \wt(v'\Rightarrow v) + \wt(wv\Rightarrow w'v')\leq (v')^{-1}\mu'\pmod{\Phi_J^\vee}
\end{align*}
implies
\begin{align*}
v^{-1}\mu + \wt(v's_\alpha\Rightarrow v) + \wt(wv\Rightarrow w'v's_\alpha)\leq (v's_\alpha)^{-1}\mu'\pmod{\Phi_J^\vee}.
\end{align*}
\end{enumerate}
\end{lemma}
\begin{proof}
\begin{enumerate}[(a)]
\item We have
\begin{align*}
(v')^{-1}\mu'\geq&(vs_\alpha)^{-1}\mu + \wt(v'\Rightarrow vs_\alpha) + \wt(wvs_\alpha\Rightarrow w'v')
\\\geq&v^{-1}\mu - \langle v^{-1}\mu,\alpha\rangle \alpha^\vee + \wt(v'\Rightarrow v) - \wt(vs_\alpha\Rightarrow v) \\&+ \wt(wv\Rightarrow w'v') - \wt(wv\Rightarrow wvs_\alpha)
\\\underset{(\ast)}\geq&v^{-1}\mu - \langle v^{-1}\mu,\alpha\rangle \alpha^\vee + \wt(v'\Rightarrow v) - \Phi^+(v\alpha)\alpha^\vee\\&+ \wt(wv\Rightarrow w'v') - \Phi^+(-wv\alpha)\alpha^\vee
\\=&v^{-1}\mu + \wt(v'\Rightarrow v) + \wt(wv\Rightarrow w'v') - (\ell(x,v\alpha)+1)\alpha^\vee
\\\geq&v^{-1}\mu + \wt(v'\Rightarrow v) + \wt(wv\Rightarrow w'v')\pmod{\Phi_J^\vee}.
\end{align*}
The inequality $(\ast)$ is \weightEstimateLong.
\item The calculation is completely analogous.\qedhere
\end{enumerate}
\end{proof}

\begin{proof}[Proof of Theorem~\ref{thm:bruhat1} using Theorem~\ref{thm:bruhat2}]
We use the notation of Theorem~\ref{thm:bruhat1}. In view of Lemma~\ref{lem:criterionAdjustment} and \crossRef{Lemma~}{lem:rootFunctionalAdjustment}, the condition
\begin{align*}
\exists v_2\in W:~
v_1^{-1}\mu_1 + \wt(v_2\Rightarrow v_1) + \wt(w_1 v_1\Rightarrow w_2 v_2)\leq v_2^{-1}\mu_2\tag{$\ast$}
\end{align*}
is true for all $v_1\in \LP(x)$ iff it is true for all $v_1\in W$. We see that asking condition $(\ast)$ for all $v_1\in W$ is equivalent to asking condition (2) of Theorem~\ref{thm:bruhat2} \emph{for each Bruhat-deciding datum}. In this sense, Theorem~\ref{thm:bruhat2} implies Theorem~\ref{thm:bruhat1}.
\end{proof}

If $x'$ is in a shrunken Weyl chamber, there is a canonical choice for $v'$.
\begin{corollary}\label{cor:bruhatOrderShrunken}
Let $x = w\varepsilon^\mu$ and $x' = w'\varepsilon^{\mu'}$. Assume that $x'$ is in a shrunken Weyl chamber and that $v'$ is the length positive element for $x'$. Pick any length positive element $v$ for $x$. Then $x\leq x'$ if and only if
\begin{align*}
v^{-1}\mu + \wt(v'\Rightarrow v) + \wt(wv\Rightarrow w'v')\leq (v')^{-1}\mu'.
\end{align*}
\end{corollary}
\begin{proof}
$(v,\emptyset)$ is a Bruhat-deciding datum for $x$. By Lemma~\ref{lem:criterionAdjustment} and \crossRef{Corollary~}{cor:rootFunctionalAdjustments}, the inequality in Theorem~\ref{thm:bruhat2} (2) is satisfied by \emph{some} $v'\in W$ iff it is satisfied by the unique length positive element $v'$ for $x'$.
\end{proof}
We now show how Theorem~\ref{thm:bruhat2} can be used to describe Bruhat covers in $\widetilde W$. The following proposition generalizes the previous results of Lam-Shimozono \cite[Proposition~4.1]{Lam2010} and Mili\'cevi\'c \cite[Proposition~4.2]{Milicevic2021}.
\begin{proposition}\label{prop:bruhatCovers}
Let $x = w\varepsilon^\mu, x' = w'\varepsilon^{\mu'}\in \widetilde W$ and $v\in \LP(x)$. Then the following are equivalent:
\begin{enumerate}[(a)]
\item $x\lessdot x'$, i.e.\ $x < x'$ and $\ell(x) = \ell(x')-1$.
\item There exists some $v'\in \LP(x')$ such that
\begin{itemize}
\item[(b.1)] $v^{-1}\mu + \wt(v'\Rightarrow v) + \wt(wv\Rightarrow w'v') = (v')^{-1}\mu'$ and
\item[(b.2)] $d(v'\Rightarrow v) + d(wv\Rightarrow w'v')=1$.
\end{itemize}
\item There is a root $\alpha\in \Phi^+$ satisfying at least one of the following conditions:
\begin{itemize}
\item[(c.1)] There exists a Bruhat edge $v' := s_\alpha v\rightarrow v$ in $\QB(W)$ with $x' = xs_\alpha$ and $v'\in \LP(x')$.
\item[(c.2)] There exists a quantum edge $v' := s_\alpha v\rightarrow v$ in $\QB(W)$ with $v^{-1}\alpha\in \Phi^+, x' = xr_{(-\alpha,1)}$ and $v'\in \LP(x')$.
\item[(c.3)] There exists a Bruhat edge $wv\rightarrow s_\alpha wv$ in $\QB(W)$ such that $x' = s_\alpha x$ and $v\in \LP(x')$.
\item[(c.4)] There exists a quantum edge $wv\rightarrow s_\alpha wv$ in $\QB(W)$ with $(wv)^{-1}\alpha\in \Phi^-$, $x' = r_{(-\alpha,1)} x$ and $v\in \LP(x')$.
\end{itemize}
\item There exists a root $\alpha\in \Phi^+$ satisfying at least one of the following conditions:
\begin{itemize}
\item[(d.1)] We have $w' = ws_\alpha, \mu' = s_\alpha(\mu), \ell(s_\alpha v) = \ell(v)-1$ and for all $\beta\in \Phi^+$:
\begin{align*}
\ell(x,v\beta) + \Phi^+(s_\alpha v\beta) - \Phi^+(v\beta)\geq 0.
\end{align*}
\item[(d.2)] We have $w' = ws_\alpha, \mu' = s_\alpha(\mu)-\alpha^\vee, \ell(s_\alpha v) = \ell(v)-1+\langle v^{-1}\alpha^\vee,2\rho\rangle$ and for all $\beta\in \Phi^+$:
\begin{align*}
\ell(x,v\beta) + \langle \alpha^\vee,v\beta\rangle + \Phi^+(s_\alpha v\beta) - \Phi^+(v\beta)\geq 0.
\end{align*}
\item[(d.3)] We have $w' = s_\alpha w, \mu' = \mu, \ell(s_\alpha wv) = \ell(wv)+1$ and for all $\beta\in \Phi^+$:
\begin{align*}
\ell(x,v\beta) + \Phi^+(wv\beta) - \Phi^+(s_\alpha wv\beta)\geq 0.
\end{align*}
\item[(d.4)] We have $w' = s_\alpha w, \mu' = \mu-w^{-1}\alpha^\vee, \ell(s_\alpha wv) = \ell(wv)+1+\langle (wv)^{-1}\alpha^\vee,2\rho\rangle$ and for all $\beta\in \Phi^+$:
\begin{align*}
\ell(x,v\beta) + \langle \alpha^\vee,wv\beta\rangle + \Phi^+(wv\beta) - \Phi^+(s_\alpha wv\beta)\geq 0.
\end{align*}
\end{itemize}
\end{enumerate}
\end{proposition}
\begin{proof}
(a) $\iff$ (b): We start with a key calculation for $v'\in \LP(x')$:
\begin{align*}
&\langle (v')^{-1}\mu' - \wt(v'\Rightarrow v) - \wt(wv\Rightarrow w'v') - v^{-1}\mu,2\rho\rangle
\\\underset{\text{L\ref{lem:weight2rho}}}=&\langle (v')^{-1}\mu,2\rho\rangle - d(v'\Rightarrow v) -\ell(v') + \ell(v) \\&\qquad-d(wv\Rightarrow w'v') - \ell(wv) + \ell(w'v') - \langle v^{-1}\mu,2\rho\rangle
\\\underset{\text{\positiveLengthFormulaShort}}=&\ell(x') - \ell(x) - d(v'\Rightarrow v) - d(wv\Rightarrow w'v').
\end{align*}
First assume that (a) holds, i.e.\ $x\lessdot x'$. By Theorem~\ref{thm:bruhat2} and Lemma~\ref{lem:criterionAdjustment}, we find $v'\in \LP(x')$ such that
\begin{align*}
(v')^{-1}\mu' - \wt(v'\Rightarrow v) - \wt(wv\Rightarrow w'v') - v^{-1}\mu\geq 0
\end{align*}
By the above key calculation, we see that
\begin{align*}
\ell(x') \geq \ell(x) + d(v'\Rightarrow v) + d(wv\Rightarrow w'v'),
\end{align*}
where equality holds if and only if (b.1) is satisfied. Note that $x\lessdot x'$ implies that $x^{-1}x'$ must be an affine reflection, thus $w\neq w'$. We see that $v\neq v'$ or $wv\neq w'v'$, thus in particular
\begin{align*}
\ell(x)+1= \ell(x') \geq \ell(x) + d(v'\Rightarrow v) + d(wv\Rightarrow w'v')\geq \ell(x)+1.
\end{align*}
Since equality must hold, we get (b.1) and (b.2).

Now assume conversely that (b) holds. By (b.1) and Theorem~\ref{thm:bruhat2}, we see that $x<x'$. Now using the key calculation and (b.2), we get $\ell(x') =\ell(x)+1$.

(b) $\iff$ (c): The condition (b.2) means that either $v=v'$ and $wv\rightarrow w'v'$ is an edge in $\QB(W)$, or $wv = w'v'$ and $v'\rightarrow v$ is an edge. If we now distinguish between Bruhat and quantum edges, we get the explicit conditions of (c) (or (d)).

Let us first assume that (b) holds. We distinguish the following cases:
\begin{enumerate}[(1)]
\item $wv = w'v'$ and $v'\rightarrow v$ is a Bruhat edge: Then we can write $v' = s_\alpha v$ for some $\alpha\in \Phi^+$ with $v^{-1}\alpha\in \Phi^-$. Now the condition $wv = w'v'$ implies $w' = ws_\alpha$. Condition (b.1) implies $v^{-1}\mu = (v')^{-1}\mu'$, so $\mu' = s_\alpha(\mu)$. We get (c.1).
\item $wv = w'v'$ and $v'\rightarrow v$ is a quantum edge: Then we can write $v' = s_\alpha v$ for some $\alpha\in \Phi^+$ with $v^{-1}\alpha\in \Phi^+$. Now the condition $wv = w'v'$ implies $w' = ws_\alpha$. Condition (b.1) implies $v^{-1}\mu+v^{-1}\alpha^\vee = (v')^{-1}\mu'$, so $\mu' = s_\alpha(\mu)-\alpha^\vee$. We get (c.2).
\item $v = v'$ and $wv\rightarrow w'v'$ is a Bruhat edge: Then we can write $w'v' = s_\alpha wv$ for some $\alpha\in \Phi^+$ with $(wv)^{-1}\alpha\in \Phi^-$. Now the condition $v=v'$ implies $w' = s_\alpha w$. Condition (b.1) implies $v^{-1}\mu = (v')^{-1}\mu$, so $\mu' = \mu$. We get (c.3).
\item $v=v'$ and $wv\rightarrow w'v'$ is a quantum edge: Then we can write $w'v' = s_\alpha wv$ for some $\alpha\in \Phi^+$ with $(wv)^{-1}\alpha\in \Phi^-$. Now the condition $v=v'$ implies $w' = s_\alpha w$. Condition (b.1) implies $v^{-1}\mu-(wv)^{-1}\alpha^\vee = (v')^{-1}\mu$, so $\mu' = \mu-w^{-1}\alpha^\vee$. We get (c.4).
\end{enumerate}

Reversing the calculations above shows that (c) $\implies$ (b).

For (c) $\iff$ (d), we just explicitly rewrite the conditions for length positivity of $v'$, and the definition of edges in the quantum Bruhat graph.
\end{proof}
\begin{remark}
If the translation part $\mu$ of $x=w\varepsilon^\mu$ is sufficiently regular, the estimates for the length function of $x$ in part (d) of Proposition~\ref{prop:bruhatCovers} are trivially satisfied. Writing $\LP(x) = \{v\}$, we get a one-to-one correspondence
\begin{align*}
\{\text{Bruhat covers of }x\}\leftrightarrow \{\text{edges }?\rightarrow v\}\sqcup\{\text{edges }wv\rightarrow ?\}.
\end{align*}
\end{remark}
We obtain the following useful technical observation from Proposition~\ref{prop:bruhatCovers}:
\begin{corollary}\label{cor:simpleCoverLPFix}
Let $x \in \widetilde W$, $v\in \LP(x)$ and $(\alpha,k)\in \Delta_\af$ with $\ell(x,\alpha)=0$. If $v^{-1}\alpha\in \Phi^+$, then $s_\alpha v\in \LP(x)$.
\end{corollary}
\begin{proof}
Since $x(\alpha,k)\in \Phi^+$ by \crossRef{Lemma~}{lem:lengthFunctionalAsCountingAffineRoots}, we have $x<xr_a$. Since $a$ is a simple affine root, we must have $x\lessdot xr_a$.  So one of the four possibilities (c.1) -- (c.4) of Proposition~\ref{prop:bruhatCovers} must be satisfied.

If (c.3) or (c.4) are satisfied, we get $v\in \LP(x')$. Since $x' = x r_a$ is a length additive product, \crossRef{Lemma~}{lem:lengthAdditivity} shows $s_\alpha v\in \LP(x)$, finishing the proof.

Now assume that (c.1) is satisfied. Then $x' = xs_\beta$ for some $\beta\in \Phi^+$ means $k=0$ and $\alpha=\beta$. Now $v^{-1}\alpha\in \Phi^+$ means that $\ell(s_\alpha v)> \ell(v)$, so $s_\alpha v\rightarrow v$ cannot be a Bruhat edge.

Finally assume that (c.2) is satisfied. Then $x' = xr_{(-\beta,1)}$ for some $\beta\in \Phi^+$ means that $k=1$ and $\alpha=-\beta \in \Phi^-$. Then $s_\alpha v\rightarrow v$ cannot be a quantum edge, as $\ell(s_\alpha v)<\ell(v)$.

We get the desired claim or a contradiction, finishing the proof.
\end{proof}
As a second application, we discuss the semi-infinite order on $\widetilde W$ as introduced by Lusztig \cite{Lusztig1980}.  It plays a role for certain constructions related to the affine Hecke algebra, cf.\ \cite{Lusztig1980, Naito2017}.
\begin{definition}\label{def:semiInfiniteOrder}
Let $x = w\varepsilon^\mu\in \widetilde W$.
\begin{enumerate}[(a)]
\item We define the \emph{semi-infinite length} of $x$ as
\begin{align*}
\ell^{\frac\infty 2}(x) := \ell(w) +\langle \mu,2\rho\rangle.
\end{align*}
\item We define the \emph{semi-infinite order} on $\widetilde W$ to be the order $<^{\frac \infty 2}$ generated by the relations
\begin{align*}
\forall x\in \widetilde W, a\in \Phi_\af:~x<^{\frac \infty 2}xr_a\text{ if }\ell^{\frac\infty 2}(x)\leq \ell^{\frac\infty 2}(xr_a).
\end{align*}
\end{enumerate}
\end{definition}
We have the following link between the semi-infinite order and the Bruhat order:
\begin{proposition}[{\cite[Proposition~2.2.2]{Naito2017}}]\label{prop:semiInfiniteOrderVsBruhatOrder}
Let $x_1, x_2\in \widetilde W$. There exists a number $C>0$ such that for all $\lambda \in \mathbb Z\Phi^\vee$ satisfying the regularity condition $\langle \lambda,\alpha\rangle > C$ for every positive root $\alpha$,
we have
\begin{align*}
x_1\leq^{\frac\infty 2}x_2\iff x_1\varepsilon^\lambda \leq x_2\varepsilon^\lambda.
\end{align*}
\end{proposition}
\begin{corollary}\label{cor:semiInfiniteOrder}
Let $x_1 = w_1\varepsilon^{\mu_1}, x_2 = w_2\varepsilon^{\mu_2}\in \widetilde W$. Then $x_1\leq^{\frac\infty 2}x_2$ if and only if
\begin{align*}
\mu_1 + \wt(w_1\Rightarrow w_2)\leq \mu_2.
\end{align*}
\end{corollary}
\begin{proof}
Let $\lambda$ be as in Proposition~\ref{prop:semiInfiniteOrderVsBruhatOrder}. Choosing $\lambda$ sufficiently large, we may assume that $x_1\varepsilon^{\lambda}$ and $x_2\varepsilon^{\lambda}$ are superregular with $\LP(x_1\varepsilon^\lambda) = \LP(x_2\varepsilon^\lambda)=\{1\}$. Now $x_1\varepsilon^\lambda\leq x_2\varepsilon^\lambda$ if and only if
\begin{align*}
\mu_1 + \wt(w_1\Rightarrow w_2)\leq \mu_2,
\end{align*}
by Corollary~\ref{cor:bruhatOrderShrunken}.
\end{proof}
We finish this section with another application of our Theorem~\ref{thm:bruhat2}, namely a discussion of admissible and permissible sets in $\widetilde W$, as introduced by Kottwitz and Rapoport \cite{Kottwitz2000}.
\begin{definition}
Let $x = w\varepsilon^\mu\in \widetilde W$ and $\lambda\in \Xast$ a dominant coweight.
\begin{enumerate}[(a)]
\item We say that $x$ lies in the \emph{admissible} set defined by $\lambda$, denoted $x\in \Adm(\lambda)$, if there exists $u\in W$ such that $x\leq \varepsilon^{u\lambda}$ with respect to the Bruhat order on $\widetilde W$.
\item The \emph{fundamental coweight} associated with $a =(\alpha,k)\in \Delta_\af$ is the uniquely determined element $\omega_a\in \mathbb Q\Phi^\vee$ such that for each $\beta\in \Delta$,
\begin{align*}
\langle \omega_a,\beta\rangle = \begin{cases}1,&a = (\beta,0),\\
0,&a\neq (\beta,0).\end{cases}
\end{align*}
In particular, $\omega_a=0$ iff $k\neq 0$.
\item Let $a = (\alpha,k)\in \Delta_\af$, and denote by $\theta\in \Phi^+$ the longest root of the irreducible component of $\Phi$ containing $\alpha$. The \emph{normalized coweight} associated with $a$ is
\begin{align*}
\widetilde \omega_a = \begin{cases}0,&k\neq 0,\\
\frac{1}{\langle \omega_a,\theta\rangle}\omega_a,&k=0.\end{cases}
\end{align*}
\item We say that $x$ lies in the \emph{permissible} set defined by $\lambda$, denoted $x\in \Perm(\lambda)$, if $\mu\equiv\lambda\pmod{\Phi^\vee}$ and for every simple affine root $a\in \Delta_\af$, we have
\begin{align*}
(\mu + \widetilde\omega_a - w^{-1}\widetilde\omega_a)^{\dom}\leq \lambda\text{ in }\Xast\otimes\mathbb Q.
\end{align*}
\end{enumerate}
\end{definition}
It is shown in \cite{Kottwitz2000} that the admissible set is always contained in the permissible set and that equality holds for the groups $\GL_n$ and $\GSp_{2n}$ if $\lambda$ is \emph{minuscule} (i.e.\ a fundamental coweight of some special node). It is a result of Haines and Ng\^o \cite{Haines2002} that $\Adm(\lambda)\neq \Perm(\lambda)$ in general. We show how the latter result can be recovered using our methods.
\begin{proposition}[{Cf.\ \cite[Prop.~3.3]{He2021d}}]\label{prop:admissibleSubset}
Let $x = w\varepsilon^\mu\in \widetilde W$ and $\lambda\in \Xast$ a dominant coweight. Then the following are equivalent:
\begin{enumerate}[(1)]
\item $x\in \Adm(\lambda)$.
\item For all $v\in W$, we have
\begin{align*}
v^{-1}\mu + \wt(wv\Rightarrow v)\leq \lambda.
\end{align*}
\item For some $v\in \LP(x)$, we have
\begin{align*}
v^{-1}\mu + \wt(wv\Rightarrow v)\leq \lambda.
\end{align*}
\end{enumerate}
\end{proposition}
\begin{proof}
(1) $\implies$ (2): Suppose that $x\in \Adm(\lambda)$, so $x\leq \varepsilon^{u\lambda}$ for some $u\in W$. Let also $v\in W$. By Lemma~\ref{lem:bruhat1}, we find $\tilde u\in W$ such that
\begin{align*}
v^{-1}\mu + \wt(\tilde u\Rightarrow v) + \wt(wv\Rightarrow \tilde u)\leq \tilde u^{-1}u\lambda.
\end{align*}
Thus
\begin{align*}
v^{-1}\mu+\wt(wv\Rightarrow v)\leq&v^{-1}\mu + \wt(\tilde u\Rightarrow v) + \wt(wv\Rightarrow \tilde u)\\\leq &\tilde u^{-1}u\lambda
\\\leq&(\tilde u^{-1}u\lambda)^\dom = \lambda.
\end{align*}
Since (2) $\implies$ (3) is trivial, it remains to show (3) $\implies$ (1). So let $v\in \LP(x)$ satisfy $v^{-1}\mu + \wt(wv\Rightarrow v)\leq \lambda$. By Theorem~\ref{thm:bruhat2}, we immediately get $x\leq \varepsilon^{v\lambda}$, showing (1).
\end{proof}
\begin{lemma}\label{lem:permissibleSubset}
Let $x = w\varepsilon^\mu\in \widetilde W$ and $\lambda\in \Xast$ a dominant coweight. Then the following are equivalent:
\begin{enumerate}[(1)]
\item $x\in \Perm(\lambda)$.
\item For all $v\in W$, we have
\begin{align*}
v^{-1}\mu + \sup_{a\in \Delta_\af}\left( v^{-1}\widetilde\omega_a - (wv)^{-1}\widetilde\omega_a\right)\leq \lambda.
\end{align*}
\end{enumerate}
If moreover $x$ lies in a shrunken Weyl chamber, the conditions are equivalent to
\begin{itemize}
\item[(3)] For the uniquely determined $v\in \LP(x)$, we have
\begin{align*}
v^{-1}\mu + \sup_{a\in \Delta_\af} \left(v^{-1}\widetilde\omega_a - (wv)^{-1}\widetilde\omega_a\right)\leq \lambda.
\end{align*}
\end{itemize}
\end{lemma}
\begin{proof}
We have
\begin{align*}
\text{(1)}\iff&\forall a\in \Delta_\af:~\left(\mu +\widetilde\omega_a - w^{-1}\widetilde\omega_a\right)^{\dom}\leq\lambda
\\\iff&\forall a\in \Delta_\af, v\in W:~v^{-1}\left(\mu + \widetilde\omega_a - w^{-1}\widetilde\omega_a\right)\leq\lambda
\\\iff&\forall v\in W:\sup_{a\in \Delta_\af}v^{-1}\left(\mu + \widetilde\omega_a- w^{-1}\widetilde\omega_a\right)\leq\lambda
\\\iff&\text{(2)}.
\end{align*}
Now assume that $x$ is in a shrunken Weyl chamber, $\LP(x) = \{v\}$ and $a\in \Delta_\af$. We claim that
\begin{align*}
\left(\mu + \widetilde\omega_a - w^{-1}\widetilde\omega_a\right)^{\dom} = v^{-1}\left(\mu + \widetilde\omega_a - w^{-1}\widetilde\omega_a\right).
\end{align*}
Once this claim is proved, the equivalence (1) $\iff$ (3) follows.

It remains to show that $v^{-1}\left(\mu + \widetilde\omega_a - w^{-1}\widetilde\omega_a\right)$ is dominant. Hence let $\alpha\in\Phi^+$. We obtain
\begin{align*}
\left\langle v^{-1}\left(\mu + \widetilde\omega_a - w^{-1}\widetilde\omega_a\right),\alpha\right\rangle=&\langle \mu,v\alpha\rangle + \langle \widetilde\omega_a,v\alpha\rangle - \langle \widetilde\omega_a,wv\alpha\rangle
\\\geq&\langle \mu,v\alpha\rangle - \Phi^+(-v\alpha) - \Phi^+(wv\alpha)
\\=&\ell(x,v\alpha)-1\geq 0.\qedhere
\end{align*}
\end{proof}
\begin{corollary}\label{cor:admissibleVsPermissible}
For any fixed root system $\Phi$, the following are equivalent:
\begin{enumerate}[(1)]
\item For all dominant $\lambda\in \Xast$, we get the equality $\Adm(\lambda) = \Perm(\lambda)$.
\item For all $w_1, w_2\in W$, the element
\begin{align*}
\left\lceil \sup_{a\in \Delta_\af}w_2^{-1}\widetilde\omega_a-w_1^{-1}\widetilde\omega_a\right\rceil :=\min\{z\in \mathbb Z\Phi^\vee \mid z\geq \sup_{a\in \Delta_\af}w_2^{-1}\widetilde\omega_a-w_1^{-1}\widetilde\omega_a\text{ in }\mathbb Q\Phi^\vee\}
\end{align*}
agrees with $\wt(w_1\Rightarrow w_2)$.
\item Each irreducible component of $\Phi$ is of type $A_n$ $(n\geq 1)$, $B_2$, $C_3$ or $G_2$.
\end{enumerate}
\end{corollary}
\begin{proof}
(1) $\implies$ (2): Comparing condition (3) of Proposition~\ref{prop:admissibleSubset} with condition (3) of Lemma~\ref{lem:permissibleSubset} for superregular elements $x\in \widetilde W$ yields the desired claim.

(2) $\implies$ (1): We can directly compare condition (2) of Proposition~\ref{prop:admissibleSubset} with condition (2) of Lemma~\ref{lem:permissibleSubset}.

(2) $\iff$ (3): Call an irreducible root system $\Phi'$ \emph{good} if condition (2) is satisfied for $\Phi'$, and \emph{bad} otherwise. Certainly, $\Phi$ is good iff each irreducible component of $\Phi$ is good. Moreover, root systems of type $A_n$ are good, we saw this in formula \eqref{eq:wtAn}.

If $\Phi_J\subseteq \Phi$ is bad for some $J\subseteq \Delta$, then certainly $\Phi$ is bad as well.
It remains to show that root systems of types $C_3$ and $G_2$ are good, and that root systems of types $B_3, C_4$ and $D_4$ are bad. Each of these claims is easily verified using the Sagemath computer algebra system \cite{sagemath, sage-combinat}.
\end{proof}
For irreducible root systems of rank $\geq 4$, the equivalence (1) $\iff$ (3) is due to \cite{Haines2002}.
\subsection{Proof of the criterion}\label{sec:bruhat-order-proof}
The goal of this section is to prove Theorem~\ref{thm:bruhat2}. We start with the direction (1) $\implies$ (2), which is the easier one.

\begin{lemma}\label{lem:bruhat1}
Let $x = w\varepsilon^\mu, x' =w'\varepsilon^{\mu'}\in \widetilde W$ and $v\in W$. If $x\leq x'$, then there exists an element $v'\in W$ such that
\begin{align*}
v^{-1}\mu + \wt(v'\Rightarrow v) + \wt(wv\Rightarrow w'v')\leq (v')^{-1}\mu.
\end{align*}
\end{lemma}
\begin{proof}
First note that the relation
\begin{align*}
x\preceq x':\iff \forall v\exists v':~ v^{-1}\mu + \wt(v'\Rightarrow v) + \wt(wv\Rightarrow w'v')\leq (v')^{-1}\mu
\end{align*}
is transitive. Thus, it suffices to show the implication $x\leq x'\implies x\preceq x'$ for generators $(x, x')$ of the Bruhat order.

In other words, we may assume that $x' = xr_{\mathbf a}$ for an affine root $\mathbf a = (\alpha,k)\in \Phi_\af^+$ with
\begin{align*}
x\mathbf a = (w\alpha,k-\langle \mu,\alpha\rangle)\in \Phi_\af^+.
\end{align*}
This means that $w' = ws_\alpha$ and $\mu' = \mu + (k-\langle \mu,\alpha\rangle)\alpha^\vee$, where $k-\langle \mu,\alpha\rangle\geq \Phi^+(-w\alpha)$.
We now do a case distinction depending on whether the root $v^{-1}\alpha$ is positive or negative.

\textbf{Case $v^{-1}\alpha\in \Phi^-$.} Put $v' = s_\alpha v$ such that $wv = w'v'$. Then using \weightEstimateLong,
\begin{align*}
&v^{-1}\mu + \wt(v'\Rightarrow v) + \wt(wv\Rightarrow w'v')
\\=\,&v^{-1}\mu + \wt(v s_{-v^{-1}\alpha}\Rightarrow v) + 0
\\\leq\,&v^{-1}\mu -\Phi^+(-\alpha)v^{-1}\alpha^\vee
\\\leq\,&v^{-1}\mu - kv^{-1}\alpha^\vee
\\=\,&(s_\alpha v)^{-1}(s_\alpha(\mu) + k\alpha^\vee) = (v')^{-1}\mu'.
\end{align*}
\textbf{Case $v^{-1}\alpha\in \Phi^+$.} Put $v' = v$ such that $w'v' = wvs_{v^{-1}\alpha}$. Then using \weightEstimateLong,
\begin{align*}
&v^{-1}\mu + \wt(v'\Rightarrow v) + \wt(wv\Rightarrow w'v')
\\=\,&v^{-1}\mu + \wt(wv\Rightarrow wvs_{v^{-1}\alpha})
\\\leq\,&v^{-1}\mu + \Phi^+(-w\alpha)v^{-1}\alpha^\vee
\\\leq\,&v^{-1}\mu + (k-\langle \mu,\alpha\rangle)\alpha^\vee=(v')^{-1}\mu'.
\end{align*}
This finishes the proof.
\end{proof}
The direction (1) $\implies$ (2) of Theorem~\ref{thm:bruhat2} follows directly from this lemma. We now start the journey to prove (2) $\implies$ (1).
\begin{lemma}\label{lem:bruhat-criterion1}
Let $x = w\varepsilon^\mu, x'=w'\varepsilon^{\mu'}\in \widetilde W$, and suppose that $(1, J_1,\dotsc,J_m)$ is a Bruhat-deciding datum for both $x$ and $x'$. If the inequality
\begin{align*}
\mu + \wt(w\Rightarrow w')\leq \mu'\pmod{\Phi_{J_i}^\vee}
\end{align*}
holds for $i=1,\dotsc,m$, then $x\leq x'$.
\end{lemma}
\begin{proof}
Let $J = J_1\cap \cdots \cap J_m$. Then we get
\begin{align*}
\mu + \wt(w\Rightarrow w')\leq \mu'\pmod{\Phi_J^\vee}.
\end{align*}
Let $C_1 := \ell(x^{-1} x')$ and pick $C_2>0$ such that the conclusion of Corollary~\ref{cor:superdominantParabolicEmbedding} holds true. We can find an element $\lambda \in \mathbb Z\Phi^\vee$ such that $\langle \lambda,\alpha\rangle=0$ for all $\alpha\in J$ and
\begin{align*}
\langle \lambda,\alpha\rangle\geq C_2
\end{align*}
for all $\alpha \in \Phi^+\setminus \Phi_J$. Since $1\in W$ is length positive for both $x$ and $x'$, it follows from \crossRef{Lemma~}{lem:lengthAdditivity} that
\begin{align*}
\ell(x\varepsilon^\lambda) = \ell(x) + \ell(\varepsilon^\lambda),\qquad \ell(x'\varepsilon^\lambda) = \ell(x') + \ell(\varepsilon^\lambda).
\end{align*}
So it suffices to show $x\varepsilon^\lambda\leq x'\varepsilon^\lambda$. Note that $x\varepsilon^\lambda, x'\varepsilon^\lambda \in \Omega_J^{C_2}$ by choice of $\lambda$. Moreover, we have
\begin{align*}
\mu + \lambda + \wt(w\Rightarrow w')\leq \mu'+\lambda\pmod{\Phi_J^\vee}
\end{align*}
by assumption. Therefore, the inequality $x\varepsilon^\lambda\leq x'\varepsilon^\lambda$ follows from Corollary~\ref{cor:superdominantParabolicEmbedding}.
\end{proof}
\begin{lemma}\label{lem:bruhat-criterion2}
Let $x = w\varepsilon^\mu, x'=w'\varepsilon^{\mu'}\in \widetilde W$, and suppose that $(1, J_1,\dotsc,J_m)$ is a Bruhat-deciding datum for $x$. If the inequality
\begin{align*}
\mu + \wt(w\Rightarrow w')\leq \mu'\pmod{\Phi_{J_i}^\vee}
\end{align*}
holds for $i=1,\dotsc,m$, then $x\leq x'$.
\end{lemma}
\begin{proof}
Induction on $\ell(x')$.

If $(1,J,\dotsc,J_m)$ is also Bruhat-deciding for $x'$, we are done by Lemma~\ref{lem:bruhat-criterion1}. Otherwise, we must have that $1\in W$ is not length positive for $x'$, or that $J := J_1\cap\cdots\cap J_m$ allows some $\alpha\in \Phi_J$ with $\ell(x',\alpha)\neq 0$.

First consider the case that $1\in W$ is not length positive for $x'$. Then we find a positive root $\alpha\in \Phi^+$ with $\ell(x',\alpha)<0$. Hence $a := (-\alpha,1)\in \Phi_\af^+$ with $x' a \in \Phi^-$, so that
\begin{align*}
x'' := w''\varepsilon^{\mu''}:= x'r_{a} = w's_\alpha \varepsilon^{\mu'- (1+\langle\mu',\alpha\rangle)\alpha^\vee}<x'.
\end{align*}
We calculate
\begin{align*}
\mu + \wt(w\Rightarrow w'')\leq& \mu + \wt(w\Rightarrow w') + \wt(w'\Rightarrow w's_\alpha)
\\\leq&\mu' + \Phi^+(-w'\alpha)\alpha^\vee
\\=&\mu'- (1+\langle\mu',\alpha\rangle)\alpha^\vee + (\langle \mu',\alpha\rangle +1+\Phi^+(-w'\alpha))\alpha^\vee
\\=&\mu'' + (\ell(x',\alpha)+1)\alpha^\vee\leq \mu''\pmod{\Phi_J^\vee}.
\end{align*}
By induction, $x\leq x''$. Since $x''<x'$, we conclude $x<x'$ and are done.

Next consider the case that $1\in W$ is indeed length positive for $x'$, but we find some $\alpha\in \Phi_J$ with $\ell(x',\alpha)\neq 0$. We may assume $\alpha\in \Phi^+$, and then $\ell(x',\alpha)>0$ by length positivity. Then $a = (\alpha,0)\in \Phi_\af^+$ with $x' a\in \Phi^-$. We conclude that
\begin{align*}
x'' := w''\varepsilon^{\mu''}:= x'r_{ a} = w's_\alpha \varepsilon^{\mu' -\langle \mu',\alpha\rangle\alpha^\vee}<x'.
\end{align*}
We calculate
\begin{align*}
\mu + \wt(w\Rightarrow w'')\leq&\mu + \wt(w\Rightarrow w') + \wt(w'\Rightarrow w's_\alpha)
\\\leq&\mu' + \Phi^+(-w'\alpha)\alpha^\vee
\\=&\mu'' + (\Phi^+(-w'\alpha)+\langle \mu',\alpha\rangle)\alpha^\vee
\\\equiv&\mu''\pmod{\Phi_J^\vee},
\end{align*}
as $\alpha^\vee\in \Phi_J^\vee$. So as in the previous case, we get $x\leq x''<x'$ and are done.

This completes the induction and the proof.
\end{proof}
Before we can continue the series of incremental generalizations, we need a technical lemma.
\begin{lemma}\label{lem:criterionSlightImprovement}
Let $x = w\varepsilon^\mu, x'=w'\varepsilon^{\mu'}\in \widetilde W$. Let $J\subseteq \Delta$ and $v'\in W$ be given such that
\begin{align*}
\mu + \wt(v'\Rightarrow 1) + \wt(w\Rightarrow w'v')\leq (v')^{-1}\mu'\pmod{\Phi_J^\vee}.
\end{align*}
Then there exists an element $v''\in W$ satisfying the same inequality as $v'$ above, and satisfying moreover the condition $\ell(x',\gamma)<0$ for all $\gamma\in \maxinv(v'')$.
\end{lemma}
\begin{proof}
Among all $v'\in W$ satisfying the inequality
\begin{align*}
\mu + \wt(v'\Rightarrow 1) + \wt(w\Rightarrow w'v')\leq (v')^{-1}\mu'\pmod{\Phi_J^\vee},
\end{align*}
pick one of minimal length in $W$. We prove that $\ell(x',\gamma)<0$ for all $\gamma\in \maxinv(v')$.

Suppose that this was not the case, so $\ell(x',\gamma)\geq 0$ for some $\gamma\in \maxinv(v')$. The condition $\gamma\in \inv(v')$ implies $\ell(s_\gamma v')<\ell(v')$. Moreover, $\wt(v'\Rightarrow 1) = \wt(s_\gamma v'\Rightarrow 1)-(v')^{-1}\gamma^\vee$ by Proposition~\ref{prop:maximalEdges}. We calculate
\begin{align*}
\mu + \wt(s_\gamma v'\Rightarrow 1)& + \wt(w\Rightarrow w's_\gamma v')
\\=& 
\mu + \wt(v'\Rightarrow 1) + (v')^{-1}\gamma^\vee + \wt(w\Rightarrow w's_\gamma v')\\\leq&
\mu + \wt(v'\Rightarrow 1) + (v')^{-1}\gamma^\vee + \wt(w\Rightarrow w'v') + \wt(w'v'\Rightarrow w's_\gamma v')
\\\leq&(v')^{-1}\mu' + (v')^{-1}\gamma^\vee + \wt(w' v' \Rightarrow w' s_\gamma v')
\\=&(v')^{-1}\mu' + (v')^{-1}\gamma^\vee + \wt(w' s_\gamma v' s_{-(v')^{-1}(\gamma)}\Rightarrow w' s_\gamma v')
\\\leq&(v')^{-1}\mu' + (v')^{-1}\gamma^\vee - \Phi^+(w'\gamma) (v')^{-1}\gamma^\vee
\\=&(s_\gamma v')^{-1}\mu' + \langle \mu',\gamma\rangle (v')^{-1}\gamma^\vee + (v')^{-1}\gamma^\vee - \Phi^+(w'\gamma)(v')^{-1}\gamma^\vee
\\=&(s_\gamma v')^{-1}\mu' + \ell(x', \gamma)(v')^{-1}\gamma^\vee \leq (s_\gamma v')^{-1}\mu'\pmod{\Phi_J^\vee}.
\end{align*}
This is a contradiction to the choice of $v'$, so we get the desired claim.
\end{proof}
\begin{lemma}\label{lem:bruhat-criterion3}
Let $x = w\varepsilon^\mu, x'=w'\varepsilon^{\mu'}\in \widetilde W$, and suppose that $(1, J_1,\dotsc,J_m)$ is a Bruhat-deciding datum for $x$. If for each $i=1,\dotsc,m$, there exists some $v'_i\in W$ with
\begin{align*}
\mu + \wt(v'_i\Rightarrow 1) + \wt(w\Rightarrow w'v'_i)\leq (v'_i)^{-1}\mu'\pmod{\Phi_{J_i}^\vee},
\end{align*}then $x\leq x'$.
\end{lemma}
\begin{proof}
Induction on $\ell(x')$.

By Lemma~\ref{lem:criterionSlightImprovement}, we may assume that for each $i\in\{1,\dotsc,m\}$ and $\gamma\in \maxinv(v_i')$, we have $\ell(x',\gamma)<0$.

If $1\in W$ is length positive for $x'$, i.e.\ $\ell(x',\alpha)\geq 0$ for all $\alpha\in \Phi^+$, then we get $\maxinv(v'_i)=\emptyset$ for all $i=1,\dotsc,m$, i.e.\ $v_i'=1$. Now the claim follows from Lemma~\ref{lem:bruhat-criterion2}.

Thus suppose that the set
\begin{align*}
\{\alpha\in \Phi^+\mid \ell(x',\alpha)<0\}
\end{align*}
is non-empty. We fix a root $\alpha$ that is maximal within this set. Now $\mathbf a = (-\alpha,1)\in \Phi_{\af}^+$ satisfies $x'\mathbf a \in \Phi_\af^-$, as $\ell(x',\alpha)<0$. Consider
\begin{align*}
x'' := w''\varepsilon^{\mu''}:= x'r_{\mathbf a} = w's_\alpha \varepsilon^{\mu'- (1+\langle\mu',\alpha\rangle)\alpha^\vee}<x'.
\end{align*}
We want to show $x\leq x''$ using the inductive assumption. So pick an index $i\in\{1,\dotsc,m\}$. We do a case distinction based on whether the root $(v_i')^{-1}\alpha$ is positive or negative.

\textbf{Case $(v_i')^{-1}\alpha\in \Phi^-$.} Then $\alpha\in \inv(v_i')$, so there exists some $\gamma\in \maxinv(v_i')$ with $\alpha\leq \gamma$. By choice of $v_i'$, we get $\ell(x',\gamma)<0$. By maximality of $\alpha$ and $\alpha\leq\gamma$, we get $\alpha=\gamma$. In other words, $\alpha\in \maxinv(v_i')$.

Define $v_i'' := s_\alpha v_i'$. 
Then by Proposition~\ref{prop:maximalEdges}, $\wt(v_i'\Rightarrow 1) = \wt(v_i''\Rightarrow 1)-(v_i')^{-1}\alpha^\vee$. We compute
\begin{align*}
&\mu + \wt(v''_i\Rightarrow 1) + \wt(w\Rightarrow w'' v''_i) \\=\,&
\mu + \wt(v_i'\Rightarrow 1) + (v_i')^{-1}\alpha^\vee + \wt(w\Rightarrow w'v'_i)
\\\leq\,&(v'_i)^{-1}\mu' + (v'_i)^{-1}\alpha^\vee
\\=\,&(s_\alpha v'_i)^{-1}(\mu' - (1+\langle \mu',\alpha\rangle)\alpha^\vee) = ( v''_i)^{-1}\mu''\pmod{\Phi_{J_i}^\vee}.
\end{align*}

\textbf{Case $(v_i')^{-1}\alpha\in \Phi^+$.} We define $v_i'' := v_i'$ and use \weightEstimateLong to compute
\begin{align*}
&\mu + \wt(v''_i\Rightarrow 1) + \wt(w\Rightarrow w''v''_i) \\\leq\,&\mu + \wt(v'_i\Rightarrow 1) + \wt(w\Rightarrow w'v'_i) +\wt(w'v'_i\Rightarrow w'v'_is_{(v_i')^{-1}\alpha})
\\\leq\,&(v'_i)^{-1}\mu' + \Phi^+(-w'\alpha)(v_i')^{-1}\alpha^\vee
\\=\,&(v'_i)^{-1}(\mu' - (1+\langle \mu', \alpha\rangle)\alpha^\vee) + (\langle \mu',\alpha\rangle + 1 + \Phi^+(-w'\alpha))(v_i')^{-1}\alpha^\vee
\\=\,&(v'_i)^{-1}\mu'' + (\ell(x',\alpha)+1)(v_i')^{-1}\alpha^\vee \leq (v''_i)^{-1}\mu''\pmod{\Phi_{J_i}^\vee}.
\end{align*}

In any case, we get the desired inequality
\begin{align*}
\mu + \wt(v''_i\Rightarrow 1) + \wt(w\Rightarrow w''v''_i)\leq (v_i'')^{-1}\mu''\pmod{\Phi_{J_i}^\vee}.
\end{align*}
By induction, $x\leq x''<x'$, completing the induction and the proof.
\end{proof}
\begin{lemma}\label{lem:bruhat-criterion4}
Let $x = w\varepsilon^\mu, x'=w'\varepsilon^{\mu'}\in \widetilde W$, and suppose that $(v, J_1,\dotsc,J_m)$ is a Bruhat-deciding datum for $x$. If for each $i=1,\dotsc,m$, there exists some $v'_i\in W$ with
\begin{align*}
v^{-1}\mu + \wt(v'_i\Rightarrow v) + \wt(wv\Rightarrow w'v'_i)\leq (v'_i)^{-1}\mu'\pmod{\Phi_{J_i}^\vee},
\end{align*}then $x\leq x'$.
\end{lemma}
\begin{proof}
Induction on $\ell(v)$. If $v=1$, this follows from Lemma~\ref{lem:bruhat-criterion3}.

Let $J := J_1\cap\cdots\cap J_m$. If $\alpha\in J$, then $vs_\alpha$ trivially satisfies the same condition as $v$. So we may assume that $v\in W^J$.

Since $v\neq 1$, we find a simple root $\alpha\in \Delta$ with $v^{-1}\alpha \in \Phi^-$. In particular, $\ell(x,\alpha)\leq 0$, such that $x<xs_\alpha$.

We claim that $(s_\alpha v, J_1,\dotsc,J_m)$ is a Bruhat-deciding datum for $xs_\alpha$.
Indeed, for $\beta\in \Phi$, we use \crossRef{Lemma~}{lem:lengthFunctionalForProducts} to compute
\begin{align*}
\ell(xs_\alpha, s_\alpha v\beta) =&
\ell(x,v\beta) + \ell(s_\alpha,s_\alpha v\beta)
\\=~\,&\ell(x,v\beta)+\begin{cases}1,&v\beta = -\alpha,\\-1,&v\beta=\alpha,\\0,&v\beta\neq \pm \alpha.\end{cases}
\end{align*}
If $\beta\in \Phi^+$, the condition $v^{-1}\alpha\in \Phi^-$ forces $v\beta\neq\alpha$, showing
\begin{align*}
\ell(xs_\alpha, s_\alpha v\beta) \geq \ell(x,v\beta)\geq 0.
\end{align*}
Now consider the case $\beta\in \Phi_J^+$. Then $\ell(x,v\beta)=0$ by assumption. Moreover, $v\beta\in \Phi^+$ as $v\in W^J$, so that $v\beta\neq -\alpha$. We conclude $\ell(xs_\alpha,s_\alpha v\beta)=\ell(x,v\beta)=0$ in this case.

This shows that $(s_\alpha v,J_1,\dotsc,J_m)$ is Bruhat-deciding for $xs_\alpha$. Since $\ell(s_\alpha v)<\ell(v)$, we may apply the inductive hypothesis to $xs_\alpha$ to prove $xs_\alpha\leq \max(x', x's_\alpha)$. We distinguish two cases. 

\textbf{Case $\ell(x',\alpha)\leq 0$.} This means $x'<x's_\alpha$, so we wish to prove $xs_\alpha < x's_\alpha$, using the inductive hypothesis. So let $i\in\{1,\dotsc,m\}$. By Lemma~\ref{lem:criterionAdjustment}, we may assume that $v_i'$ is length positive for $x'$.

First assume that $(v_i')^{-1}\alpha \in \Phi^-$. By Lemma~\ref{lem:qbgDiamond}, we get
\begin{align*}
\wt(v_i'\Rightarrow v) = \wt(s_\alpha v_i'\Rightarrow s_\alpha v).
\end{align*}
Define $v_i'':= s_\alpha v_i'$. Then
\begin{align*}
&(s_\alpha v)^{-1}(s_\alpha \mu) + \wt(v_i''\Rightarrow s_\alpha v) + \wt(ws_\alpha s_\alpha v\Rightarrow w's_\alpha v_i'')
\\=\,&v^{-1}\mu + \wt(v_i'\Rightarrow v) + \wt(wv\Rightarrow w' v_i')
\\\leq\,&(v_i')^{-1}\mu' = (v_i'')(s_\alpha\mu')\pmod{\Phi_{J_i}^\vee}.
\end{align*}

Next, assume that $(v_i')^{-1}\alpha \in \Phi^+$. By length positivity, we must have $\ell(x',\alpha)=0$. By Lemma~\ref{lem:qbgDiamond}, we get
\begin{align*}
\wt(v_i'\Rightarrow v) = \wt(v_i'\Rightarrow s_\alpha v).
\end{align*}
Define $v_i'' := v_i'$. Then using \weightEstimateLong,
\begin{align*}
&(s_\alpha v)^{-1}(s_\alpha \mu) + \wt(v_i''\Rightarrow s_\alpha v) + \wt(ws_\alpha s_\alpha v\Rightarrow w's_\alpha v_i'')
\\=\,&v^{-1}\mu + \wt(v_i'\Rightarrow v) + \wt(wv\Rightarrow w's_\alpha v_i')
\\\leq\,&v^{-1}\mu + \wt(v_i'\Rightarrow v) + \wt(wv\Rightarrow w' v_i') + \wt(w'v_i'\Rightarrow w'v_i's_{(v_i')^{-1}\alpha})
\\\leq\,&(v_i')^{-1}\mu' + \Phi^+(-w'\alpha) (v_i')^{-1}\alpha
\\=\,&(v_i')^{-1}s_\alpha \mu' + (\langle \mu',\alpha\rangle + \Phi^+(-w'\alpha))(v_i')^{-1}\alpha
\\=\,&(v_i'')^{-1}s_\alpha\mu' + \ell(x',\alpha)(v_i')^{-1}\alpha = (v_i'')^{-1}s_\alpha\mu.
\pmod{\Phi_{J_i}^\vee}.
\end{align*}
We see that the inequality
\begin{align*}
(s_\alpha v)^{-1}(s_\alpha \mu) + \wt(v_i''\Rightarrow s_\alpha v) + \wt(ws_\alpha s_\alpha v\Rightarrow w's_\alpha v_i'')\leq (v_i'')^{-1}s_\alpha\mu
\pmod{\Phi_{J_i}^\vee}
\end{align*}
always holds, proving $xs_\alpha\leq x's_\alpha$. Since $s_\alpha$ is a simple reflection in $\widetilde W$, $x<xs_\alpha$ and $x'<x's_\alpha$, we conclude that $x\leq x'$ must hold as well.

\textbf{Case $\ell(x',\alpha)>0$.} We now wish to show $xs_\alpha\leq x'$, as $x'>x's_\alpha$. We prove this using the inductive assumption, so let $i\in\{1,\dotsc,m\}$. As in the previous case, we assume that $v_i'$ is length positive for $x'$. In particular, $(v_i')^{-1}\alpha \in \Phi^+$.

By Lemma~\ref{lem:qbgDiamond}, we get
\begin{align*}
\wt(v_i'\Rightarrow v) = \wt(v_i'\Rightarrow s_\alpha v).
\end{align*}
Define $v_i'' := v_i'$. Then
\begin{align*}
&(s_\alpha v)^{-1}(s_\alpha \mu) + \wt(v_i''\Rightarrow s_\alpha v) + \wt(w s_\alpha s_\alpha  v\Rightarrow w'v_i'')
\\=&v^{-1}\mu + \wt(v_i'\Rightarrow v) + \wt(wv\Rightarrow w'v'_i)
\\\leq&(v_i')^{-1}\mu' = (v_i'')^{-1}\mu'.
\end{align*}
By the inductive assumption, we get $xs_\alpha \leq x'$. Thus $x<xs_\alpha \leq x'$.

This completes the induction and the proof.
\end{proof}
\begin{proof}[Proof of Theorem~\ref{thm:bruhat2}]
The implication (1) $\Rightarrow$ (2) follows from Lemma~\ref{lem:bruhat1}.

The implication (2) $\Rightarrow$ (1) follows from Lemma~\ref{lem:bruhat-criterion4}.
\end{proof}

\subsection{Deodhar's lemma}\label{sec:bruhat-order-deodhar}
In this section, we apply Deodhar's lemma \cite{Deodhar1977} to our Theorem~\ref{thm:bruhat2}. We need the semi-affine weight functions and related notions as introduced in Section~\ref{sec:semi-affine-quotients}. We moreover need a two-sided version of Deodhar's lemma, which seems to be well-known for experts, yet our standard reference \cite[Theorem~2.6.1]{Bjorner2005} only provides a one-sided version. We thus introduce the two-sided theory briefly. For convenience, we state it for the extended affine Weyl group $\widetilde W$, even though it holds true in a more general Coxeter theoretic context.
\begin{definition}
Let $L, R\subseteq \Phi_\af$ be any sets of affine roots (we will mostly be interested in sets of simple affine roots).
\begin{enumerate}[(a)]
\item By $\widetilde W_L$, we denote the subgroup of $\widetilde W$ generated by the affine reflections $r_a$ for $a\in L$.
\item We define \begin{align*}
\prescript L{}{}{\widetilde W}{}^R := \{x\in\widetilde W: x^{-1}L\subseteq\Phi_\af^+\text{ and }xR \subseteq \Phi_\af^+\}.
\end{align*}
\end{enumerate}
\end{definition}
Recall that we called a subset $L\subseteq \Delta_\af$ \emph{regular} if $\widetilde W_L$ is finite.
\begin{proposition}\label{prop:deodhar}
Let $x,y\in \widetilde W$ and $L,R\subseteq \Delta_\af$ be regular.
\begin{enumerate}[(a)]
\item The double coset $\widetilde W_L x \widetilde W_R$ contains a unique element of minimal length, denoted $\prescript L{}{}x{}^R$, and a unique element of maximal length, denoted $\prescript {-L}{}{}x{}^{-R}$. We have
\begin{align*}
\prescript L{}{}\widetilde W{}^R \cap \left(\widetilde W_L x \widetilde W_R\right) =& \left\{\prescript L{}{}x{}^R\right\},\\
\prescript {-L}{}{}\widetilde W{}^{-R} \cap \left(\widetilde W_L x \widetilde W_R\right) =& \left\{\prescript {-L}{}{}x{}^{-R}\right\}.
\end{align*}
\item We have 
\begin{align*}
\prescript L{}{}x{}^R\leq x \leq 
\prescript {-L}{}{}x{}^{-R}
\end{align*}
in the Bruhat order, and there exist (non-unique) elements $x_L, x_L' \in \widetilde W_L$ and $x_R, x_R'\in \widetilde W_R$ such that
\begin{align*}
x = x_L \cdot\prescript L{}{}x{}^R\cdot x_R&\text{ and }\ell(x) = \ell(x_L) + \ell\left(\prescript L{}{}x{}^R\right) + \ell(x_R),\\
\prescript {-L}{}{}x{}^{-R} = x_L'\cdot x \cdot x_R'&\text{ and }\ell\left(\prescript {-L}{}{}x{}^{-R}\right) = \ell(x_L') + \ell\left(x\right) + \ell(x_R').
\end{align*}
\item If $x\leq y$, then
\begin{align*}
\prescript L{}{}x{}^R\leq \prescript L{}{}y{}^R\text{ and }
\prescript {-L}{}{}x{}^{-R}\leq \prescript {-L}{}{}y{}^{-R}.
\end{align*}
\item Suppose $L_1, \dotsc, L_\ell,R_1,\dotsc,R_r\subseteq \Delta_\af$ are regular subsets such that $L = L_1\cap\cdots\cap L_\ell$ and $R = R_1\cap\cdots\cap R_r$. Then
\begin{align*}
\prescript L{}{}x{}^R\leq \prescript L{}{}y{}^R\iff\forall i,j:~
\prescript {L_i}{}{}x{}^{R_j}\leq \prescript {L_i}{}{}y{}^{R_j}.
\end{align*}
\end{enumerate}
\end{proposition}
\begin{proof}
\begin{enumerate}[(a)]
\item We only show the claim for $\prescript L{}{}x{}^R$, as the proof for $\prescript {-L}{}{}x{}^{-R}$ is analogous.

Let $x_1\in \widetilde W_Lx\widetilde W_R$ an element of minimal length. It is clear that each such element must lie in $\prescript L{}{}\widetilde W{}^R$.

Let now $x_0\in\prescript L{}{}\widetilde W{}^R\cap\left(\widetilde W_L x \widetilde W_R\right)$ be any element. It suffices to show that $x_0 = x_1$.

Since $x_1\in \widetilde W_L x_0\widetilde W_R$, we find $x_L\in \widetilde W_L, x_R\in \widetilde W_R$ such that $x_1 = x_L x_0 x_R$. We show $x_1 = x_0$ via induction on $\ell(x_L)$. If $x_L=1$, the claim is evident.

As $x_0 \in \prescript L{}{}\widetilde W{}^R$ and $x_R \in \widetilde W_R$, it follows that $\ell(x_0 x_R) = \ell(x_0) + \ell(x_R)$, cf.\ \crossRef{Lemma~}{lem:lengthAdditivity} or \cite[Proposition~2.4.4]{Bjorner2005}. Now
\begin{align*}
\ell(x_0) \geq \ell(x_1) = \ell(x_L x_0 x_R) \geq \ell(x_0 x_R) - \ell(x_L) = \ell(x_0) + \ell(x_R) - \ell(x_L).
\end{align*}
We conclude that $\ell(x_L)\geq \ell(x_R)$. By an analogous argument, we get $\ell(x_L)\leq \ell(x_R)$, such that $\ell(x_L) = \ell(x_R)$. It follows that \begin{align*}
\ell(x_0) = \ell(x_1) = \ell(x_L x_0 x_R) = \ell(x_0 x_R) - \ell(x_L).
\end{align*}
Since we may assume $x_L\neq 1$, we find a simple affine root $a\in L$ with $x_L(a)\in \Phi_\af^-$, so that $(x_0 x_R)^{-1}(a) \in \Phi_\af^-$. Since $x_0\in \prescript L{}{}\widetilde W{}^R$, we have $x_0^{-1}(a)\in \Phi_\af^+$, so $r_{x_0^{-1}(a)}x_R < x_R$.

We see that we can write
\begin{align*}
x_1 = x_L x_0 x_R = \underbrace{(x_L r_a)}_{<x_L} x_0 \underbrace{(r_{x_0^{-1}(a)} x_R)}_{<x_R},
\end{align*}
finishing the induction and thus the proof.
\item The claims on the Bruhat order are implied by the claimed existences of length additive products, so it suffices to show the latter. We again focus on $\prescript L{}{}x{}^R$.

Among all elements in
\begin{align*}
\{\tilde x \in \widetilde W\mid \exists x_L \in \widetilde W_L, x_R\in \widetilde W_R:~x = x_L \tilde x x_R\text{ and }\ell(x) = \ell(x_L) + \ell(\tilde x)+\ell(x_R)\},
\end{align*}
choose an element $x_0$ of minimal length. As in (a), one shows easily that $x_0\in \prescript L{}{}\widetilde W{}^R$. By (a), we get $x_0 = \prescript L{}{}x{}^R$, so the claim follows.
\item This is \cite[Proposition~2.5.1]{Bjorner2005}.
\item
If $\prescript L{}{}x{}^R\leq \prescript L{}{}y{}^R$ and $i\in\{1,\dotsc,\ell\},j\in\{1,\dotsc,r\}$, we get $L\subseteq L_i, R\subseteq R_i$ such that
\begin{align*}
\prescript {L_i}{}{}x{}^{R_j} = 
\prescript {L_i}{}{}\left(\prescript L{}{}x{}^R\right){}^{R_j}\underset{\text{(c)}}\leq
\prescript {L_i}{}{}\left(\prescript L{}{}y{}^R\right){}^{R_j} = 
\prescript {L_i}{}{}y{}^{R_j}.
\end{align*}
It remains to show the converse.

In case $R = \emptyset$ and $r=0$, this is exactly \cite[Theorem~2.6.1]{Bjorner2005}. Similarly, the claim follows if $L=\emptyset$ and $\ell=0$. Writing $\prescript L{}{}x{}^R = \prescript L{}{}\left(x^R\right)$ etc.\, one reduces the claim to applying \cite[Theorem~2.6.1]{Bjorner2005} twice.
\qedhere
\end{enumerate}
\end{proof}
We first describe a replacement for the length functional $\ell(x,\cdot)$ that is well-behaved with passing to $\prescript L{}{}x{}^R$.
\begin{definition}
Let $L, R\subseteq \Delta_\af$ be regular. Then we define for each $x=w\varepsilon^\mu\in \widetilde W$ the \emph{coset length functional}
\begin{align*}
&\prescript L{} {}\ell{}^R(x,\cdot):\Phi\rightarrow\mathbb Z,\quad \alpha\mapsto \prescript L{}{} \ell {}^R(x,\alpha),
\\&\prescript L{}{} \ell{}^R(x,\alpha) := \langle \mu,\alpha\rangle + \chi_R(\alpha) - \chi_L(w\alpha).
\end{align*}
We refer to Definition~\ref{def:semi-affine-weight} for the definition of $\chi_L, \chi_R$.
\end{definition}
\begin{lemma}\label{lem:cosetLengthFunctional}
Let $K, L, R \subseteq \Delta_\af$ be regular subsets and let $x=w\varepsilon^\mu\in \widetilde W$.
\begin{enumerate}[(a)]
\item For $\alpha\in \Phi$, we have
\begin{align*}
\chi_K(\alpha) + \chi_K(-\alpha)= \begin{cases}1,&\alpha\in \Phi\setminus \Phi_K,\\
0,&\alpha\in \Phi_K.\end{cases}
\end{align*}
If $\alpha,\beta\in \Phi$ satisfy $\alpha+\beta\in \Phi$, then
\begin{align*}
\chi_K(\alpha) + \chi_K(\beta)-\chi_K(\alpha+\beta)\in\{0,1\}.
\end{align*}
\item $\prescript L{}{} \ell{}^R(x,\cdot)$ is a root functional, as studied in \crossRef{Section~}{sec:root-functionals}.
\end{enumerate}
\end{lemma}
\begin{proof}
\begin{enumerate}[(a)]
\item We have
\begin{align*}
\chi_K(\alpha) + \chi_K(-\alpha) = 1-\Phi_K^+(\alpha) - \Phi_K^+(-\alpha)=\begin{cases}1,&\alpha\in \Phi\setminus \Phi_K,\\0,&\alpha\in \Phi_K.\end{cases}
\end{align*}
Now suppose $\alpha+\beta\in \Phi$. Observe that the set
\begin{align*}
R := \Phi_\af^-\cup(\Phi_\af)_K\subseteq \Phi_\af
\end{align*}
is closed under addition, in the sense that for $a,b\in R$ with $a+b\in \Phi_\af$, we have $a+b\in R$.

By definition, $(\alpha,-\chi_K(\alpha)),(\beta,-\chi_K(\beta))\in R$. Thus
\begin{align*}
c := (\alpha+\beta,-\chi_K(\alpha)-\chi_K(\beta))\in R.
\end{align*}
If $c\in (\Phi_\af)_K$, then $\chi_K(\alpha+\beta) = \chi_K(\alpha)+\chi_K(\beta)$ by definition of $\chi_K(\alpha+\beta)$. Hence let us assume that $c\in \Phi_\af^-\setminus (\Phi_\af)_K$.

The condition $c\in \Phi_\af^-$ means that
\begin{align*}
-\chi_K(\alpha)-\chi_K(\beta)\leq -\Phi^+(\alpha+\beta)\leq -\chi_K(\alpha+\beta).
\end{align*}
This shows $\chi_K(\alpha)+\chi_K(\beta)-\chi_K(\alpha+\beta)\geq 0$. We want to show it lies in $\{0,1\}$, so suppose that \begin{align*}
\chi_K(\alpha)+\chi_K(\beta)-\chi_K(\alpha+\beta)\geq 2.
\end{align*} We observe that
\begin{align*}
\underbrace{(\alpha,1-\chi_K(\alpha))}_{\in\Phi_\af\setminus R} + 
\underbrace{(\beta,1-\chi_K(\beta))}_{\in\Phi_\af\setminus R} = \underbrace{(\alpha+\beta,2-\chi_K(\alpha)-\chi_K(\beta))}_{\in R}.
\end{align*}
Since also the set $\Phi_\af\setminus R$ is closed under addition, this is impossible. The contradiction shows the claim.
\item This is immediate from (a):
\begin{align*}
\prescript L{}{}\ell{}^R (x,\alpha) + \prescript L{}{}\ell{}^R(x,-\alpha)=&\langle \mu,\alpha\rangle + \langle \mu,-\alpha\rangle + \underbrace{\chi_R(\alpha) + \chi_R(-\alpha)}_{\in\{0,1\}}
\\&\qquad-\underbrace{(\chi_L(w\alpha) + \chi_L(-w\alpha))}_{\in\{0,1\}}
\\\in&\{-1,0,1\}.
\end{align*}
Now if $\alpha+\beta\in \Phi$, we get
\begin{align*}
&\prescript L{}{}\ell{}^R (x,\alpha) + \prescript L{}{}\ell{}^R (x,\beta) - \prescript L{}{}\ell{}^R(x,\alpha+\beta)
\\=&\langle \mu,\alpha\rangle + \langle \mu,\beta\rangle - \langle \mu,\alpha+\beta\rangle + \underbrace{\chi_R(\alpha) + \chi_R(\beta) - \chi_R(\alpha+\beta)}_{\in\{0,1\}} \\&\qquad- \underbrace{(\chi_L(w\alpha) + \chi_L(w\beta) - \chi_L(w\alpha+w\beta))}_{\in\{0,1\}}
\\\in&\{-1,0,1\}.\qedhere
\end{align*}
\end{enumerate}
\end{proof}
We are ready to state our main result for this subsection:
\begin{proposition}\label{prop:bruhat3}
Let $x = w\varepsilon^\mu, x'=w'\varepsilon^{\mu'}\in \widetilde W$, let $L,R\subseteq \Delta_\af$ be regular subsets and $v\in W$ be positive for $\prescript L{}{}\ell{}^R(x,\cdot)$. Moreover, fix subsets $J_1,\dotsc,J_m\subseteq \Delta$ such that $J:=J_1\cap\cdots\cap J_m$ satisfies
\begin{align*}
\forall \alpha\in \Phi_J:~\prescript L{}{}\ell{}^R(x,v\alpha)\geq 0.
\end{align*}
We have $\prescript L{}{}x{}^R\leq \prescript L{}{}(x'){}^R$ if and only if for each $i=1,\dotsc,m$, there exists some $v_i'\in W$ with
\begin{align*}
v^{-1}\mu + \prescript R{}\wt(v_i'\Rightarrow v) + \prescript L{}\wt(wv\Rightarrow w'v'_i)\leq (v_i')^{-1}\mu'\pmod{\Phi_{J_i}^\vee}.
\end{align*}
\end{proposition}
We remark that this recovers Theorem~\ref{thm:bruhat2} in case $L=R=\emptyset$.

We now start the work towards proving Proposition~\ref{prop:bruhat3}.
\begin{lemma}\label{lem:chiReflection}
Let $K\subseteq \Delta_\af$ be regular, $\alpha\in \Phi_K$ and $\beta\in \Phi$. Then
\begin{align*}
\chi_K(s_\alpha(\beta)) = \chi_K(\beta) - \langle \alpha^\vee,\beta\rangle\chi_K(\alpha).
\end{align*}
\end{lemma}
\begin{proof}
Consider the affine roots $a = (\alpha,-\chi_K(\alpha))\in (\Phi_\af)_K$ and $b = (\beta,-\chi_K(\beta))\in \Phi_\af$.

If $\beta\in \Phi_K$, then $b\in (\Phi_\af)_K$ such that $r_a(b) \in (\Phi_\af)_K$. Explicitly,
\begin{align*}
r_a(b) = \left(s_\alpha(\beta), -\chi_K(\beta) + \langle \alpha^\vee,\beta\rangle \chi_K(\alpha)\right),
\end{align*}
such that the claim follows from the definition of $\chi_K(s_\alpha(\beta))$.

Next assume that $\beta\notin \Phi_K$, such that $b \in (\Phi_\af)^-\setminus (\Phi_\af)_K$. Since $r_a$ stabilizes the set $(\Phi_\af)^-\setminus (\Phi_\af)_K$, we get $r_a(b) \in (\Phi_\af)^-\setminus (\Phi_\af)_K$. This proves (together with the above calculation) that
\begin{align*}
-\chi_K(\beta) + \langle \alpha^\vee,\beta\rangle \chi_K(\alpha) \leq -\Phi^+(s_\alpha(\beta)) = -\chi_K(s_\alpha(\beta)).
\end{align*}
If the inequality above was strict, we would get
\begin{align*}
b' := \left(s_\alpha(\beta),-\chi_K(\beta) + \langle \alpha^\vee,\beta\rangle \chi_K(\alpha)+1\right)\in \Phi_\af^-\setminus (\Phi_\af)_K
\end{align*}
with
\begin{align*}
r_a(b') = (\beta,1-\chi_K(\beta))\in \Phi_\af^+,
\end{align*}
contradiction.
\end{proof}
\begin{lemma}\label{lem:cosetLengthFunctionalProjection}
Let $x \in \widetilde W, x_L\in \widetilde W_L$ and $x_R\in \widetilde W_R$ where $L,R\subseteq \Delta_\af$ are regular subsets. Denoting the image of $x_R$ in $W$ by $\cl(x_R)$, we have the following identity for every $\alpha\in \Phi$:
\begin{align*}
\prescript L{}{}\ell{}^R(x_L x x_R,\alpha) = \prescript L{}{}\ell{}^R(x,\cl(x_R)(\alpha)).
\end{align*}
\end{lemma}
\begin{proof}
We start with two special cases:

In case $x_L = r_a$ and $x_R=1$ for some $(\beta,k):=a\in L$, we obtain
\begin{align*}
\prescript L{}{}\ell{}^R(x_L x x_R,\alpha) =& \prescript L{}{}\ell{}^R\left(s_\beta w\varepsilon^{\mu+kw^{-1}\beta^\vee},\alpha\right)
\\=&\langle \mu+kw^{-1}\beta^\vee,\alpha\rangle + \chi_R(\alpha) - \chi_L(s_\beta w\alpha)
\\=&\langle\mu,\alpha\rangle - \chi_L(\beta)\langle \beta^\vee,w\alpha\rangle + \chi_R(\alpha) - \chi_L(s_\beta w\alpha)
\\\underset{\text{L\ref{lem:chiReflection}}}=&\langle\mu,\alpha\rangle + \chi_R(\alpha) - \chi_L(w\alpha) = \prescript L{}{}\ell{}^R(x,\alpha).
\end{align*}

In case $x_L = 1$ and $x_R = r_a$ for some $(\beta,k):=a\in R$, we obtain
\begin{align*}
\prescript L{}{}\ell{}^R(x_L x x_R,\alpha) =& \prescript L{}{}\ell{}^R\left(ws_\beta \varepsilon^{s_\beta(\mu)+k\beta^\vee},\alpha\right)
\\=&\langle s_\beta(\mu)+k\beta^\vee,\alpha\rangle + \chi_R(\alpha) - \chi_L(ws_\beta\alpha)
\\=&\langle \mu,s_\beta(\alpha)\rangle - \chi_R(\beta)\langle\beta^\vee,\alpha\rangle + \chi_R(\alpha) -\chi_L(ws_\beta\alpha)
\\\underset{\text{L\ref{lem:chiReflection}}}=&\langle\mu,s_\beta(\alpha)\rangle +\chi_R(s_\beta\alpha) - \chi_L(ws_\beta\alpha)
\\=&\prescript L{}{}\ell{}^R(x,s_\beta\alpha).
\end{align*}

Now in the general case, pick reduced decompositions for $x_L\in \widetilde W_L$ and $x_R\in \widetilde W_R$ and iterate the previous arguments.
\end{proof}
\begin{definition}
By a \emph{valid tuple}, we mean a seven tuple \begin{align*}(x=w\varepsilon^\mu, x'=w'\varepsilon^{\mu'}, v, v', L, R, J)\end{align*} consisting of
\begin{itemize}
\item elements $x = w\varepsilon^\mu, x'=w'\varepsilon^{\mu'}\in \widetilde W$,
\item elements $v, v'\in W$,
\item regular subsets $L, R\subseteq \Delta_\af$ and
\item a subset $J\subseteq \Delta$,
\end{itemize}
satisfying the condition
\begin{align*}
v^{-1}\mu + \prescript R{}\wt(v'\Rightarrow v)+\prescript L{}\wt(wv\Rightarrow w'v')\leq (v')^{-1}\mu'\pmod{\Phi_J^\vee}.
\end{align*}
The tuple is called \emph{strict} if $v$ is positive for $\prescript L{}{}\ell{}^R(x,\cdot)$ and $v'$ is positive for $\prescript L{}{}\ell{}^R(x',\cdot)$.
\end{definition}
We have the following analogue of Lemma~\ref{lem:criterionAdjustment}:
\begin{lemma}\label{lem:deodharAdjustment}
Let $(x=w\varepsilon^\mu, x'=w'\varepsilon^{\mu'}, v, v', L, R, J)$ be a valid tuple.
If $v'$ is not positive for $\prescript L{}{}\ell{}^R(x',\cdot)$ and $v''$ is an adjustment in the sense of \crossRef{Definition~}{def:rootFunctionals}, then $(x, x', v, v'', L, R, J)$ is also a valid tuple.
\end{lemma}
\begin{proof}
This means that there is a root $\alpha\in \Phi^+$ such that $v'' = v's_\alpha$ and either
\begin{align*}\prescript L{}{}\ell{}^R(x',v'\alpha)<0\text{ or }\prescript L{}{}\ell{}^R(x',-v'\alpha)>0.\end{align*}
We abbreviate this condition to $\pm\prescript L{}{}\ell{}^R(x',\pm v'\alpha)<0$ and calculate
\begin{align*}
&v^{-1}\mu + \prescript R{}\wt(v''\Rightarrow v)+\prescript L{}\wt(wv\Rightarrow w'v'')
\\=&
v^{-1}\mu + \prescript R{}\wt(v's_\alpha\Rightarrow v)+\prescript L{}\wt(wv\Rightarrow w'v's_\alpha)
\\\underset{\text{L\ref{lem:semiAffineWeightProperties}}}\leq&
v^{-1}\mu + \prescript R{}\wt(v'\Rightarrow v)+\chi_R(v'\alpha)\alpha^\vee+\prescript L{}\wt(wv\Rightarrow w'v')+\chi_L(-w'v'\alpha)\alpha^\vee
\\\leq&(v')^{-1}\mu + (\chi_R(v'\alpha) + \chi_L(-w'v'\alpha))\alpha^\vee
\\=&(v'')^{-1}\mu + \left(\langle\mu,\alpha\rangle + \chi_R(v'\alpha) + \chi_L(-w'v'\alpha)\right)\alpha^\vee\pmod{\Phi_J^\vee}
\end{align*}
In case $\prescript L{}{}\ell{}^R(x',v'\alpha)<0$, we use the fact $\chi_L(-w'v'\alpha)\leq 1-\chi_L(w'v'\alpha)$ (cf.\ Lemma~\ref{lem:cosetLengthFunctional}) to show
\begin{align*}
&\langle\mu,\alpha\rangle + \chi_R(v'\alpha) + \chi_L(-w'v'\alpha)
\\\leq&
\langle\mu,\alpha\rangle + \chi_R(v'\alpha) + 1-\chi_L(w'v'\alpha)
\\=&\prescript L{}{}\ell{}^R(x',\alpha)+1\leq 0.
\end{align*}
Similarly if $\prescript L{}{}\ell{}^R(x',-v'\alpha)>0$, we get
\begin{align*}
&\langle\mu,\alpha\rangle + \chi_R(v'\alpha) + \chi_L(-w'v'\alpha)
\\\leq&
\langle\mu,\alpha\rangle + 1-\chi_R(-v'\alpha) +\chi_L(-w'v'\alpha)
\\=&1-\prescript L{}{}\ell{}^R(x',-\alpha)\leq 0.
\end{align*}
In any case, we see that
\begin{align*}
\langle\mu,\alpha\rangle + \chi_R(v'\alpha) + \chi_L(-w'v'\alpha)\leq 0,
\end{align*}
from where the desired claim is immediate.
\end{proof}
\begin{lemma}\label{lem:deodharCriterionInvariance}
Let $(x=w\varepsilon^\mu, x'=w'\varepsilon^{\mu'}, v, v', L, R, J)$ be a (strict) valid tuple.
Let moreover $x_L, x'_L\in \widetilde W_L$ and $x_R, x'_R\in \widetilde W_R$ be any elements. Then
\begin{align*}
(x_L x x_R, x'_L x' x'_R,\cl(x_R) v, \cl(x'_R)v',L,R,J)
\end{align*}
is a (strict) valid tuple as well.
\end{lemma}
\begin{proof}
Similar to the proof of Lemma~\ref{lem:cosetLengthFunctionalProjection}, it suffices to show the claim in case three of the four elements $x_L, x'_L, x_R, x_R'$ are trivial and the remaining one is a simple affine reflection.

We just explain the argument in case $x_L = r_a, x'_L = x_R = x_R'=1$ for some $a\in L$, as the remaining arguments are very similar. Write $a = (\alpha,k)$ so that $\chi_L(\alpha)=-k$. Then $x_L x = s_\alpha w \varepsilon^{\mu + kw^{-1}\alpha^\vee}$ We calculate
\begin{align*}
&v^{-1}\left(\mu + kw^{-1}\alpha^\vee\right) + \prescript R{}\wt(v'\Rightarrow v) + \prescript L{}\wt(s_\alpha w v\Rightarrow w'v')
\\\underset{\text{L\ref{lem:semiAffineWeightProperties}}}=&v^{-1}\mu + k(wv)^{-1}\alpha^\vee + \prescript R{}\wt(v'\Rightarrow v) + \chi_L(\alpha)(wv)^{-1}\alpha^\vee + \prescript L{}\wt(w v\Rightarrow w'v')
\\=~\,&v^{-1}\mu + \prescript R{}\wt(v'\Rightarrow v) + \prescript L{}\wt(w v\Rightarrow w'v').
\end{align*}
It follows that $(x_L x, x', v, v', L, R, J)$ is a valid tuple. The strictness assertion follows from Lemma~\ref{lem:cosetLengthFunctionalProjection}.
\end{proof}
Using Lemma~\ref{lem:deodharCriterionInvariance}, it will suffice to show Proposition~\ref{prop:bruhat3} only in the case $x\in \prescript L{}{}\widetilde W{}^R$ and $x'\in \prescript{-L}{}{}\widetilde W{}^{-R}$.
\begin{lemma}\label{lem:bruhat3-preparation}
Let $(x=w\varepsilon^\mu, x'=w'\varepsilon^{\mu'}, v, v', L, R, J)$ be a strict valid tuple.
\begin{enumerate}[(a)]
\item If $x\in \prescript L{}{}\widetilde W{}^R$ and $\alpha\in \Phi$ satisfies $\prescript L{}{}\ell{}^R(x,\alpha)\geq 0$, then $\ell(x,\alpha)\geq 0$.
\item If $x\in \prescript L{}{}\widetilde W{}^R$ and $\alpha\in \Phi_L^+$ satisfies $(wv)^{-1}\alpha\in \Phi^-$, then
\begin{align*}
(x, x', s_{w^{-1}\alpha}v, v', L, R, J)
\end{align*}
is a strict valid tuple as well.
\item If $x'\in \prescript {-L}{}{}\widetilde W{}^{-R}$ and $\alpha\in \Phi_R^+$ satisfies $v^{-1}\alpha\in \Phi^-$, then
\begin{align*}
(x, x', v, s_\alpha v', L, R, J)
\end{align*}
is a strict valid tuple as well.
\end{enumerate}
\end{lemma}
\begin{proof}
We write
\begin{align*}
\prescript L{}{}\ell{}^R(x,\alpha) =& \langle \mu,\alpha\rangle + \chi_R(\alpha) -\chi_L(w\alpha)
\\=& \langle \mu,\alpha\rangle + \Phi^+(\alpha)-\Phi^+_R(\alpha) -\Phi^+(w\alpha)+\Phi_L^+(w\alpha)
\\=&\ell(x,\alpha) - \Phi^+_R(\alpha) + \Phi_L^+(w\alpha)
\end{align*}
\begin{enumerate}[(a)]
\item If $w\alpha\notin \Phi_L^+$, then
\begin{align*}
\ell(x,\alpha) = \prescript L{}{}\ell{}^R(x,\alpha) + \Phi^+_R(\alpha)\geq 0.
\end{align*}
If $w\alpha\in \Phi_L^+$, then the condition $x\in \prescript L{}{}\widetilde W{}^R$ already implies $\ell(x,\alpha)\geq 0$.
\item The condition $\alpha\in \Phi_L^+$ together with $x\in  \prescript L{}{}\widetilde W{}^R$ yields $\ell(x,w^{-1}\alpha)\geq 0$. We have 
\begin{align*}
\prescript L{}{}\ell{}^R(x,-w^{-1}\alpha) = \prescript L{}{}\ell{}^R(x,v( -(wv)^{-1}\alpha))\geq 0
\end{align*} by the positivity assertion on $v$. By (a), we conclude $\ell(x,-w^{-1}\alpha)\geq 0$, so altogether we get $\ell(x,w^{-1}\alpha)=0$.

By the above computation, we get
\begin{align*}
\prescript L{}{}\ell{}^R(x,w^{-1}\alpha) =- \Phi^+_R(w^{-1}\alpha) + \Phi_L^+(\alpha)
=1-\Phi^+_R(w^{-1}\alpha).
\end{align*}
On the other hand, we have  
\begin{align*}
\prescript L{}{}\ell{}^R(x,w^{-1}\alpha) = 
\prescript L{}{}\ell{}^R(x,v (wv)^{-1}\alpha) \leq 0
\end{align*}
by the positivity assertion on $v$. Thus $
\prescript L{}{}\ell{}^R(x,w^{-1}\alpha)=0$ and $w^{-1}\alpha\in \Phi^+_R$.

Consider the elements $a = (\alpha,\Phi^+(-\alpha))\in (\Phi_\af)_L^+$ and $b = (w^{-1}\alpha,\Phi^+(-w^{-1}\alpha))\in (\Phi_\af)_R^+$. We have
\begin{align*}
x(b) =& (\alpha,\Phi^+(-w^{-1}\alpha) - \langle \mu,w^{-1}\alpha\rangle)
\\=&(\alpha,\Phi^+(-\alpha)+\ell(x,-w^{-1}\alpha)\rangle) = (\alpha,\Phi^+(-\alpha))=a.
\end{align*}
We see that $x = r_a x r_b$. Now the claim follows from Lemma~\ref{lem:deodharCriterionInvariance}.
\item The proof is analogous to (b).
\qedhere
\end{enumerate}
\end{proof}
\begin{proof}[Proof of Proposition~\ref{prop:bruhat3}]
Let us fix $L, R, J_1, \dotsc, J_m, J$ for the entire proof. To keep our notation concise, we make the following convention: We call a triple $(x, x', v)$ \emph{valid} if, for each $i=1,\dotsc,m$, there exists $v_i'\in W$ such that $(x, x', v, v_i', L, R, J_i)$ is a strict valid tuple.

First assume that $\prescript L{}{}x{}^R\leq \prescript L{}{}x'{}^R$. We want to show that $(x, x', v)$ is valid. Write $x = x_L\cdot \prescript L{}{}x{}^R \cdot x_R$ with $x_L\in \widetilde W_L, x_R\in \widetilde W_R$. It suffices to show that $\left(\prescript L{}{}x{}^R, x', \cl(x_R)^{-1}v\right)$ is valid by Lemma~\ref{lem:deodharCriterionInvariance}.

In other words, we may assume that $x\in \prescript L{}{}\widetilde W{}^R$ and $x\leq x'$ for proving that $(x, x', v)$ is valid. By Lemma~\ref{lem:bruhat1}, we find $v'\in W$ such that
\begin{align*}
v^{-1}\mu + \wt(v'\Rightarrow v) + \wt(wv\Rightarrow w'v')\leq (v')^{-1}\mu'.
\end{align*}
Now recall from Lemma~\ref{lem:semiAffineWeight} that \begin{align*}
&\prescript R{}\wt(v'\Rightarrow v)\leq \wt(v'\Rightarrow v),\\
&\prescript L{}\wt(wv\Rightarrow w'v')\leq \wt(wv\Rightarrow w'v').
\end{align*}
We conclude that $(x, x', v, v', L, R, J_i)$ is valid for all $i=1,\dotsc,m$. Up to iteratively choosing adjustments for $v'$, we may assume that the tuple is strict valid, so $(x, x', v)$ is indeed valid.

For the converse direction, let us assume that $(x, x', v)$ is valid. We have to show $\prescript L{}{}x{}^R \leq \prescript L{}{}(x'){}^R$. Again, we can use Lemma~\ref{lem:deodharCriterionInvariance} and Lemma~\ref{lem:cosetLengthFunctionalProjection} to reduce this to any other elements in $\widetilde W_L x \widetilde W_R$ resp.\ $\widetilde W_L x' \widetilde W_R$.

Thus, we may and will assume that $x\in \prescript L{}{}\widetilde W{}^R$ and $x'\in \prescript {-L}{}{}\widetilde W{}^{-R}$. We then have to show $x\leq x'$ using the fact that $(x, x', v)$ is valid for some $v\in W$.

Among all $v\in W$ such that $(x, x', v)$ is valid, choose one such that $\prescript L{}\ell(wv)$ is as small as possible. If $wv\notin \prescript L{}W$, then we find some $\alpha\in \Phi_L^+$ with $(wv)^{-1}\in\Phi^-$. By Lemma~\ref{lem:bruhat3-preparation}, also $(x, x', s_{w^{-1}\alpha}v)$ is valid and by Lemma~\ref{lem:semiAffineLength}, $\prescript L{}\ell(s_\alpha wv) < \prescript L{}\ell(wv)$. This is a contradiction to the minimality of $\prescript L{}\ell(wv)$.

We see that we always find some $v\in W$ such that $(x, x', v)$ is valid and $wv\in \prescript L{}W$.

We now prove that $x\leq x'$ using Theorem~\ref{thm:bruhat2}.

By Lemma \ref{lem:bruhat3-preparation} (a), it follows that $v\in W$ is length positive for $x$ and that $\ell(x,v\alpha)\geq 0$ for all $\alpha\in \Phi_J$. Since $\Phi_J = -\Phi_J$ and $\ell(x,-v\alpha) = - \ell(x,v\alpha)$, this is only possible if $\ell(x,v\alpha)=0$ for all $\alpha\in \Phi_J$. We conclude that $(v, J_1,\dotsc,J_m)$ is a Bruhat-deciding datum for $x$.

Now for each $i=1,\dotsc,m$, by assumption, there exists some $v_i\in W$ such that $(x, x', v, v_i', L, R, J_i)$ is a strict valid tuple. Minimizing $\prescript R{}\ell(v_i')$ as before, we may assume that $v_i'\in \prescript R{} W$ by Lemma~\ref{lem:bruhat3-preparation}.

We see that $(x, x', v, v_i', L, R, J_i)$ is a strict valid tuple with $wv\in \prescript L{}W$ and $v_i'\in \prescript R{}W$. By definition of the semi-affine weight function, we get
\begin{align*}
\prescript R{}\wt(v_i'\Rightarrow v) &= \wt(v_i'\Rightarrow v),
\\\prescript L{}\wt(wv\Rightarrow w'v'_i)& = \wt(wv\Rightarrow w'v'_i).
\end{align*}
We conclude
\begin{align*}
&v^{-1}\mu + \wt(v_i'\Rightarrow v) + \wt(wv\Rightarrow w'v_i')
\\=~\,&v^{-1}\mu + \prescript R{}\wt(v_i'\Rightarrow v) + \prescript L{}\wt(wv\Rightarrow w'v_i')
\\\underset{\text{valid}}\leq&(v')^{-1}\mu'\pmod{\Phi_{J_i}^\vee}.
\end{align*}
This is exactly the inequality we had to check in order to apply Theorem~\ref{thm:bruhat2}. So we conclude $x\leq x'$, finishing the proof.
\end{proof}
As an application, we present our most general criterion for the Bruhat order on affine Weyl groups.
\begin{definition}
Let $x \in \widetilde W$. A \emph{Deodhar datum for $x$} consists of the following:
\begin{itemize}
\item Regular subsets $L_1, \dotsc,L_\ell,R_1,\dotsc,R_r\subseteq \Delta_\af$ with $\ell,r\geq 1$ such that $L := L_1\cap\cdots\cap L_\ell$ and $R:= R_1\cap\cdots\cap R_r$ satisfy $x\in \prescript L{}{}\widetilde W{}^R$.
\item For each $i\in\{1,\dotsc,\ell\}$ and $j\in\{1,\dotsc,r\}$ an element $v_{i,j}\in W$ that is positive for $\prescript{L_i}{}{}\ell{}^{R_j}(x, \cdot)$.
\item For each $i\in\{1,\dotsc,\ell\}$ and  $j\in\{1,\dotsc,r\}$ a collection of subsets
\begin{align*}
J(i,j)_1,\dotsc,J(i,j)_{m(i,j)}\subseteq \Delta
\end{align*}
such that $m(i,j)\geq 1$ and $J(i,j) := J(i,j)_1\cap \cdots\cap J(i,j)_{m(i,j)}$ satisfies
\begin{align*}
\forall \alpha\in \Phi_{J(i,j)}:~\prescript {L_i}{}{}\ell{}^{R_j}(x, v_{i,j}\alpha)\geq 0.
\end{align*}
\end{itemize}
\end{definition}
\begin{theorem}\label{thm:bruhat4}
Let $x = w\varepsilon^\mu\in \widetilde W$ and fix a Deodhar datum
\begin{align*}
L_1,\dotsc,L_\ell, \quad R_1,\dotsc,R_r,\quad (v_{\bullet,\bullet}),\quad (J(\bullet,\bullet)_{\bullet}).
\end{align*}
Let $x' = w'\varepsilon^{\mu'}\in \widetilde W$. Then $x\leq x'$ if and only if for each $i\in\{1,\dotsc,\ell\}, j\in\{1,\dotsc,r\}$ and $k\in\{1,\dotsc,m(i,j)\}$, there exists some $v_{i,j,k}'\in W$ such that
\begin{align*}
v_{i,j}^{-1}\mu + \prescript{R_j}{}\wt(v'_{i,j,k}\Rightarrow v_{i,j}) + \prescript{L_i}{}\wt(wv_{i,j}\Rightarrow w'v'_{i,j,k})\leq (v'_{i,j,k})^{-1}\mu'\pmod{\Phi_{J(i,j)_k}^\vee}.
\end{align*}
\end{theorem}
\begin{proof}
In view of Proposition~\ref{prop:bruhat3}, the existence of the $v_{i,j,k}'$ for fixed $i,j$ means precisely
\begin{align*}
\prescript {L_i}{}{}x{}^{R_j}\leq \prescript {L_i}{}{}(x'){}^{R_j}.
\end{align*}
By Deodhar's lemma, i.e.\ Proposition~\ref{prop:deodhar}, this is equivalent to $x = \prescript L{}{}x{}^R\leq x'$.
\end{proof}
\begin{lemma}\label{lem:semiAffineWeightSupremum}
Let $w_1, w_2\in W$.
Let moreover $R_1,\dotsc,R_k\subseteq \Delta_{\af}$ be regular subsets with $k\geq 1$ and $R:=R_1\cap\cdots\cap R_k$.
Then we have the following equality in $\mathbb Z\Phi^\vee$:
\begin{align*}
\prescript R{}\wt(w_1\Rightarrow w_2) = \sup_{i=1,\dotsc,k}\prescript {R_i}{}\wt(w_1\Rightarrow w_2).
\end{align*}
\end{lemma}
\begin{proof}
Consider Proposition~\ref{prop:bruhat3} for $\mu$ and $\mu'$ sufficiently regular, with $L=\emptyset$ and $(J_1,\dotsc,J_m)=(\emptyset)$. Then by Proposition~\ref{prop:deodhar},
\begin{align*}
x^R\leq (x')^R\iff \forall i\in\{1,\dotsc,k\}:~x^{R_i}\leq (x')^{R_i}.
\end{align*}
The claim follows from Proposition~\ref{prop:bruhat3} with little effort.
\end{proof}
Together with Lemma~\ref{lem:semiAffineWeightSimplification}, this result allows us to express the weight function of the quantum Bruhat graph $\wt: W\times W\rightarrow \mathbb Z\Phi^\vee$ as a supremum of semi-affine weight functions.

As our final application of Proposition~\ref{prop:bruhat3}, we generalize Proposition~\ref{prop:admissibleSubset} to the admissible subsets considered in \cite{Rapoport2002}.
\begin{proposition}\label{prop:generalizedAdmissibleSubset}
Let $K\subseteq \Delta_\af$ be regular, $x = w\varepsilon^\mu\in \widetilde W$ and $\lambda\in \Xast$ dominant.
Then the following are equivalent:
\begin{enumerate}[(i)]
\item $x \in \widetilde W_K\Adm(\lambda)\widetilde W_K$.
\item For every $v\in W$, we have
\begin{align*}
v^{-1}\mu + \prescript K{}\wt(wv\Rightarrow v)\leq \lambda.
\end{align*}
\item There exists some $v\in W$ that is positive for $\prescript K{}{}\ell{}^K(x,\cdot)$ and satisfies \begin{align*}
v^{-1}\mu+\prescript K{}\wt(wv\Rightarrow v)\leq \lambda.
\end{align*}
\end{enumerate}
\end{proposition}
\begin{proof}
By definition, (i) means that there exists $u\in W$ such that
\begin{align*}
\prescript K{}{}{}x{}^K\leq \prescript K{}{}{}(\varepsilon^{u\lambda}){}^K.
\end{align*}
By Proposition~\ref{prop:bruhat3}, we get condition (ii) for every $v\in W$ that is positive for $\prescript K{}{}\ell{}^K(x,\cdot)$. Now a simple adjustment argument, similar to Lemma~\ref{lem:deodharAdjustment}, shows that (ii) holds for every $v\in W$.

(ii) $\implies$ (iii) is clear, as we always find a positive element for each root functional \crossRef{Corollary~}{cor:rootFunctionalAdjustments}.

(iii) $\implies$ (i): It suffices to show that $\prescript K{}{}x{}^K\leq \varepsilon^{v\lambda}$. This follows immediately from Proposition~\ref{prop:bruhat3}.
\end{proof}
% !TeX spellcheck = en_GB
\section{Demazure product}\label{chap:demazure-products}
The Demazure product $\ast$ is another operation on the extended affine Weyl group $\widetilde W$. In the context of the Iwahori-Bruhat decomposition of a reductive group, the Demazure product describes the closure of the product of two Iwahori double cosets, cf.\ \cite[Section~2.2]{He2021c}. In a more Coxeter-theoretic style, we can define the Demazure product of $\widetilde W$ as follows:
\begin{proposition}[{\cite[Lemma~1]{He2009}}]\label{prop:demazure}
Let $x_1, x_2\in \widetilde W$. Then each of the following three sets contains a unique maximum (with respect to the Bruhat order), and the maxima agree:
\begin{align*}
\{x_1x_2'\mid x_2'\leq x_2\},\quad \{x_1' x_2\mid x_1'\leq x_1\},\quad \{x_1' x_2'\mid x_1'\leq x_1,~x_2'\leq x_2\}.
\end{align*}
The common maximum is denoted $x_1\ast x_2$. If we write $x_1 \ast x_2 = x_1x_2' = x_1' x_2$, then
\begin{align*}
\ell(x_1\ast x_2) = \ell(x_1) + \ell(x_2') = \ell(x_1') + \ell(x_2).&\rightqed
\end{align*}
\end{proposition}
Demazure products have recently been studied in the context of affine Deligne-Lusztig varieties \cite{Sadhukhan2021, He2021b, He2021c}. While the Demazure product is a somewhat simple Coxeter-theoretic notion, it is connected to the question of generic Newton points of elements in $\widetilde W$. He \cite{He2021b} shows how to compute generic Newton points in terms of iterated Demazure products, a method that we will review in Section~\ref{sec:demazure-newton}. Conversely, He and Nie \cite{He2021c} use the Mili\'cevi\'c's formula for generic Newton points \cite{Milicevic2021} to show new properties of the Demazure product.

In this chapter, we prove a new description of Demazure products in $\widetilde W$, generalizing the aforementioned results of \cite{He2021c}. As applications, we obtain new results on the quantum Bruhat graph that shed some light on our previous results on the Bruhat order. Moreover, we give a new description of generic Newton points.
\subsection{Computation of Demazure products}
If one plays a bit with our Theorem~\ref{thm:bruhat2} or \cite[Proposition~3.3]{He2021c}, one will soon get an idea of how Demazure products should roughly look like. We capture the occurring formulas as follows.
\begin{situation}\label{sit:demazure}
Let $x_1 = w_1 \varepsilon^\mu_1, x_2 = w_2 \varepsilon^\mu_2\in \widetilde W$. Let $v_1, v_2\in W$ and define
\begin{align*}
x_1':=& w_1'\varepsilon^{\mu_1'} := (w_1 v_1) (w_2 v_2)^{-1}\varepsilon^{w_2 v_2 v_1^{-1}\mu_1 - w_2 v_2 \wt(v_1\Rightarrow w_2 v_2)},\\
x_2':=& w_2'\varepsilon^{\mu_2'} := v_1 v_2^{-1}\varepsilon^{\mu_2 - v_2 \wt(v_1\Rightarrow w_2 v_2)},\\
x_\ast :=& w_\ast \varepsilon^{\mu_\ast} := w_1 v_1 v_2^{-1}\varepsilon^{v_2 v_1^{-1}\mu_1 + \mu_2 - v_2\wt(v_1\Rightarrow w_2 v_2)} = x_1' x_2 = x_1 x_2'.
\end{align*}
\end{situation}
In this situation, we want to compute the Demazure product $x_1\ast x_2$, knowing that $x_1\ast x_2$ can be written as $\tilde x_1 x_2 = x_1\tilde x_2$ for some $\tilde x_1\leq x_1$ and $\tilde x_2\leq x_2$. If $x_1$ is in a shrunken Weyl chamber with $\LP(x_1) = v_1$, and $x_2$ is shrunken with $\LP(x_2) = \{v_2\}$, then $x_\ast = x_1\ast x_2$ by \cite[Proposition~3.3]{He2021c}, so $\tilde x_1 = x_1'$ and $\tilde x_2 = x_2'$.

In the general case, our goal is to find conditions on $v_1, v_2\in W$ to ensure that $x_\ast = x_1\ast x_2$.

Before examining this situation further, it will be very convenient for our proofs to see that the property
\begin{align*}
(x_1\ast x_2)^{-1} = x_2^{-1}\ast x_1^{-1}
\end{align*}
is reflected by our construction in Situation~\ref{sit:demazure}.

\begin{lemma}\label{lem:demazureDuality}
Suppose we are in Situation~\ref{sit:demazure}. Let us write $y_1 := x_2^{-1}$ and $y_2 := x_1^{-1}$. Define $v_1' :=w_2  v_2 w_0$ resp.\ $v_2' :=w_1  v_1 w_0$.

Construct $y_1', y_2', y_\ast$ associated with $(y_1, y_2, v_1', v_2')$ as in Situation~\ref{sit:demazure}. Then
\begin{align*}
y_1' = (x_2')^{-1},\quad y_2' = (x_1')^{-1},\quad y_\ast = x_\ast^{-1}.
\end{align*}
Moreover,
\begin{itemize}
\item $v_1 \in \LP(x_1)$ iff $v_2' \in \LP(y_1)$.
\item $v_2 \in \LP(x_2)$ iff $v_1' \in \LP(y_2)$.
\item $d_{\QB(W)}(v_1\Rightarrow w_2 v_2) = d_{\QB(W)}(v_1'\Rightarrow w_1^{-1} v_2')$ and\\ $\wt(v_1\Rightarrow w_2 v_2) = -w_0\wt(v_1'\Rightarrow w_1^{-1} v_2)$.
\end{itemize}
\end{lemma}
\begin{proof}
Write
\begin{align*}
y_1 = w_2^{-1}\varepsilon^{-w_2 \mu_2},\quad y_2 = w_1^{-1}\varepsilon^{-w_1 \mu_1}
\end{align*}
and compute
\begin{align*}
y_2' =& (w_2v_2  w_0)( w_1v_1 w_0)^{-1}\varepsilon^{-w_1 \mu_1 - w_1v_1  w_0\wt(w_2 v_2 w_0\Rightarrow (w_1)^{-1} w_1 v_1 w_0)} 
\\=&(w_2 v_2) (w_1 v_1)^{-1} \varepsilon^{-w_1 \mu_1 + w_1 v_1\wt(v_1\Rightarrow w_2 v_2)} = (x_1')^{-1}.
\end{align*}
A similar computation, or a repetition of this argument for $x_1 = (y_2)^{-1}, x_2 = (y_1)^{-1}$, shows that $y_1' = (x_2')^{-1}$. Then the conclusion $y_\ast = x_\ast^{-1}$ is immediate.

For the \enquote{Moreover} statements, recall that
\begin{align*}
\LP(y_1) = \LP(x_2^{-1})\underset{\text{\crossRef{Lemma~}{lem:lengthFunctionalForProducts}}}=w_2 \LP(x_2) w_0.
\end{align*}
The same holds for $y_2 = x_1^{-1}$. The final statement is due to the fact that $v_1' = w_2 v_2 w_0$ and $w_1^{-1} v_2' = v_1 w_0$ using the duality anti-automorphism of the quantum Bruhat graph, cf.\ Lemma~\ref{lem:weightIdentities}.
\end{proof}
The first step towards proving $x_1\ast x_2 = x_\ast$ is the following estimate:
\begin{lemma}\label{lem:demazureLengthEstimate}
Let $x_1, x_2\in \widetilde W$ and $v_2\in \LP(x_1\ast x_2)$. There exists $v_1\in \LP(x_1)$ such that
\begin{align*}
\ell(x_1 \ast x_2) \leq \ell(x_1) + \ell(x_2) - d(v_1\Rightarrow w_2 v_2).
\end{align*}
\end{lemma}
\begin{proof}
Write $x_1\ast x_2 = y x_2$ for some element $y = w' \varepsilon^{\mu'}\leq x_1$. Observe that $\ell(y x_2) = \ell(y) + \ell(x_2)$, so that $v_2$ must be length positive for $x_2$ and $w_2 v_2$ must be length positive for $y$.

Since $y\leq x_1$, using Lemma~\ref{lem:bruhat1}, we find a length positive element $v_1$ for $x_1$ such that
\begin{align*}
(w_2 v_2)^{-1}\mu' + \wt(v_1\Rightarrow w_2 v_2) + \wt(w'w_2 v_2\Rightarrow w_1 v_1)\leq (v_1)^{-1}\mu_1.
\end{align*}
Pairing with $2\rho$ and using Lemma~\ref{lem:weight2rho}, we compute
\begin{align*}
&\langle 2\rho, (w_2 v_2)^{-1}\mu'\rangle + \ell(v_1) - \ell(w_2 v_2) \\&\qquad+ d(v_1\Rightarrow w_2 v_2) + \ell(w'w_2 v_2) - \ell(w_1 v_1) + d(w' w_2 v_2\Rightarrow w_1 v_1)\\&\leq \langle 2\rho,(v_1)^{-1}\mu_1\rangle.
\end{align*}
Using the length positivity of $w_2 v_2$ for $y$ and $v_1$ for $x_1$ (\positiveLengthFormulaLong{}), we conclude
\begin{align*}
\ell(y) + d(v_1\Rightarrow w_2 v_2) + d(w'w_2 v_2\Rightarrow w_1v_1)\leq \ell(x_2).
\end{align*}
Thus
\begin{align*}
\ell(x_1\ast x_2) = \ell(y) + \ell(x_2) \leq \ell(x_1) + \ell(x_2) - d(v_1\Rightarrow w_2v_2) - d(w' w_2 v_2\Rightarrow w_1 v_1).
\end{align*}
We obtain the desired conclusion.
\end{proof}
We now study the Situation~\ref{sit:demazure} further.
\begin{lemma}\label{lem:analysis_x1}
Consider Situation~\ref{sit:demazure}, and assume that $v_1\in \LP(x_1)$. Then we always have the estimate
\begin{align*}
\ell(x_1') \geq \ell(x_1) - d_{\QB(W)}(v_1\Rightarrow w_2 v_2).
\end{align*}
The following are equivalent:
\begin{enumerate}[(i)]
\item Equality holds above:
\begin{align*}
\ell(x_1') = \ell(x_1) - d_{\QB(W)}(v_1\Rightarrow w_2 v_2).
\end{align*}
\item $w_2 v_2$ is length positive for $x_1'$.
\item For any positive root $\alpha$, we have
\begin{align*}
\ell(x_1,v_1\alpha) - \langle \wt(v_1\Rightarrow w_2 v_2),\alpha\rangle+ \Phi^+(w_2 v_2\alpha) - \Phi^+(v_1\alpha) \geq 0
\end{align*}
\end{enumerate}
In that case, $x_1'\leq x_1$, so that $x_\ast \leq x_1\ast x_2$.
\end{lemma}
\begin{proof}
Consider the calculation
\begin{align*}
\ell(x_1') \underset{\text{\positiveLengthFormulaShort}}\geq& \left\langle (w_2 v_2)^{-1}\left(w_2 v_2 v_1^{-1}\mu_1 - w_2 v_2 \wt(v_1\Rightarrow w_2 v_2)\right),2\rho\right\rangle - \ell(w_2 v_2) + \ell(w_1 v_1)
\\\underset{\text{L\ref{lem:weight2rho}}}=&\langle v_1^{-1}\mu,2\rho\rangle - \ell(v_1) + \ell(w_2 v_2) - d(v_1\Rightarrow w_2 v_2)- \ell(w_2 v_2) + \ell(w_1 v_1)
\\\underset{\text{\positiveLengthFormulaShort}}=&\ell(x_1) - d(v_1\Rightarrow w_2 v_2).
\end{align*}
This shows the estimate and (i) $\iff$ (ii). In order to show (ii) $\iff$ (iii), we compute
\begin{align*}
\ell(x_1', w_2 v_2\alpha)=&\langle w_2 v_2\alpha, w_2 v_2 v_1^{-1}\mu_1 - w_2 v_2\wt(v_1\Rightarrow w_2 v_2),\alpha\rangle + \Phi^+(w_2 v_2\alpha) - \wt(w_1 v_1\alpha)
\\=&\ell(x_1,v_1\alpha) - \Phi^+(v_1\alpha) - \langle \wt(v_1\Rightarrow w_2 v_2),\alpha\rangle + \Phi^+(w_2 v_2\alpha).
\end{align*}

Finally, assume that (i) -- (iii) are satisfied. We have to show $x_1'\leq x_1$. For this, we calculate
\begin{align*}
&(w_2 v_2)^{-1}\left(w_2 v_2 v_1^{-1}\mu_1 - w_2 v_2 \wt(v_1\Rightarrow w_2 v_2)\right) + \wt(v_1\Rightarrow w_2 v_2) \\&\qquad+ \wt(w_1 v_1\Rightarrow w_1 v_1) \\=&~ v_1^{-1}\mu_1.
\end{align*}
Since we assumed $w_2 v_2\in \LP(x_1')$, we conclude $x_1'\leq x_1$ by Theorem~\ref{thm:bruhat2}. Now by definition of the Demazure product, we get $x_\ast = x_1' x_2\leq x_1\ast x_2$.
\end{proof}
By the duality presented in Lemma~\ref{lem:demazureDuality}, we obtain the following:
\begin{lemma}\label{lem:analysis_x2}
Consider Situation~\ref{sit:demazure}, and assume that $v_2\in \LP(x_2)$. Then we always have the estimate
\begin{align*}
\ell(x_2') \geq \ell(x_2) - d_{\QB(W)}(v_1\Rightarrow w_2 v_2).
\end{align*}
The following are equivalent:
\begin{enumerate}[(i)]
\item Equality holds above:
\begin{align*}
\ell(x_2') = \ell(x_2) - d_{\QB(W)}(v_1\Rightarrow w_2 v_2).
\end{align*}
\item $v_2$ is length positive for $x_2'$.
\item For any positive root $\alpha$, we have
\begin{align*}
\ell(x_2, v_2\alpha) - \langle\wt(v_1\Rightarrow w_2 v_2),\alpha\rangle + \Phi^+(w_2 v_2\alpha) - \Phi^+(v_1\alpha)\geq 0.
\end{align*}
\end{enumerate}
In that case, $x_2'\leq x_2$, so that $x_\ast \leq x_1\ast x_2$.
\end{lemma}
\begin{proof}
Under Lemma~\ref{lem:demazureDuality}, this is precisely Lemma~\ref{lem:analysis_x1}.
\end{proof}
\begin{lemma}\label{lem:analysis_xast}
Suppose we are given Situation~\ref{sit:demazure}, and that $v_1\in \LP(x_1)$ and $v_2\in \LP(x_2)$. We have the estimate
\begin{align*}
\ell(x_\ast) \geq \ell(x_1) + \ell(x_2) - d(v_1\Rightarrow w_2 v_2).
\end{align*}
Equality holds if and only if $v_2\in \LP(x_\ast)$.
\end{lemma}
\begin{proof}Using again \positiveLengthFormulaLong and Lemma~\ref{lem:weight2rho}, we calculate
\begin{align*}
\ell(x_\ast)\geq&\left\langle v_2^{-1}\left( v_2 v_1^{-1}\mu_1 + \mu_2 - v_2\wt(v_1\Rightarrow w_2 v_2)\right),2\rho\right\rangle - \ell(v_2) + \ell(w_1 v_1)
\\=&\langle v_1^{-1}\mu_1,2\rho\rangle + \langle v_2^{-1}\mu_2,2\rho\rangle - d(v_1\Rightarrow w_2 v_2) - \ell(v_1) + \ell(w_2 v_2) + \ell(v_2) + \ell(w_1 v_1)
\\=&\ell(x_1) + \ell(x_2) - d(v_1\Rightarrow w_2 v_2)
\end{align*}
Both claims follow from this calculation.
\end{proof}
\begin{lemma}\label{lem:genericActionConstruction}
Let $x = w\varepsilon^\mu\in \widetilde W$ and $u\in W$. Among all $v\in \LP(x)$, there is a unique one such that $d(v\Rightarrow u)$ becomes minimal. For this particular $v$, we have
\begin{align*}
\forall \alpha\in \Phi^+:~
\ell(x,v\alpha) - \langle \wt(v\Rightarrow u),\alpha\rangle+ \Phi^+(u\alpha) - \Phi^+(v\alpha) \geq 0.
\end{align*}
\end{lemma}
\begin{proof}
Let $x_2 = t^{u\lambda}$ with $\lambda\in \Xast$ superregular and dominant.
Let $v = v_1\in \LP(x)$ such that $d(v\Rightarrow u)$ becomes minimal. Set $v_2 = u$.

Consider Situation~\ref{sit:demazure} for $x_1 = x$ and $x_2$ as above. Now the condition (iii) of Lemma~\ref{lem:analysis_x2} is satisfied by superregularity of $\lambda$. We conclude that $x_2'\leq x_2$, so that $x_\ast \leq x\ast x_2$.

Combining Lemma~\ref{lem:demazureLengthEstimate} with Lemma~\ref{lem:analysis_xast} shows
\begin{align*}
\ell(x) + \ell(x_2) - d(v\Rightarrow u) \geq \ell(x_1\ast x_2) \geq \ell(x_\ast)\geq
\ell(x) + \ell(x_2) - d(v\Rightarrow u).
\end{align*}
In particular, we get $x_1\ast x_2 = x_\ast$.

The above argument works whenever $v\in \LP(x)$ is chosen such that $d(v\Rightarrow u)$ becomes minimal. Since the value of $x_1\ast x_2$ does not depend on the choice of such an element $v$, nor does $x_\ast = x_1\ast x_2$. In particular, the classical part $\mathrm{cl}(x_\ast) = wvu^{-1}$ does not depend on $v$, hence $v$ is uniquely determined.

The formula $x_\ast = x_1\ast x_2 = x_1' x_2$ implies that $\ell(x_\ast) = \ell(x_1') + \ell(x_2)$. Using the previously computed length of $x_\ast$, we conclude $\ell(x_1') = \ell(x_1) - d(v\Rightarrow u)$. Now the estimate follows from Lemma~\ref{lem:analysis_x1}.
\end{proof}
Considering Lemma~\ref{lem:genericActionConstruction} for the inverse $x^{-1}$, we obtain the following:
\begin{lemma}\label{lem:genericActionDualConstruction}
Let $x = w\varepsilon^\mu\in \widetilde W$ and $u\in W$. Among all $v\in \LP(x)$, there is a unique one such that $d(u\Rightarrow wv)$ becomes minimal. For this particular $v$, we have
\begin{align*}
&\forall \alpha\in \Phi^+:~
\ell(x,v\alpha) - \langle\wt(u\Rightarrow wv),\alpha\rangle- \Phi^+(u\alpha) + \Phi^+(wv\alpha) \geq 0.\rightqed
\end{align*}
\end{lemma}
\begin{definition}\label{def:genericAction}
Let $x \in \widetilde W$ and $u\in W$. The uniquely determined $v\in \LP(x)$ such that $d(v\Rightarrow u)$ is minimal will be denoted by $v = \rho^\vee_x(u)$. The uniquely determined $v\in \LP(x)$ such that $d(u\Rightarrow wv)$ is minimal will be denoted by $v = \rho_x(u) = w^{-1}\rho^\vee_{x^{-1}}(uw_0)w_0$.
\end{definition}
The functions $\rho_x$ and $\rho_x^\vee$ will be studied in Section~\ref{sec:generic-action}. For now, we state our announced description of Demazure products in $\widetilde W$.
\begin{theorem}\label{thm:demazure}
Let $x_1 = w_1\varepsilon^{\mu_1}, x_2 = w_2\varepsilon^{\mu_2}\in \widetilde W$. Among all pairs $(v_1, v_2)\in \LP(x_1)\times \LP(x_2)$, pick one such that the distance $d(v_1\Rightarrow w_2 v_2)$ becomes minimal.

Construct $x_\ast$ as in Situation~\ref{sit:demazure}. Then
\begin{align*}
&x_1\ast x_2 = x_\ast = w_1 v_1 \varepsilon^{v_1^{-1}\mu_1 + v_2^{-1}\mu_2 - \wt(v_1\Rightarrow w_2 v_2)}v_2^{-1},\\& \ell(x_1\ast x_2) = \ell(x_1) + \ell(x_2) - d(v_1\Rightarrow w_2 v_2),\\&v_2\in \LP(x_1\ast x_2).
\end{align*}
\end{theorem}
\begin{proof}
We have $x_\ast \leq x_1\ast x_2$ by Lemmas \ref{lem:genericActionConstruction} and \ref{lem:analysis_x1}. By Lemma~\ref{lem:demazureLengthEstimate}, we find $(v_1', v_2')\in \LP(x_1)\times \LP(x_2)$ such that
\begin{align*}
\ell(x_1) + \ell(x_2) - d(v_1'\Rightarrow w_2 v_2') \geq \ell(x_1\ast x_2) \geq \ell(x_\ast)\geq 
\ell(x_1) + \ell(x_2) - d(v_1\Rightarrow w_2 v_2).
\end{align*}
By choice of $(v_1, v_2)$, the result follows.
\end{proof}
We note the following consequences of Theorem~\ref{thm:demazure}.
\begin{proposition}\label{prop:minDistancePairs}
Let $x_1 = w_1 \varepsilon^{\mu_1}, x_2 = w_2 \varepsilon^{\mu_2}\in \widetilde W$. Write
\begin{align*}
M = M(x_1, x_2) := \{&(v_1, v_2)\in \LP(x_1)\times \LP(x_2)\mid \\&\forall (v_1', v_2')\in \LP(x_1)\times \LP(x_2):~d(v_1\Rightarrow w_2 v_2)\leq d(v_1'\Rightarrow w_2 v_2')\}
\end{align*}
for the set of all pairs $(v_1, v_2)$ such that the theorem's condition is satisfied.
\begin{enumerate}[(a)]
\item The following two functions on $M$ are both constant:
\begin{align*}
&\varphi_1 : M\rightarrow W,\quad (v_1, v_2)\mapsto v_1 v_2^{-1},
\\&\varphi_2 : M\rightarrow \mathbb Z\Phi^\vee,\quad (v_1, v_2)\mapsto v_2 \wt(v_1\Rightarrow w_2 v_2).
\end{align*}
\item The following is a well-defined bijective map:
\begin{align*}
M\rightarrow \LP(x_1\ast x_2),\quad (v_1, v_2)\mapsto v_2.
\end{align*}
\end{enumerate}
\end{proposition}
\begin{proof}
\begin{enumerate}[(a)]
\item From the Theorem, we get that the function
\begin{align*}
M\rightarrow\widetilde W,\quad (v_1, v_2)\mapsto &w_1 v_1 v_2^{-1}\varepsilon^{v_2 v_1^{-1}\mu_1+\mu_2 - v_2\wt(v_1\Rightarrow w_2 v_2)}
\\=&w_1 \varphi_1(v_1, v_2) \varepsilon^{\varphi_1(v_1, v_2)^{-1}\mu_1 + \mu_2 - \varphi_2(v_1, v_2)}
\end{align*}
is constant with image $\{x_1\ast x_2\}$. This proves that $\varphi_1$ and $\varphi_2$ are constant. 
\item Injectivity follows from (a). Well-definedness follows from the theorem. For surjectivity, let $v_2\in \LP(x_1\ast x_2)$. Then certainly $v_2\in \LP(x_2)$. By Lemma~\ref{lem:demazureLengthEstimate}, we find $v_1\in W$ such that $\ell(x_1\ast x_2)\leq \ell(x_1) + \ell(x_2) - d(v_1\Rightarrow w_2 v_2)$. By the theorem, we find $(v_1', v_2')\in M$ with $\ell(x_1\ast x_2) = \ell(x_1) + \ell(x_2) - d(v_1'\Rightarrow w_2 v_2')$, such that $d(v_1\Rightarrow w_2 v_2) \leq d(v_1'\Rightarrow w_2 v_2')$. It follows that $(v_1, v_2)\in M$, finishing the proof of surjectivity.\qedhere
\end{enumerate}
\end{proof}
\begin{remark}
In case $\ell(x_1x_2) =\ell(x_1) + \ell(x_2)$, we get $x_1x_2= x_1\ast x_2$. In this case, we recover \crossRef{Lemma~}{lem:lengthAdditivity}.
\end{remark}

\subsection{Generic action}\label{sec:generic-action}
Studying the Demazure product where one of the factors is superregular induces actions of $(\widetilde W, \ast)$ on $W$, that we denoted by $\rho_x$ resp.\ $\rho^\vee_x$ in Definition~\ref{def:genericAction}. In this section, we study these actions and the consequences for the quantum Bruhat graph.
\begin{lemma}\label{lem:genericActionChaining}
Let $x_1 = w_1\varepsilon^{\mu_1}, x_2 = w_2\varepsilon^{\mu_2}\in \widetilde W$. Then
\begin{align*}
\rho_{x_1\ast x_2} = \rho_{x_2}\circ \rho_{x_1}.
\end{align*}
\end{lemma}
\begin{proof}
Note that if $z\in \widetilde W$ is in a shrunken Weyl chamber with $\LP(z) = \{u\}$ and $x\in \widetilde W$, then by Proposition~\ref{prop:minDistancePairs},
\begin{align*}
\LP(z\ast x) = \{\rho_x(u)\}.
\end{align*}
Hence we have
\begin{align*}
\{\rho_{x_2}(\rho_{x_1}(u))\} = \LP\left((z\ast x_1)\ast x_2\right) = \LP\left(z\ast(x_1\ast x_2)\right) = \{\rho_{x_1\ast x_2}(u)\}.
\end{align*}
This shows the desired claim.
\end{proof}
\begin{remark}
\begin{enumerate}[(a)]
\item
There is a dual, albeit more complicated statement for the dual generic action $\rho^\vee$.
\item
If $x = \omega r_{a_1}\cdots r_{a_n}$ is a reduced decomposition with simple affine roots $a_1,\dotsc,a_n\in \Delta_\af$ and $\omega\in \Omega$ of length zero, then
\begin{align*}
\rho_x &= \rho_{\omega\ast r_{a_1}\ast\cdots\ast r_{a_n}} = \rho_{r_{a_n}}\circ\cdots\circ \rho_{r_{a_1}}\circ \rho_\omega.
\end{align*}
The map $\rho_\omega$ is simply given by $\rho_\omega(v) = \cl(\omega)v$, as $\LP(\omega) = W$.
We now describe the $\rho_{r_{a_i}}$ as follows:

For a simple affine root $(\alpha,k)\in \Delta_\af$, we have
\begin{align*}
\ell(r_{(\alpha,k)},\beta) = \begin{cases}1,&\beta =\alpha,\\
-1,&\beta=-\alpha,\\
0,&\beta\neq\pm\alpha.\end{cases}
\end{align*}
Thus
\begin{align*}
\LP(r_{(\alpha,k)}) = \{v\in W\mid v^{-1}\alpha\in \Phi^+\}.
\end{align*}
Let $v\in W$. If $v^{-1}\alpha\in \Phi^-$, then $s_\alpha v\in \LP(r_{(\alpha,k)})$ with $d(v\Rightarrow s_\alpha(s_\alpha v))=0$. Hence $\rho_{r_{(\alpha,k)}}(v) = s_\alpha v$.

If $v^{-1}\alpha\in \Phi^+$, then $v\in \LP(r_{(\alpha,k)})$ with $d(v\Rightarrow s_\alpha v)=1$ by Lemma~\ref{lem:qbgDiamond}. Since there exists no $u\in \LP(r_{(\alpha,k)})$ with $d(v\Rightarrow s_\alpha u)=0$, a distance of $1$ is already minimal. We see that $\rho_{r_{(\alpha,k)}}(v) = v$. Summarizing:
\begin{align*}
\rho_{r_{(\alpha,k)}}(v) = \begin{cases}v,&v^{-1}\alpha\in \Phi^+,\\
s_\alpha v,&v^{-1}\alpha\in \Phi^-.\end{cases}
\end{align*}
This gives an alternative method to compute $\rho_x$. One easily obtains a dual method to compute $\rho_x^\vee$ in a similar fashion.
\end{enumerate}
\end{remark}
\begin{lemma}\label{lem:LPshortestPath}
Let $x \in \widetilde W$ and $v, v'\in \LP(x)$ be two length positive elements. There exists a shortest path $p$ from $v$ to $v'$ in the quantum Bruhat graph such that each vertex in $p$ lies in $\LP(x)$.
\end{lemma}
\begin{proof}
Let us first study the case $v' = 1$.

We do induction on $\ell(v)$. If $\ell(v) = 0$, the statement is clear.

Otherwise, there exists a quantum edge $v\rightarrow vs_\alpha$ for some quantum root $\alpha\in \Phi^+$ such that $d(v\Rightarrow v') = d(vs_\alpha\Rightarrow v')+1$ (Lemma~\ref{lem:quantumEdgesSuffice}). In this case, it suffices to show that $vs_\alpha\in \LP(x)$.

The quantum edge condition means that $\ell(vs_\alpha) = \ell(v) - \ell(s_\alpha)$. In other words, every positive root $\beta\in \Phi^+$ with $s_\alpha(\beta)\in \Phi^-$ satisfies $v(\beta)\in \Phi^-$.

Let $\beta\in \Phi^+$, we want to show that $\ell(x,vs_\alpha(\beta))\geq 0$. This follows from length positivity of $v$ if $s_\alpha(\beta)\in \Phi^+$. So let us assume that $s_\alpha(\beta)\in \Phi^-$. Then $vs_\alpha(\beta)\in \Phi^+$, applying the above observation to $-s_\alpha(\beta)$. Hence $\ell(x,vs_\alpha(\beta))\geq 0$, as $1\in \LP(x)$. This finishes the induction, so the claim is established whenever $v'=1$.

For the general case, we do induction on $\ell(v')$. If $v'=1$, we have proved the claim, so let us assume that $\ell(v')>0$. Then we find a simple root $\alpha\in \Delta$ with $s_\alpha v' < v'$. In particular, $(v')^{-1}\alpha\in \Phi^-$ so that $\ell(x,\alpha)\leq 0$. Consider the element $x' := xs_\alpha \gtrdot x$. We observe that for any $u\in W$ and $\beta\in \Phi$, 
\begin{align*}
\ell(x',s_\alpha u\beta) = \ell(x,u\beta) + \ell(s_\alpha, -u\beta) = \begin{cases}\ell(x,u\beta),&u\beta \neq \pm \alpha,\\
-\ell(x,\alpha)+1>0,&u\beta = -\alpha,\\
\ell(x,\alpha)-1<0,&u\beta = \alpha.\end{cases}
\end{align*}
It follows that
\begin{align*}
\LP(x') = \{s_\alpha u \mid u\in \LP(x)\text{ and }u^{-1}\alpha \in \Phi^-\}.
\end{align*}
In particular, $s_\alpha v'\in \LP(x')$. Now suppose that $v^{-1}\alpha \in \Phi^-$. Then also $s_\alpha v\in \LP(x')$. We may apply the inductive assumption to get a path $p'$ from $s_\alpha v$ to $s_\alpha v'$ in $\LP(x')$. Multiplying each vertex by $s_\alpha$ on the left, we obtain the desired path $p$ in $\LP(x)$.

Finally assume that $v^{-1}\alpha \in \Phi^+$. Then $s_\alpha v\in \LP(x)$ by Corollary~\ref{cor:simpleCoverLPFix}.

By Lemma~\ref{lem:qbgDiamond}, $v\rightarrow s_\alpha v$ is an edge in $\QB(W)$ and \begin{align*}
d_{\QB(W)}(v\Rightarrow v') = d_{\QB(W)}(v\Rightarrow s_\alpha  v') = d_{\QB(W)}(s_\alpha v\Rightarrow v')+1.
\end{align*} We get a path from $s_\alpha v$ to $v'$ in $\LP(x)$ by repeating the above argument, then concatenate it with $v\rightarrow s_\alpha v$.

This finishes the induction and the proof.
\end{proof}
\begin{corollary}\label{cor:lengthPositiveComparison}
Let $x = w\varepsilon^\mu\in \widetilde W$ and $v, v'\in \LP(x)$. Then
\begin{align*}
v^{-1}\mu - (v')^{-1}\mu - \wt(v\Rightarrow v') + \wt(wv\Rightarrow wv')=0.
\end{align*}
In particular, $d(v\Rightarrow v') = d(wv\Rightarrow wv')$.
\end{corollary}
\begin{proof}
Let
\begin{align*}
p : v = v_1\xrightarrow{\alpha_1} v_2\xrightarrow{\alpha_2} \cdots\xrightarrow{\alpha_{n-1}} v_n = v'
\end{align*}
be a path in $\LP(x)$ of weight $\wt(v\Rightarrow v')$. Now for $i=1,\dotsc,n-1$, observe that both $v_i$ and $v_i s_{\alpha_i}$ are in $\LP(x)$. Thus $\ell(x, v_i \alpha_i)=0$. We conclude that
\begin{align*}
&(v_i)^{-1}\mu - (v_{i+1})^{-1}\mu - \wt(v_i \Rightarrow v_{i+1}) + \wt(wv_i\Rightarrow wv_{i+1})
\\=&\langle v_i \alpha_i, \mu\rangle \alpha_i^\vee - \Phi^+(-v_i \alpha_i)\alpha_i^\vee + \wt(wv_i\Rightarrow wv_i s_{\alpha_i})
\\\leq& \langle v_i\alpha_i, \mu\rangle\alpha_i^\vee -  \Phi^+(-v_i \alpha_i)\alpha_i^\vee + \Phi^+(w v_i\alpha_i)\alpha_i^\vee
\\=&\ell(x,v_i\alpha_i)\alpha_i^\vee = 0.
\end{align*}
Summing these estimates for $i=1,\dotsc,n-1$, we conclude
\begin{align*}
v^{-1}\mu - (v')^{-1}\mu - \wt(v\Rightarrow v') + \wt(wv\Rightarrow w'v')\leq 0.
\end{align*}
Considering the same argument for $x^{-1}, wvw_0, wv'w_0$, we get the other inequality.

The \enquote{in particular} part follows from inspecting the argument given. Alternatively, pair the identity just proved with $2\rho$, then apply Lemma~\ref{lem:weight2rho} and \positiveLengthFormulaLong{}.
\end{proof}
\begin{remark}
The corollary can be shown directly by evaluating the Demazure product
\begin{align*}
\varepsilon^{w v' 2\rho} \ast x \ast \varepsilon^{v 2\rho}
\end{align*}
in two different ways, using the associativity property of Demazure products.
\end{remark}
\begin{proposition}\label{prop:LPTiltedBruhat}
Let $x = w\varepsilon^\mu\in \widetilde W$, $v\in \LP(x)$ and $u\in W$. Then
\begin{align*}
d(u\Rightarrow wv) = d(u\Rightarrow w\rho_x(u)) + d(w\rho_x(u)\Rightarrow wv).
\end{align*}
\end{proposition}
\begin{proof}
Let $\lambda$ be superregular and $y := \varepsilon^{u\lambda}$. Define the element
\begin{align*}
z := y \ast x = u \rho_x(u)^{-1} \varepsilon^{\rho_x(u) \lambda + \mu - \rho_x(u) \wt(u\Rightarrow w\rho_x(u))}.
\end{align*}
Then $z$ is superregular with $\LP(z) = \{\rho_x(u)\}$. Consider the element
\begin{align*}
\tilde y' := u(wv)^{-1}\varepsilon^{wv\lambda - wv \wt(u\Rightarrow wv)}.
\end{align*}
This is superregular with $\LP(\tilde y') = \{wv\}$. Note that Theorem~\ref{thm:bruhat2} implies $\tilde y'\leq y$, as
\begin{align*}
(wv)^{-1}(wv \lambda - wv\wt(u\Rightarrow wv)) + \wt(u\Rightarrow wv) + \wt(u\Rightarrow u) = \lambda.
\end{align*}
Thus $\tilde z\leq z$, where
\begin{align*}
\tilde z = \tilde y x = u v^{-1} \varepsilon^{v \lambda + \mu - v \wt(u\Rightarrow wv)}.
\end{align*}
Note that $\tilde z$ is superregular with $\LP(\tilde z) = \{v\}$. In light of Theorem~\ref{thm:bruhat2}, the inequality $\tilde z\leq z$ means
\begin{align*}
v^{-1}(v \lambda + \mu - v \wt(u\Rightarrow wv)) +& \wt(\rho_x(u)\Rightarrow v) + \wt(u\Rightarrow u)\\&\leq \rho_x(u)^{-1}(\rho_x(u) \lambda + \mu - \rho_x(u) \wt(u\Rightarrow w\rho_x(u))).
\end{align*}
Rewriting this, we get
\begin{align*}
v^{-1}\mu - \wt(u\Rightarrow wv) + \wt(\rho_x(u)\Rightarrow v) \leq \rho_x(u)^{-1}\mu - \wt(u\Rightarrow w\rho_x(u)).
\end{align*}
Corollary~\ref{cor:lengthPositiveComparison} yields the equation
\begin{align*}
v^{-1}\mu -\rho_x(u)^{-1}\mu +\wt(\rho_x(u)\Rightarrow v) = \wt(w\rho_x(u) \Rightarrow wv).
\end{align*}
We conclude
\begin{align*}
\wt(u\Rightarrow wv) \geq \wt(u\Rightarrow w\rho_x(u)) + \wt(w\rho_x(u)\Rightarrow wv).
\end{align*}
This implies the desired claim.
\end{proof}
By the duality from Lemma~\ref{lem:demazureDuality}, we obtain the following.
\begin{corollary}\label{cor:LPTiltedBruhatDual}
Let $x = w\varepsilon^\mu\in \widetilde W$, $v\in \LP(x)$ and $u\in W$. Then
\begin{align*}
&d(v\Rightarrow u) = d(v\Rightarrow \rho^\vee_x(u)) + d(\rho^\vee_x(u)\Rightarrow u).\rightqed
\end{align*}
\end{corollary}
\begin{remark}
In the language of \cite[Section~6]{Brenti1998}, this means that the set $w\LP(x)$ contains a unique minimal element with respect to the tilted Bruhat order $\preceq_u$. Since $w \LP(x) = \LP(x^{-1})w_0$, it follows that the set $\LP(x)$ contains a unique maximal element with respect to $\preceq_u$. If $x = \varepsilon^\mu$ is a pure translation element, this recovers \cite[Theorem~7.1]{Lenart2015}.

The converse statements are generally false, i.e.\ $\LP(x)$ will in general not contain tilted Bruhat minima, and $w\LP(x)$ will not contain maxima. For a concrete example, choose $x$ to be a simple affine reflection of type $A_2$.

The set $\LP(x)$ satisfies a number of interesting structural properties with respect to the quantum Bruhat graph, namely containing shortest paths for any pair of elements (Lemma \ref{lem:LPshortestPath}) and the existence of tilted Bruhat maxima. One may ask the question which subsets of $W$ occur as the set $\LP(x)$ for some $x\in \widetilde W$.
\end{remark}
\begin{corollary}\label{cor:genericActionMaximality}
Let $x =w\varepsilon^\mu \in \widetilde W$ and $u_1, u_2\in W$. Then the function
\begin{align*}
\varphi: W\rightarrow \Xast,~ v\mapsto v^{-1}\mu - \wt(u_1\Rightarrow wv) - \wt(v\Rightarrow u_2)
\end{align*}
has a global maximum at $\rho_x(u_1)$, and another global maximum at $\rho^\vee_x(u_2)$.
\end{corollary}
\begin{proof}
If $v\in W$ is not length positive for $x$, and $vs_\alpha$ is an adjustment, it is easy to see that $\varphi(v)\leq\varphi(vs_\alpha)$. So we may focus on $\varphi\vert_{\LP(x)}$.

Let $v\in \LP(x)$ and $v' = \rho_x(u_1)$, so that
\begin{align*}
\varphi(v) =~\,& v^{-1}\mu - \wt(u_1\Rightarrow wv) - \wt(v\Rightarrow u_2) \\=~\,& v^{-1}\mu - \wt(u_1\Rightarrow w v') - \wt(wv'\Rightarrow wv) - \wt(v\Rightarrow u_2)
\\\underset{\text{C\ref{cor:lengthPositiveComparison}}}=&(v')^{-1}\mu - \wt(v'\Rightarrow v) - \wt(u_1\Rightarrow w v') - \wt(v\Rightarrow u_2)
\\\leq~\,&(v')^{-1}\mu - \wt(u_1\Rightarrow wv') - \wt(v'\Rightarrow u_2) = \varphi(v').
\end{align*}
This shows the first maximality claim. The second one follows from the duality of Lemma~\ref{lem:demazureDuality}.
\end{proof}
\begin{remark}\label{rem:bruhat5}Let $x_1 = w_1 \varepsilon^{\mu_1}, x_2 = w_2 \varepsilon^{\mu_2}\in \widetilde W$ and $v_1\in \LP(x_1)$. Theorem~\ref{thm:bruhat2} states that $x_1\leq x_2$ in the Bruhat order if and only if there is some $v_2\in W$ with
\begin{align*}
v_1^{-1}\mu_1 + \wt(v_2\Rightarrow v_1) + \wt(w_1v_1\Rightarrow w_2v_2) \leq v_2^{-1}\mu_2.
\end{align*}
By the above corollary, it is equivalent to require this inequality for $v_2 = \rho_{x_2}(w_1 v_1)$. One can alternatively require it for $v_2 = \rho_{x_2}^\vee(v_1)$.
\end{remark}
\begin{lemma}\label{lem:minimalityConditionReformulation}
Let $x_1 = w_1\varepsilon^{\mu_1}, x_2 = w_2 \varepsilon^{\mu_2}\in \widetilde W$ and $v_1 \in \LP(x_1), v_2\in \LP(x_2)$. The following are equivalent:
\begin{enumerate}[(i)]
\item The distance $d(v_1\Rightarrow w_2 v_2)$ is minimal for all pairs in $\LP(x_1)\times \LP(x_2)$, i.e.\ $(v_1, v_2)\in M(x_1, x_2)$.
\item $v_1 = \rho^\vee_{x_1}(w_2 v_2)$ and $v_2 = \rho_{x_2}(v_1)$.
\end{enumerate}
\end{lemma}
\begin{proof}
(i) $\Rightarrow$ (ii): Certainly, $v_1$ minimizes the function $d(\cdot\Rightarrow w_2 v_2)$ on $\LP(x_1)$, showing the first claim. The second claim is analogous.

(ii) $\Rightarrow$ (i): Consider Situation~\ref{sit:demazure}. By Lemmas \ref{lem:analysis_x1} and \ref{lem:genericActionConstruction}, we conclude that $w_2 v_2$ must be length positive for $x_1'$. It follows that $x_\ast \leq x_1\ast x_2$ and \begin{align*}
\ell(x_\ast) = \ell(x_1') + \ell(x_2) = \ell(x_1) + \ell(x_2) - d(v_1\Rightarrow w_2 v_2).
\end{align*}By Lemma~\ref{lem:analysis_xast}, $v_2$ is length positive for $x_\ast$.
Write $x_1\ast x_2$ as $\tilde w \varepsilon^{\tilde \mu}$.
Using Lemma~\ref{lem:bruhat1} with Lemma~\ref{lem:criterionAdjustment}, the condition $x_\ast \leq x_1\ast x_2$ yields some $v_2'\in \LP(x_1 \ast x_2)$ with
\begin{align*}
v_1^{-1}\mu_1 + v_2^{-1}\mu_2 - \wt(v_1\Rightarrow w_2 v_2) + \wt(v_2'\Rightarrow v_2) + \wt(w_1 v_1\Rightarrow \tilde w v_2')\leq (v_2')^{-1}\tilde\mu.
\end{align*}
By Proposition~\ref{prop:minDistancePairs}, we find $v_1'$ such that $(v_1', v_2')\in M(x_1, x_2)$. By Theorem~\ref{thm:demazure}, we can express $x_1\ast x_2$ in terms of $(v_1', v_2')$. Then the above inequality becomes
\begin{align*}
&v_1^{-1}\mu_1 + v_2^{-1}\mu_2 - \wt(v_1\Rightarrow w_2 v_2) + \wt(v_2'\Rightarrow v_2) + \wt(w_1 v_1\Rightarrow w_1 v_1')\\&\leq (v_1')^{-1}\mu_1 + (v_2')^{-1}\mu_2 - \wt(v_1'\Rightarrow w_2 v_2').
\end{align*}
Since $v_1, v_1'\in \LP(x_1)$ and $v_2, v_2'\in \LP(x_2)$, we can apply Corollary~\ref{cor:lengthPositiveComparison} twice to obtain
\begin{align*}
\wt(v_1\Rightarrow v_1') + \wt(w_2 v_2'\Rightarrow w_2 v_2) - \wt(v_1\Rightarrow w_2 v_2)\leq -\wt(v_1'\Rightarrow w_2 v_2').
\end{align*}
Rewriting, we get
\begin{align*}
\wt(v_1\Rightarrow v_1') +\wt(v_1'\Rightarrow w_2 v_2') + \wt(w_2 v_2'\Rightarrow w_2 v_2) \leq \wt(v_1\Rightarrow w_2 v_2).
\end{align*}
In other words, there is a shortest path from $v_1$ to $w_2 v_2$ that passes through $v_1'$ and $w_2 v_2'$. By condition (ii), this is only possible if $v_1 = v_1'$ and $v_2 = v_2'$, showing (i).
\end{proof}
\begin{corollary}\label{cor:minimalPairFactorization}
Consider Situation~\ref{sit:demazure} with $v_1 \in \LP(x_1), v_2\in \LP(x_2)$. There exists $(v_1', v_2')\in M(x_1, x_2)$ such that \begin{align*}d(v_1\Rightarrow w_2 v_2) = d(v_1\Rightarrow v_1') + d(v_1'\Rightarrow w_2 v_2') + d(w_2v_2'\Rightarrow w_2 v_2).\end{align*}
\end{corollary}
\begin{proof}
For convenience, we define a set of \emph{admissible pairs} by
\begin{align*}
&A := \{(v_1', v_2')\in \LP(x_1)\times \LP(x_2)\mid \\&\qquad d(v_1\Rightarrow w_2 v_2) = d(v_1\Rightarrow v_1') + d(v_1'\Rightarrow w_2 v_2') + d(w_2v_2'\Rightarrow w_2 v_2)\}.
\end{align*}
Then $(v_1, v_2)\in A$, so that $A$ is non-empty. Choose $(v'_1, v_2')\in A$ such that $d(v_1'\Rightarrow w_2 v_2')$ becomes minimal among all pairs in $A$. We claim that $(v_1', v_2')\in M(x_1, x_2)$. For this, we use Lemma~\ref{lem:minimalityConditionReformulation}. It remains to show that $v_1' = \rho^\vee_{x_1}(w_2 v_2')$ and $v_2' = \rho_{x_2}(v_1)$. By Proposition~\ref{prop:LPTiltedBruhat} and Corollary~\ref{cor:LPTiltedBruhatDual}, we obtain
\begin{align*}
&d(v_1'\Rightarrow w_2 v_2') = d(v_1'\Rightarrow \rho^\vee_{x_1}(w_2 v_2')) + d(\rho^\vee_{x_1}(w_2 v_2')\Rightarrow w_2 v_2'),
\\&d(v_1'\Rightarrow w_2 v_2') = d(v_1'\Rightarrow w_2 \rho_{x_2}(v_1)) + d(w_2 \rho_{x_2}(v_1)\Rightarrow w_2 v_2').
\end{align*}
It follows that $(\rho^\vee_{x_1}(w_2 v_2'), v_2') \in A$ and $(v_1', \rho_{x_2}(v_1')) \in A$. By choice of $(v_1', v_2')$ and the above computation, we get that $v_1' = \rho^\vee_{x_1}(w_2 v_2')$ and $v_2' = \rho_{x_2}(v_1')$. This finishes the proof.
\end{proof}
\begin{corollary}\label{cor:lengthPositiveDemazureProduct}
For $x_1, x_2\in \widetilde W$, we have $\LP(x_1\ast x_2)= \rho_{x_2}(\LP(x_1)) = \rho_{x_1}^\vee(w_2 \LP(x_2))$, where $w_2\in W$ is the classical part of $x_2$.
\end{corollary}
\begin{proof}
We only show $\LP(x_1\ast x_2) = \rho_{x_2}(\LP(x_1))$, the other claim is completely dual.

If $v_2\in \LP(x_1\ast x_2)$, we find $v_1\in \LP(x_1)$ such that $(v_1, v_2)\in M(x_1, x_2)$. By Lemma~\ref{lem:minimalityConditionReformulation}, $v_2 = \rho_{x_2}(v_1)\in \rho_{x_2}(\LP(x_1))$.

Now let $v_2\in \rho_{x_2}(\LP(x_1))$ and write $v_2 = \rho_{x_2}(v_1)$ for some $\widetilde{v_1}\in \LP(x_1)$. By Corollary~\ref{cor:minimalPairFactorization}, we find $(v_1', v_2')\in M(x_1, x_2)$ such that
\begin{align*}d(v_1\Rightarrow w_2 v_2) =& d(v_1\Rightarrow w_2 v_2') + d(w_2v_2'\Rightarrow w_2 v_2).\end{align*}
Since $v_2 = \rho_{x_2}(v_1)$, we use Proposition~\ref{prop:LPTiltedBruhat} to obtain
\begin{align*}d(v_1\Rightarrow w_2 v_2') =& d(v_1\Rightarrow w_2 v_2) + d(w_2v_2\Rightarrow w_2 v_2').\end{align*}
This is only possible if $v_2 = v_2'$. Since $v_2'\in \LP(x_1\ast x_2)$ by Proposition~\ref{prop:minDistancePairs}, we obtain the desired claim $v_2\in \LP(x_1\ast x_2)$.
\end{proof}
\ifthesis\else
\subsection{Generic $\sigma$-conjugacy class}\label{sec:demazure-newton}
To conclude the paper, we apply our results to the notion of generic $\sigma$-conjugacy classes. For this, we have to assume that our affine Weyl group actually comes from a quasi-split reductive group $G$ over a non-archimedian local field $F$, as described in \crossRef{Section~}{sec:notation}. This means that $W$ is the finite Weyl group of $G$, and $X_\ast$ are the $\Gal(\overline{\breve F}/\breve F)$-coinvariants of the cocharacter group of a maximal torus. Denote by $B(G)$ the set of $\sigma$-conjugacy classes in $G(\breve F)$. For $x\in \widetilde W$, we write $[x]\in B(G)$ for the $\sigma$-conjugacy classes associated with any representative of $x$ in $G(\breve F)$, and $[b_x]$ for the generic $\sigma$-conjugacy class of the Iwahori double coset indexed by $x$.

The Frobenius action on $W$ and $\widetilde W$ will be denoted $\presig{(\cdot)}$, so the Frobenius image of $x$ is $\presig x$.

Throughout this section, we fix an element $x=w\varepsilon^\mu\in \widetilde W$. Following He \cite{He2021b}, we consider twisted Demazure powers of $x$.
\begin{definition}
Let $n\geq 1$. We define the $n$-th $\sigma$-twisted Demazure power of $x$ as
\begin{align*}
x^{\ast,\sigma,n} := x\ast (\presig x)\ast\cdots \ast\left(\prescript{\sigma^{n-1}}{}x\right)\in \widetilde W.
\end{align*}
\end{definition}

For $n\geq 2$, let us write
\begin{align*}
x_n := \prescript{\sigma^{1-n}}{}{}\left(\left(x^{\ast,\sigma,n-1}\right)^{-1}x^{\ast,\sigma,n}\right),
\end{align*}
such that
\begin{align*}
x^{\ast,\sigma,n} = x^{\ast,\sigma,n-1}\ast \left(\prescript{\sigma^{n-1}}{}x\right) = x^{\ast,\sigma,n-1}\cdot \left(\prescript{\sigma^{n-1}}{}x_n\right).
\end{align*}
We can calculate $x_n$ in terms of $x$ and $\prescript{\sigma^{1-n}}{}\LP(x^{\ast,\sigma,n-1})$ using Theorem~\ref{thm:demazure}. By Corollary~\ref{cor:lengthPositiveDemazureProduct}, we have
\begin{align*}
\LP(x^{\ast,\sigma,n}) = \rho_{\prescript{\sigma^{n-1}}{}x}\left(\LP(x^{\ast,\sigma,n-1})\right)
=\cdots=\rho_{\prescript{\sigma^{n-1}}{}x}\circ\cdots\circ \rho_{\presig x}\left(\LP(x)\right).
\end{align*}
Observe that by definition of the generic action $\rho_x$, we may write
\begin{align*}
\rho_{\prescript{\sigma^n}{}x}(\prescript{\sigma^n}{}(u)) = \prescript{\sigma^n}{}(\rho_x(u)).
\end{align*}
Let us define the map $\rho_{x,\sigma} := \rho_x\circ \prescript{\sigma^{-1}}{}(\cdot):W\rightarrow W$ by
\begin{align*}
\rho_{x,\sigma}(u) := \rho_x(\prescript{\sigma^{-1}}{}(u)).
\end{align*}
Then
\begin{align*}
\LP(x^{\ast,\sigma,n}) =&\rho_{\prescript{\sigma^{n-1}}{}x}\circ\cdots\circ \rho_{\presig x}\left(\LP(x)\right).
\\=&\left(\prescript{\sigma^{n-1}}{}(\cdot)\circ \rho_x\circ \prescript{\sigma^{1-n}}{}(\cdot)\right)\circ\cdots\circ\left(\prescript{\sigma^{1}}{}(\cdot)\circ \rho_x\circ \prescript{\sigma^{-1}}{}(\cdot)\right)(\LP(x))
\\=&\prescript{\sigma^{n-1}}{}(\cdot)\circ \rho_{x,\sigma}\circ\cdots\circ \rho_{x,\sigma}(\LP(x))
\\=&\prescript{\sigma^{n-1}}{}{}\left(\rho_{x,\sigma}^{n-1}(\LP(x))\right).
\end{align*}
\begin{lemma}\label{lem:xInfConstruction}
\begin{enumerate}[(a)]
\item There exists an integer $N>1$ such that for each $n\geq N$, \begin{align*}
x_N = x_n\text{ and }\rho_{x,\sigma}^N(\LP(x)) = \rho_{x,\sigma}^n(\LP(x)).
\end{align*}Denote the eventual values by $x_\infty:= x_N$ resp.\ $\rho_{x,\sigma}^\infty(\LP(x)):= \rho_{x,\sigma}^N(\LP(x))$.
\item We have
\begin{align*}
\rho_{x,\sigma}^\infty(\LP(x)) =&\{v\in \LP(x)\mid\exists n\geq 1:~v = \rho_{x,\sigma}^n(v)\}.
\\\lim_{n\to\infty}\frac{\ell(x^{\ast,\sigma,n})}n = &\ell(x_\infty).
\end{align*}
\item The element $x_\infty$ is fundamental. For each $v\in \rho_{x,\sigma}^\infty(\LP(x))$, it can be written as
\begin{align*}
x_\infty = (\prescript{\sigma^{-1}}{}v)\rho_{x,\sigma}(v)^{-1}\varepsilon^{\mu - \rho_{x,\sigma}(v)\wt\left(\prescript{\sigma^{-1}}{}v\Rightarrow w\rho_{x,\sigma}(v)\right)}.
\end{align*}
\end{enumerate}
\end{lemma}
\begin{proof}
\begin{enumerate}[(a)]
\item Observe that $\rho^n_{x,\sigma}$ induces an endomorphism $\LP(x)\rightarrow \LP(x)$. We obtain a weakly decreasing sequence of subsets of $W$
\begin{align*}
\LP(x)\supseteq \rho_{x,\sigma}(\LP(x))\supseteq\rho_{x,\sigma}^2(\LP(x))\supseteq\cdots.
\end{align*}Since $W$ is finite, this sequence must stabilize eventually.

Because $x_n$ only depends on the values of $\rho_{x,\sigma}^{n-1}(\LP(x))$ and $x$, the result follows.
\item
Both claims follow immediately from (a).
\item Let $N$ be as in (a), and let $n\geq 1$. Then
\begin{align*}
x^{\ast,\sigma,N+n} = x^{\ast,\sigma,N}\cdot \prescript{\sigma^N}{}x_\infty\cdots \prescript{\sigma^{N+n-1}}{}x_\infty
\end{align*}
is a length additive product. In particular,
\begin{align*}
\ell(x_\infty\cdots\prescript{\sigma^{n-1}}{}x_\infty) = n\ell(x_\infty).
\end{align*}
By \cite[Theorem~1.3]{Nie2015} or \crossRef{Proposition~}{prop:fundamental}, $x_\infty$ is fundamental.

Next let $v\in \rho^\infty_{x,\sigma}(\LP(x))$. Then also $\rho_{x,\sigma}(v)\in \rho^\infty_{x,\sigma}(\LP(x))$, and we get
\begin{align*}
\prescript{\sigma^{N}}{}\rho_{x,\sigma}(v)\in \LP(x^{\ast,\sigma,N+1}) = \LP(x^{\ast,\sigma,N}\ast \prescript{\sigma^N}{}x) = \LP(x^{\ast,\sigma,N}\cdot\prescript{\sigma^N}{}(x_\infty)).
\end{align*}
In view of Proposition~\ref{prop:minDistancePairs}, we find a uniquely determined element $\prescript{\sigma^{N}}{}v'\in \LP(x^{\ast,\sigma,N})$ such that
\begin{align*}
(\prescript{\sigma^{N}}{}v', \prescript{\sigma^{N}}{}\rho_{x,\sigma}(v))\in M(x^{\ast,\sigma,N}, \prescript{\sigma^N}{}x).
\end{align*}
Then by Theorem~\ref{thm:demazure},
\begin{align*}
x_\infty = v'\rho_{x,\sigma}(v)^{-1}\varepsilon^{\mu - \rho_{x,\sigma}(v)\wt(v'\Rightarrow w\rho_{x,\sigma}(v))}.
\end{align*}
Note that $\presig v'\in\prescript{\sigma^{1-N}}{}\LP(x^{\ast,\sigma,N}) = \rho_{x,\sigma}^\infty(\LP(x))$. The minimality condition on the tuple
$(\prescript{\sigma^{N}}{}v', \prescript{\sigma^{N}}{}\rho_{x,\sigma}(v))$ moreover implies that $\rho_x(v') = \rho_{x,\sigma}(\presig v') = \rho_{x,\sigma}(v)$ (Lemma~\ref{lem:minimalityConditionReformulation}).

The map $\rho_{x,\sigma} : \rho_{x,\sigma}^\infty(\LP(x))\rightarrow \rho_{x,\sigma}^\infty(\LP(x))$ is a surjective, and the set $\rho_{x,\sigma}^\infty(\LP(x))$ is finite. It follows that the restriction of $\rho_{x,\sigma}$ to  $\rho_{x,\sigma}^\infty(\LP(x))$ is bijective. Recall that $v$ and $\presig v'$ are two elements of $\rho_{x,\sigma}^\infty(\LP(x))$ whose images under $\rho_{x,\sigma}$ coincide. Thus $v =\presig v'$, finishing the proof.\qedhere\end{enumerate}\end{proof}
\begin{theorem}\label{thm:gnpViaDemazure}
\begin{enumerate}[(a)]
\item
The $\sigma$-conjugacy class $[x_\infty]\in B(G)$ is the generic $\sigma$-conjugacy class of $x$.
\item For any $v\in \rho_{x,\sigma}^\infty(\LP(x))$, we have $\ell(x_\infty) = \ell(x)-d(v\Rightarrow\presig(w\rho_{x,\sigma}(v)))$.
\item Fix $v\in \rho_{x,\sigma}^\infty(\LP(x))$ and define $J=\supp_\sigma(\rho_{x,\sigma}(v)^{-1}v)$, so $J\subseteq \Delta$ consists of all $\sigma$-orbits of simple roots whose corresponding simple reflections occur in some reduced decomposition of $\rho_{x,\sigma}(v)^{-1}v\in W$.

We can express the generic Newton point of $x$ as
\begin{align*}
\nu_x = \pi_J\left(v^{-1}\mu - \wt(v\Rightarrow \presig(w v))\right).
\end{align*}
Here, $\pi_J$ denotes the projection function as defined in \cite[Definition~3.2]{Chai2000}.
\end{enumerate}
\end{theorem}
\begin{proof}
\begin{enumerate}[(a)]
\item
By a result of Viehmann \cite[Corollary~5.6]{Viehmann2014}, we can express the generic $\sigma$-conjugacy class of $x$ as
\begin{align*}
[b_x] = \max\{[y]\mid y\leq x\} = \max\{[y]\mid y\leq x\text{ and }y\text{ is fundamental}\}.
\end{align*}
In particular, $[b_x]\geq [x_\infty]$. For the converse inequality, pick some $y\leq x$ fundamental with $[b_x] = [y]\in B(G)$.

By definition of the Demazure product, we get
\begin{align*}
x^{\ast,\sigma,n} = x \ast \left(\presig x\right)\cdots\ast \left(\prescript{\sigma^{n-1}}{}x\right) \geq y \left(\presig y\right)\cdots \left(\prescript{\sigma^{n-1}}{}y\right).
\end{align*}
Thus, using the fact that $y$ and $x_\infty$ are fundamental, we get
\begin{align*}
\langle \nu(x_\infty),2\rho\rangle =& \ell(x_\infty) = \lim_{n\to \infty} \frac{\ell(x^{\ast,\sigma,n})}n \\\geq& \lim_{n\to \infty} \frac{\ell(y \presig y\cdots \prescript{\sigma^{n-1}}{}y)}n = \lim_{n\to \infty} \ell(y) = \langle \nu(y),2\rho\rangle = \langle \nu(b_x),2\rho\rangle.
\end{align*}
This estimate shows that $[x_\infty] = [b_x]$.

\item This follows from the explicit description of $x_\infty$ in Lemma \ref{lem:xInfConstruction} together with Lemma \ref{lem:positiveLengthFormula}
and the simple observation $\rho_{x,\sigma}(v)\in \LP(x_\infty)$.
\item Let us write $x_\infty = w_\infty\varepsilon^{\mu_\infty}$. The generic Newton point of $x$ is the Newton point of $x_\infty$, which we express using \crossRef{Lemma~}{lem:newtonPointsAsAverages}.

Let $N\geq 1$ such that the action of $(\sigma\circ w_\infty)$ on $X_\ast$ becomes trivial. We want to show for each $v\in \rho_{x,\sigma}^\infty(\LP(x))$ that
\begin{align*}
v^{-1}\sum_{k=1}^N(\sigma\circ w_\infty)^k \mu_\infty\in X_\ast\otimes\mathbb Q
\end{align*}
is dominant.

Note each $v\in \rho_{x,\sigma}^\infty(\LP(x))$ may be written as $v = \rho_{x,\sigma}(u)$ for some $u\in \rho_{x,\sigma}^\infty(\LP(x))$. By Lemma~\ref{lem:xInfConstruction}, it follows that $w_\infty = (\prescript{\sigma^{-1}}{}u)v^{-1}$. Thus $u = \presig(w_\infty v)\in \rho_{x,\sigma}^\infty(\LP(x))$. This shows $\presig(w_\infty v)\in \rho_{x,\sigma}^\infty(\LP(x))$ for each $v\in \rho_{x,\sigma}^\infty(\LP(x))$. It follows for each $\alpha\in \Phi^+$ that
\begin{align*}
&\left\langle v^{-1}\sum_{k=1}^N(\sigma\circ w_\infty)^k \mu_\infty,\alpha\right\rangle =\sum_{k=1}^N \langle \mu_\infty,(\sigma\circ w_\infty)^k v\alpha\rangle
\\=&\sum_{k=1}^N \left(\langle \mu_\infty,(\sigma\circ w_\infty)^k v\alpha\rangle + \Phi^+((\sigma\circ w_\infty)^k v\alpha) - \Phi^+((\sigma\circ w_\infty)^{k+1} v\alpha)\right)
\\=&\sum_{k=1}^N \ell(x_\infty,(\sigma\circ w_\infty)^k v\alpha)\geq 0.
\end{align*}

This shows the above dominance claim. As $v\in \rho_{x,\sigma}^\infty(\LP(x))$ was arbitrary, the same claim holds for $\rho_{x,\sigma}(v)$. With \begin{align*}J:=\supp_\sigma(\rho_{x,\sigma}(v)^{-1}\presig(w_\infty \rho_{x,\sigma}(v))) = \supp_\sigma(\rho_{x,\sigma}(v)^{-1} v),\end{align*}  \crossRef{Lemma~}{lem:newtonPointsAsAverages} proves that
\begin{align*}
\nu(x_\infty) =~& \pi_J(\rho_{x,\sigma}(v)^{-1}\mu_\infty) \underset{\text{L\ref{lem:xInfConstruction}}}=\pi_J(\rho_{x,\sigma}(v)^{-1}\mu - \wt(\prescript{\sigma^{-1}}{} v\Rightarrow w\rho_{x,\sigma}(v)))
\\\underset{\text{Def.~}\pi_J}=&\pi_J(\rho_{x,\sigma}(v)^{-1}\mu - \wt(v\Rightarrow\presig(w\rho_{x,\sigma}(v)))).
\end{align*}
Now observe that
\begin{align*}
\rho_{x,\sigma}(v)^{-1}\mu\equiv \,&v^{-1}\mu\pmod{\mathbb Q\Phi_J^\vee},
\\
\wt(v\Rightarrow\presig(w\rho_{x,\sigma}(v)))\equiv\,& \wt(v\Rightarrow\presig(wv))\pmod{\mathbb Q\Phi_J^\vee}.\qedhere
\end{align*}
\end{enumerate}
\end{proof}
Part (a) of the above Theorem readily implies \cite[Theorem~0.1]{He2021b}. Our previous result \crossRef{Corollary~}{cor:genericGKPMinDistance} expresses the generic Newton point $\nu_x$ as a formula similar to part (c) of the above Theorem, but the allowed elements $v\in \LP(x)$ in the cited result are usually different ones. If $x$ is in a shrunken Weyl chamber, this formula for the generic Newton point coincides with \cite[Proposition~3.1]{He2021c}.
\fi
% !TeX spellcheck = en_GB
\addcontentsline{toc}{section}{Bibliography}
\printbibliography
\end{document}